\documentclass[12pt]{amsart}
\usepackage{amssymb}
\usepackage{amsmath}
\usepackage{amsthm}
\usepackage{times}
\usepackage{graphicx}
\usepackage{changepage}
\usepackage{caption}
\usepackage{geometry}
\usepackage[curve,all]{xy}
 \xyoption{curve}
 \xyoption{line}
\xyoption{arc}
\xyoption{color}
\usepackage{array}
\usepackage[square, numbers, comma, sort&compress]{natbib}
\usepackage{color}
\usepackage{amsthm}
\usepackage{enumitem}
\newtheorem{theorem}{Theorem}[section]
\newtheorem{lemma}[theorem]{Lemma}
\newtheorem{corollary}[theorem]{Corollary}
\newtheorem{proposition}[theorem]{Proposition}

\newcommand{\thistheoremname}{}
\newtheorem{genericthm}[theorem]{\thistheoremname}

\newtheorem*{genericthm*}{\thistheoremname}
\newenvironment{namedthm*}[1]
  {\renewcommand{\thistheoremname}{#1}%
   \begin{genericthm*}}
  {\end{genericthm*}}

\theoremstyle{definition}
\newtheorem{definition}[theorem]{Definition}

\theoremstyle{remark}
\newtheorem{remark}[theorem]{Remark}
\newtheorem*{remark*}{Remark}

\newtheorem{notation}[theorem]{Notation}
\newtheorem{convention}[theorem]{Convention}

\numberwithin{equation}{section}

\newtheorem*{theorem*}{\bf{Theorem}}
\newtheorem*{corollary*}{\bf{Corollary}}

\newcommand{\calE}{\mathcal{E}}

\newcommand{\bbA}{\mathbb{A}}
\newcommand{\bbF}{\mathbb{F}}

\newcommand{\bbN}{\mathbb{N}}
\newcommand{\bbP}{\mathbb{P}}

\newcommand{\bbQ}{\mathbb{Q}}
\newcommand{\bbR}{\mathbb{R}}

\newcommand{\bbZ}{\mathbb{Z}}

\def\Aut{{\text{Aut}}}

\def\Pic{{\text{Pic}}}

\def\Num{{\rm{Num}}}

\def\deg{{\text{deg}}}

\def\Num{{\text{Num}}}

\def\I{{\text{I}}}
\def\II{{\text{II}}}
\def\III{{\text{III}}}
\def\IV{{\text{IV}}}
\def\V{{\text{V}}}
\def\VI{{\text{VI}}}
\def\VII{{\text{VII}}}
\def\VIII{{\text{VIII}}}
\def\Pic{{\text{Pic}}}
\def\Char{{\text{char}}}

\def\Fix{{\text{Fix}}}

\def\MW{{\text{MW}}}

\begin{document}
\title [Enriques surfaces with finite automorphism group] {Enriques surfaces with finite automorphism group in positive characteristic}

\author{Gebhard Martin}
\address{Technische Universit\"at M\"unchen, 
Zentrum Mathematik-M11, 
Boltzmannstr. 3, 85748 Garching bei M\"unchen, Germany }
\email{martin@ma.tum.de}

\begin{abstract}
We classify Enriques surfaces with smooth K3 cover and finite automorphism group in arbitrary positive characteristic. The classification is the same as over the complex numbers except that some types are missing in small characteristics. Moreover, we give a complete description of the moduli of these surfaces. Finally, we realize all types of Enriques surfaces with finite automorphism group over the prime fields $\bbF_p$ and $\bbQ$ whenever they exist.
\end{abstract}

\maketitle

{\bf Contents}
\begin{itemize}
\item[\S \ref{sec1}]
Introduction \hfill \pageref{sec1}

\item[\S \ref{prelim}]
Preliminaries \hfill \pageref{prelim}

\item[\S \ref{secI}] 
Enriques surfaces of type $\I$ \hfill \pageref{secI}

\item[\S \ref{secII}] 
Enriques surfaces of type $\II$ \hfill \pageref{secII}

\item[\S \ref{secIII}] 
Enriques surfaces of type $\III$ \hfill \pageref{secIII}

\item[\S \ref{secIV}] 
Enriques surfaces of type $\IV$ \hfill \pageref{secIV}

\item[\S \ref{secV}] 
Enriques surfaces of type $\V$ \hfill \pageref{secV}

\item[\S \ref{secVI}] 
Enriques surfaces of type $\VI$ \hfill \pageref{secVI}

\item[\S \ref{secVII}] 
Enriques surfaces of type $\VII$ \hfill \pageref{secVII}

\item[\S \ref{graphs}]
The classification-theorem \hfill \pageref{graphs}

\item[\S \ref{arithmetic}]
Arithmetic of Enriques surfaces with finite automorphism group \hfill \pageref{arithmetic}

\item[\S \ref{semisection}]
Semi-symplectic automorphisms \hfill \pageref{semisection}
\end{itemize}

\section*{Convention}
Unless mentioned otherwise, we will work over an algebraically closed field $k$ of arbitrary characteristic. 
By Enriques surface we will mean Enriques surface with a smooth K3 cover throughout this paper. This means that we will not be dealing with classical and supersingular Enriques surfaces in characteristic $2$. 

\section{Introduction}\label{sec1}
Classically known as the first examples of non-rational surfaces with $q = p_g = 0$, Enriques surfaces are one of the building blocks of minimal, smooth and projective surfaces of Kodaira dimension $0$. Thus, in the Enriques-Kodaira classification of complex surfaces (see for example \cite{BarthHulek}), they appear next to Abelian, bielliptic, and K3 surfaces to which they are closely related. In positive characteristics, the classification of Bombieri and Mumford \cite{Mumford}, \cite{BombieriMumford2} and \cite{BombieriMumford3} shows that the close relation to K3 surfaces persists unless the characteristic of the base field is $2$. However, in characteristic $2$, three distinct types of Enriques surfaces appear. These types are distinguished by the torsion component of the identity of the Picard scheme $\Pic^{\tau}$, which is one of $\{\bbZ/2\bbZ,\mu_2,\alpha_2\}$, and they are called classical, singular and supersingular, respectively. Among these, only the singular Enriques surfaces admit a smooth canonical K3 cover, whereas the canonical cover $\tilde{X}$ of the other types is only "K3-like", in the sense that $\tilde{X}$ is integral Gorenstein and $\omega_{\tilde{X}} \cong \mathcal{O}_{\tilde{X}}$. However, $\tilde{X}$ might even be non-normal.

Using the period map for complex Enriques surfaces \cite{Horikawa1}, \cite{Horikawa2}, one can construct "a" \cite{GritsenkoHulek} moduli space of unpolarized Enriques surfaces over the complex numbers, which is $10$-dimensional, quasi-affine \cite{Borcherds} and rational \cite{KondoRational}. There is a codimension-one subvariety parametrizing Enriques surfaces containing a $(-2)$-curve, i.e. an irreducible curve with self-intersection $(-2)$, and a codimension-one subvariety of the boundary of the period domain parametrizing Coble surfaces, i.e. smooth rational surfaces $X$ with $|-K_X| = \emptyset$ and $|-2K_X| \neq \emptyset$. Both of these codimension-one subvarieties are rational \cite{DolgachevKondo}. Thus, one expects that a $1$-dimensional family of Enriques surfaces degenerates to a Coble surface at some point. We will also see this kind of behaviour for our examples in positive characteristic.

In positive and mixed characteristic, a similar picture has been established by C. Liedtke in \cite{Liedtke} and T. Ekedahl, J. Hyland and N. Shepherd-Barron in \cite{EkedahlHylShep}: The moduli space of Cossec-Verra polarized Enriques surfaces is a quasi-separated Artin stack of finite type over $\rm{Spec}$ $\bbZ$, which is irreducible, unirational, $10$-dimensional, smooth in odd characteristics and consists of two connected components with these properties in characteristic $2$. These two connected components parametrize singular and classical Enriques surfaces, respectively. Their $9$-dimensional intersection parametrizes supersingular Enriques surfaces. The stack of unpolarized Enriques surfaces is very badly behaved (see \cite[Remark 5.3]{Liedtke}), because the automorphism group of a generic Enriques surface $X$ is an infinite and, unless $X$ is supersingular or an exceptional \cite{EkedahlShep} and classical Enriques surface in characteristic $2$, discrete group.
 
The automorphism group of a general complex Enriques surfaces was computed by W. Barth and C. Peters \cite{BarthPeters}, independently also by V. V. Nikulin \cite{Nikulin2}, and is equal to the $2$-congruence subgroup of the group of positive-cone-preserving automorphisms of the $E_{10}$ lattice. However, an Enriques surface may acquire additional $(-2)$-curves under specializations, causing the automorphism group to become smaller. Therefore, it is a natural question whether this group can degenerate to a finite group.

In 1984, I. Dolgachev \cite{Dolgachev} found an example (type $\I$) of an Enriques surface with finite automorphism group and later it was discovered that G. Fano \citep{Fano} had also found an example (type $\VII$) as early as in 1910, although the automorphism group is not $\mathfrak{S}_3$ as Fano claimed, but $\mathfrak{S}_5$ (see \cite[p.191]{Kondo}). The  full classification of Enriques surfaces with finite automorphism group over the complex numbers was then carried out by Nikulin \cite{Nikulin} in terms of their root invariants and by S. Kond\=o \cite{Kondo} using elliptic fibrations. 
There are seven types $\I,\hdots,\VII$ of such Enriques surfaces, distinguished by their dual graphs of $(-2)$-curves, the first two of which form a $1$-dimensional family and the others are unique \cite{Kondo}.

The key observation for Nikulin's approach to the classification is the fact that for a complex Enriques surface $X$ the subgroup $W_X \subseteq \rm{O}(\Num(X))$ generated by reflections along classes of $(-2)$-curves has finite index if and only if $\Aut(X)$ is finite. However, while in any characteristic $W_X$ being of finite index in $\rm{O}(\Num(X))$ implies that the automorphism group $\Aut(X)$ is finite \cite[Main Theorem]{Dolgachev}, the converse uses the Global Torelli Theorem proven by E. Horikawa \cite{Horikawa1}, \cite{Horikawa2}, which is not available in positive characteristic. For this reason, we will not pursue Nikulin's approach. Nevertheless, it will follow from our explicit classification that $\Aut(X)$ being finite implies that $W_X \subseteq \rm{O}(\Num(X))$ has finite index.

Kond\=o's approach is based on the observation -- due to Dolgachev \cite[\S 4]{Dolgachev} -- that the Mordell-Weil group of the Jacobian of every elliptic fibration of an Enriques surface $X$ acts on $X$, hence it has to be finite if we want $X$ to have finite automorphism group. Using this approach, we will obtain the classification of Enriques surfaces with finite automorphism group and smooth K3 cover in positive characteristic. Recall that the K3 cover of an Enriques surface $X$ is smooth if and only if $\Char(k) \neq 2$ or $X$ is a singular Enriques surface, i.e. $\Pic^{\tau}(X) \cong \mu_2$.

\begin{namedthm*}{Main Theorem}[Classification]\label{classification}
Let $X$ be an Enriques surface with smooth K3 cover over an algebraically closed field $k$.
\begin{enumerate}
\item $X$ has finite automorphism group if and only if the dual graph of all $(-2)$-curves on $X$ is one of the seven dual graphs in Table \ref{main}.
\item The automorphism groups, the characteristics in which they exist, and the moduli of Enriques surfaces of each of the seven types are as in Table \ref{main}.
\end{enumerate}


\begin{table}[!htb]
\centering
\begin{tabular}{|>{\centering\arraybackslash}m{1cm}|>{\centering\arraybackslash}m{6cm}|>{\centering\arraybackslash}m{2cm}|>{\centering\arraybackslash}m{1.5cm}|>{\centering\arraybackslash}m{1.5cm}|>{\centering\arraybackslash}m{2cm}|}
\hline
\rm{Type} & \rm{Dual Graph of $(-2)$-curves} & \rm{Aut} & \rm{$\Aut_{nt}$} &\rm{$\Char(k)$} & \rm{Moduli}\\
\hline \hline
\rm{I} & \vspace{1mm} \includegraphics[width=35mm]{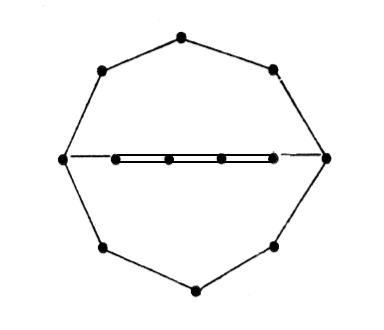} & $D_4 $& $\bbZ/2\bbZ$ & any & $\bbA^1 - \{0,-256\}$ \\ \hline
\rm{II} & \vspace{1mm} \includegraphics[width=35mm]{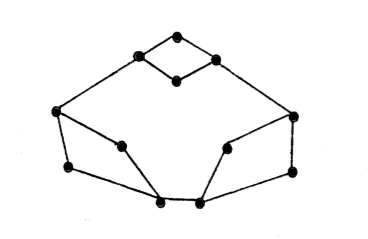} & $\mathfrak{S}_4 $& $\{1\}$ & any & $\bbA^1 - \{0,-64\}$ \\  \hline
\rm{III} & \vspace{1mm} \includegraphics[width=55mm]{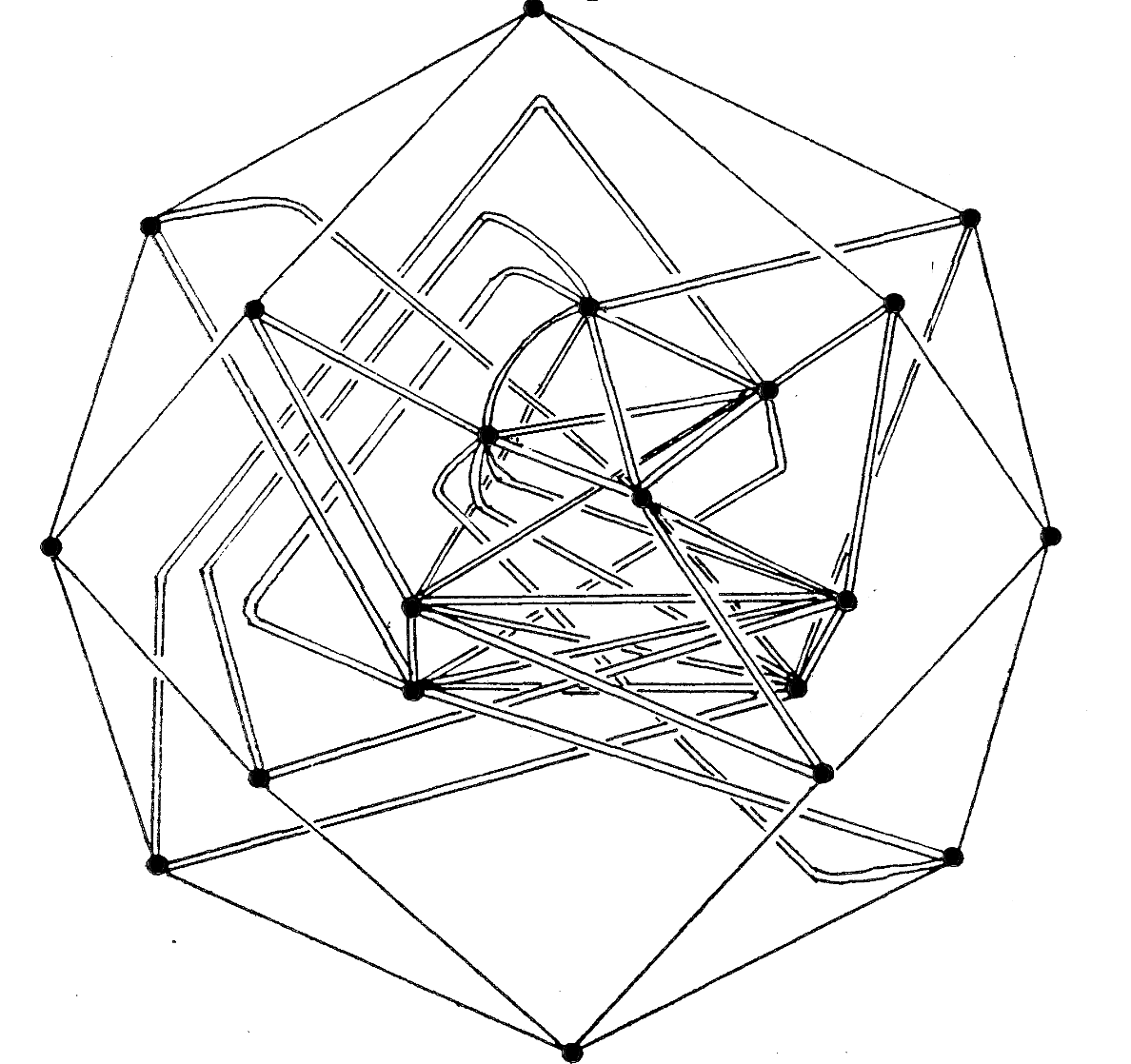} & $ (\bbZ/4\bbZ \times (\bbZ/2\bbZ)^2) \rtimes D_4 $ & $\bbZ/2\bbZ$ & $\neq 2$ & unique \\  \hline
\rm{IV} & \vspace{1mm} \includegraphics[width=55mm]{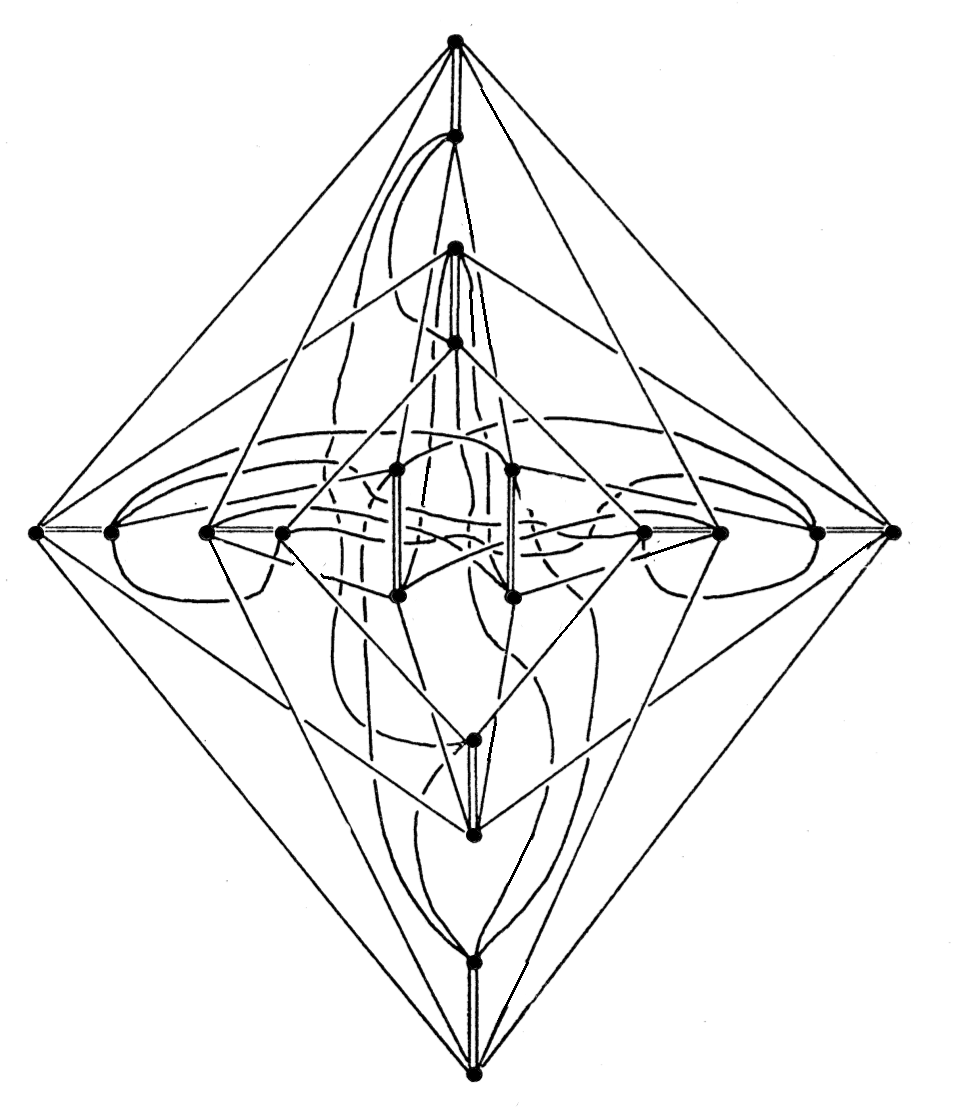} &  $(\bbZ/2\bbZ)^4 \rtimes (\bbZ/5\bbZ \rtimes \bbZ/4\bbZ)$ & $\{1\}$ & $\neq 2$ & unique \\  \hline
\end{tabular}
\end{table}

\begin{table}[!htbp]
\begin{tabular}{|>{\centering\arraybackslash}m{1cm}|>{\centering\arraybackslash}m{6cm}|>{\centering\arraybackslash}m{2cm}|>{\centering\arraybackslash}m{1.5cm}|>{\centering\arraybackslash}m{1.5cm}|>{\centering\arraybackslash}m{2cm}|} \hline
\rm{V} & \vspace{1mm} \includegraphics[width=60mm]{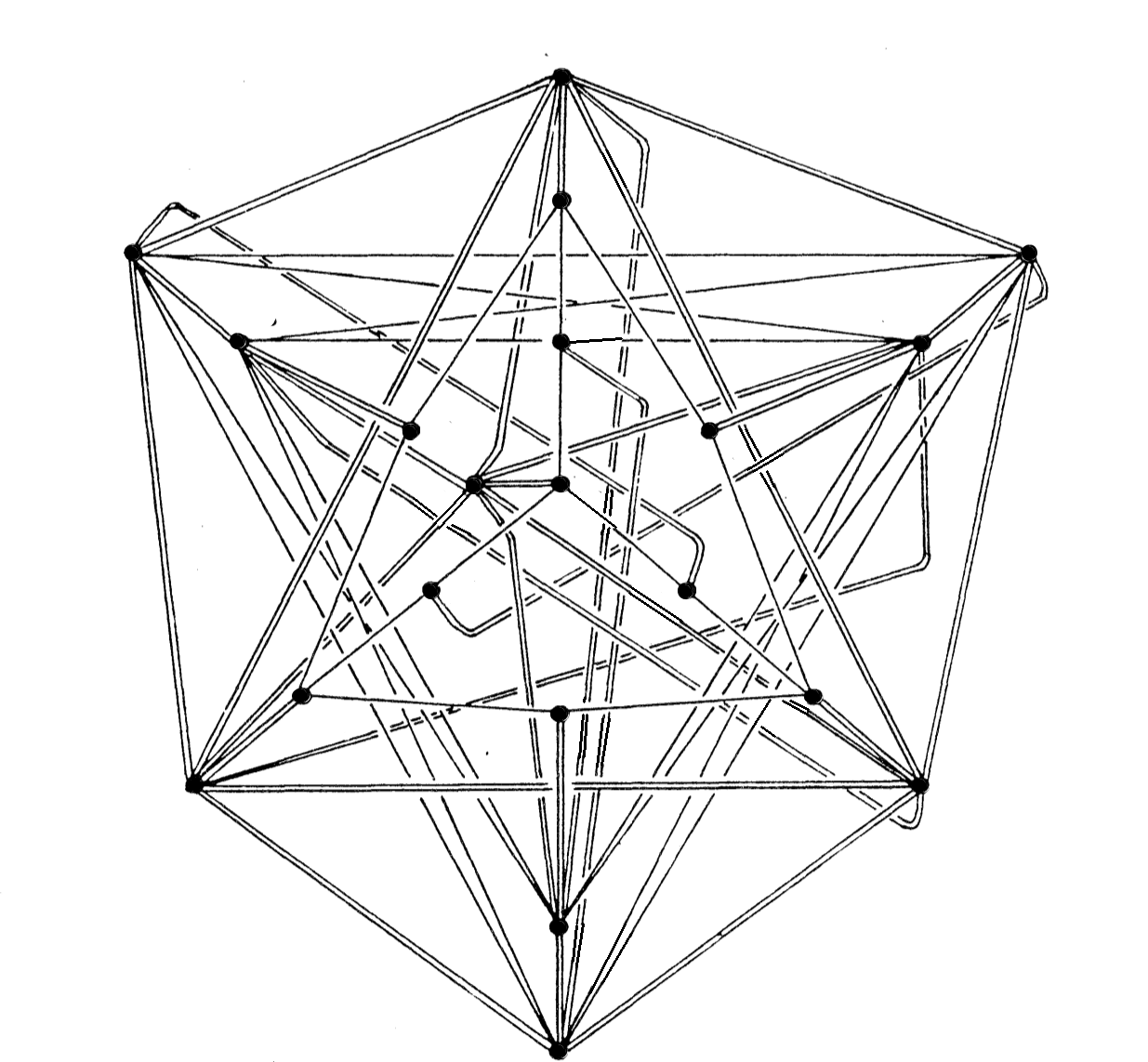} & $\mathfrak{S}_4 \times \bbZ/2\bbZ$ & $\bbZ/2\bbZ$ & $\neq 2,3$ & unique \\  \hline
\rm{VI} & \vspace{1mm} \includegraphics[width=60mm]{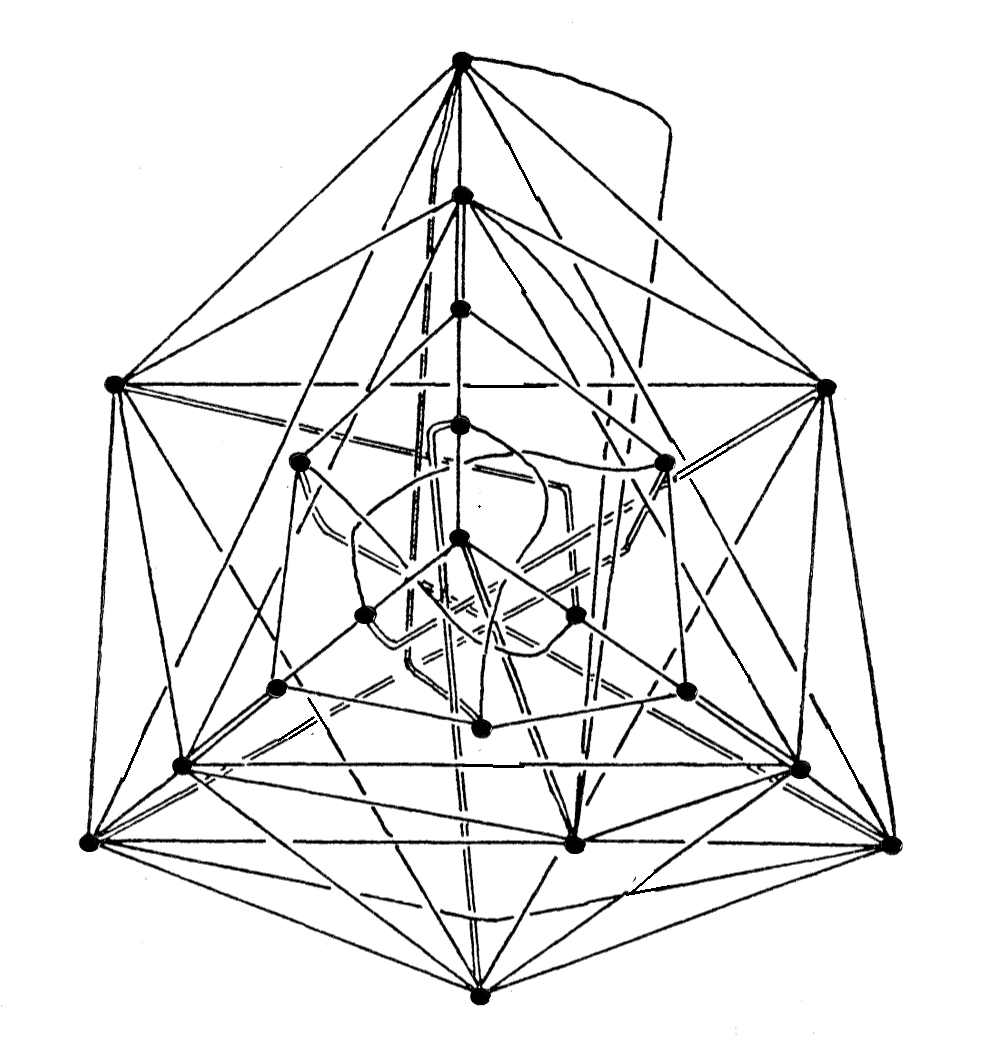} & $\mathfrak{S}_5$ & $\{1\}$ & $\neq 3,5$ & unique \\  \hline
\rm{VII} & \vspace{1mm} \includegraphics[width=60mm]{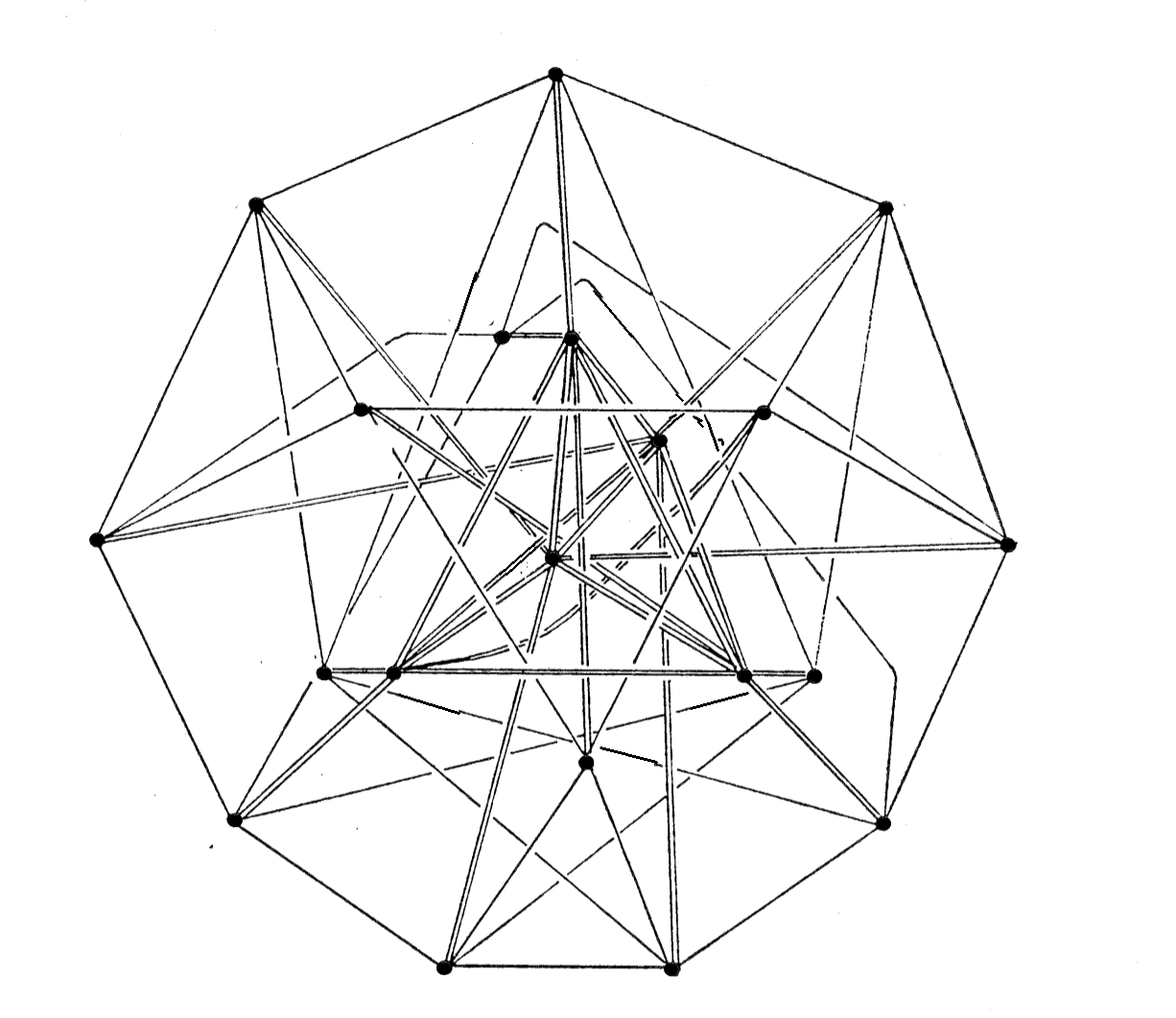} & $\mathfrak{S}_5$ & $\{1\}$ & $\neq 2,5$ & unique \\  \hline
\end{tabular}
\caption{Classification}
\label{main}
\end{table}
\end{namedthm*}

In Table \ref{main}, $\mathfrak{S}_n$ is the symmetric group on $n$ letters, $D_4$ is the dihedral group of order $8$, and for two groups $N$ and $H$, $N \rtimes H$ denotes a semi-direct product of $N$ and $H$.
%

In characteristic $2$, the search for Enriques surfaces with finite automorphism group has been started recently by T. Katsura and S. Kond\=o \cite{KatsuraKondo}. There, the question of existence of the seven types in characteristic $2$ was settled. Our classification shows that the examples of singular Enriques surfaces with finite automorphism group in \cite{KatsuraKondo} are in fact all possible examples of such surfaces. For the classification of classical and supersingular Enriques surfaces with finite automorphism group in characteristic $2$, we refer the reader to \cite{KatsuraKondoMartin}.

\clearpage
\begin{remark*}
As an application of our classification, we determine the semi-symplectic parts of the automorphism groups of Enriques surfaces with finite automorphism group. For the precise statement, we refer the reader to Theorem \ref{semisymplecticthm} and Table \ref{semisymplectic}.
\end{remark*}
Even though our approach to the classification of possible dual graphs is similar to the one of Kond\=o in many aspects, we will encounter several obstacles due to the lack of a Torelli Theorem, the existence of different finite order automorphisms of K3 surfaces \cite{DolgachevKeum1} and a different list of extremal and rational elliptic surfaces in small characteristics \cite{Lang1}, \cite{Lang2}.
We overcome these problems by extending Kond\=o's universal base change construction \cite[Lemma 2.6]{Kondo} of Enriques surfaces with special elliptic fibration $\pi$ (i.e. an elliptic fibration with a $(-2)$-curve as bisection, see Definition \ref{special}) to arbitrary characteristic and by using the explicit description of the covering involution of the canonical cover $\tilde{X}$ of $X$ to obtain a map
\begin{equation*}
jac_2: \MW(J(\pi)) \to \{\text{ special bisections of } \pi \hspace{1mm} \}
\end{equation*}
producing new $(-2)$-curves on $X$ from sections of the Jacobian $J(\pi)$ of $\pi$. Moreover, we exhibit "critical" subgraphs, which are dual graphs of singular fibers of a special elliptic fibration $\pi$ on $X$ together with some special bisection $N$, for each of Kond\=o's seven types and we show that an Enriques surface whose dual graph of all $(-2)$-curves contains such a diagram is one of the seven types. Therefore, we can use the universal base change construction to construct $\pi$ and $N$ and hence the Enriques surface itself. Since the base change construction is universal, we can give an explicit description of the moduli of Enriques surfaces with finite automorphism group. Finally, the equations we give can actually be interpreted as integral models of these Enriques surfaces and some of them were found using the integral models of extremal and rational elliptic surfaces of T. Jarvis, W. Lang and J. Ricks \cite{Lang3}. 


As we have just mentioned, a closer look at our equations reveals that they do in fact define integral models of these surfaces in the following sense.

\begin{namedthm*}{Theorem 11.3}[Integral models]
Let $K \in \{\I,\hdots,\VII\}$ and $P_K$ be as in Table \ref{integralmodelstable}. There is a family $\varphi_K: \mathcal{X} \to \rm{Spec}(\bbZ[\frac{1}{P_K}])$ whose fibers are Enriques surfaces of type $K$ with Picard rank $10$.

\begin{table}[!htb]
\centering
\begin{tabular}{|>{\centering\arraybackslash}m{2.5cm}|>{\centering\arraybackslash}m{6cm}|}
\hline
\rm{Type} & $\rm{P_K}$ \\
\hline \hline
\rm{I} &  $255, 257$ \\ \hline
\rm{II} &  $63, 65$ \\  \hline
\rm{III} & $2$ \\  \hline
\rm{IV} & $2$ \\  \hline
\rm{V} &  $6$ \\  \hline
\rm{VI} &  $15$ \\  \hline
\rm{VII} &  $10$ \\  \hline
\end{tabular}
\caption{Integral models}
\label{integralmodelstable}
\end{table}
\end{namedthm*}

Note that for $K \neq \I,\II$, $P_K$ is exactly the product over the characteristics where type $K$ does not exist. If $K = \I,\II$, we give two integral models to obtain the following corollary, which solves the existence of the seven types over arbitrary fields.
\begin{namedthm*}{Corollary 11.5}
Suppose that there exists an Enriques surface of type $K \in \{\I,\hdots,\VII\}$ in characteristic $p$. Then, there exists an Enriques surface of type $K$ with Picard rank $10$ over $\bbF_p$ (resp. over $\bbQ$ if $p = 0$).
\end{namedthm*}

Moreover, we exhibit special generators of the automorphism groups of Enriques surfaces with finite automorphism group, leading to our third result.

\begin{namedthm*}{Theorem 11.6}
Let $X$ be an Enriques surface of type $K \in \{\I,\hdots,\VII\}$ over a field $k$ such that $\Pic(X) = \Pic(X_{\bar{k}})$. 
\begin{itemize}
\item If $K \neq \III,\IV$, then $\Aut(X)$ is defined over $k$.
\item If $K = \III$, then $\Aut(X)$ is defined over $L \supseteq k$ with $[L:k] \leq 2$.
\item If $K = \IV$, then $\Aut(X)$ is defined over $L \supseteq k$ with $[L:k] \leq 16$. 
\end{itemize}
\end{namedthm*}

Let us explain the structure of the paper.
In \S $2$, we extend Kond\=o's base change construction to positive characteristic after recalling several facts on Enriques surfaces and elliptic fibrations. In \S $3,\hdots,$ \S $9$, we construct Enriques surfaces of types $\I,\hdots,\VII$ and compute their automorphism groups as well as their moduli. After that, in \S $10$, we classify the dual graphs of Enriques surfaces with finite automorphism group, finishing the proof of our Main Theorem. In \S $11$, we explain how to obtain information on the arithmetic of these surfaces and in \S $12$, we give the list of semi-symplectic automorphism groups of Enriques surfaces with finite automorphism group.
 
\medskip
\noindent
{\bf Acknowledgement.}
It is a pleasure for me to thank my Ph.D. advisor C. Liedtke for suggesting this research topic, for his support and helpful discussions and S. Kond\=o for his permission to use his beautiful figures of the dual graphs of Enriques surfaces of type $\I,\hdots,\VII$. Furthermore, I would like to thank I. Dolgachev for many interesting discussions on Enriques surfaces and for explaining the occurrence of Coble surfaces in the context of Enriques surfaces. Research of the author is supported by the DFG Sachbeihilfe LI 1906/3 - 1 "Automorphismen von Enriques Fl\"achen".

\clearpage
\section{Preliminaries}\label{prelim}

\subsection{Generalities on Enriques surfaces, dual graphs and elliptic fibrations}
Here we recall some basic facts about Enriques surfaces, clarify our terminology, and refer the reader to \cite{CossecDolgachev} for proofs and to \cite{Silverman2} for anything related to elliptic curves. In the first ten sections, we will be working over an algebraically closed field $k$.

\begin{definition}

A \emph{K3 surface} is a smooth, projective surface $\tilde{X}$ over $k$ with $\omega_{\tilde{X}} \cong \mathcal{O}_{\tilde{X}}$ and $\mathrm{H}^1(\tilde{X},\mathcal{O}_{\tilde{X}}) = 0$.
An \emph{Enriques surface} $X$ with smooth K3 cover is the quotient of a K3 surface by a fixed point free involution $\sigma$. We call the K3 surface $\tilde{X}$ with $\tilde{X}/\sigma = X$ the \emph{canonical cover} or \emph{K3 cover} of $X$.
\end{definition}

\begin{convention}
From now on, we will drop the "with smooth K3 cover" and we will always assume that the Enriques surfaces we talk about have such a cover.
\end{convention}

\begin{definition}
An \emph{elliptic fibration} (with base curve $\bbP^1$) of a smooth surface $\tilde{X}$ is a surjective morphism $\tilde{\pi}: \tilde{X} \to \bbP^1$ such that almost all fibers are smooth genus $1$ curves, $\tilde{\pi}_* \mathcal{O}_{\tilde{X}} = \mathcal{O}_{\bbP^1}$ and no fiber contains a $(-1)$-curve. We do not require that $\tilde{\pi}$ has a section.
\end{definition}

\begin{proposition}{\rm (Bombieri and Mumford \cite[Theorem 3]{BombieriMumford3})}\label{ellipticpencil}
Every Enriques surface admits an elliptic fibration.
\end{proposition}

The reason why we do not assume that elliptic fibrations have a section is that this is never the case for Enriques surfaces:


\begin{proposition}{\rm(Cossec and Dolgachev \cite[Theorem 5.7.2, Theorem 5.7.5, Theorem 5.7.6]{CossecDolgachev})}\label{typeofdoublefiber}
Let $\pi$ be an elliptic fibration of an Enriques surfaces. Then,
\begin{itemize}
\item if $\Char(k) \neq 2$, $\pi$ has exactly two tame double fibers, both of which are either of multiplicative type or smooth, and
\item if $\Char(k) = 2$, $\pi$ has exactly one wild double fiber, which is either of multiplicative type or a smooth ordinary elliptic curve.
\end{itemize}
\end{proposition}

\begin{remark}
Since being supersingular is an isogeny-invariant, one can check the type of the double fiber on the K3 cover.
\end{remark}
Therefore, the intersection number of any curve with a fiber of an elliptic fibration of an Enriques surface is even. Thus, the best approximation to a section will be a bisection.

\begin{definition}\label{special}
Let $N$ be an irreducible curve on an Enriques surface $X$ and let $\pi$ be an elliptic fibration of $X$.
\begin{itemize}
\item $N$ is a \emph{$(-2)$-curve} if $N^2 = -2$. Equivalently, $N \cong \bbP^1$.
\item $N$ is a \emph{special bisection} of $\pi$ if $N$ is a $(-2)$-curve with $F.N = 2$, where $F$ is a general fiber of $\pi$.
\item If $\pi$ admits a special bisection, we call $\pi$ \emph{special}.
\end{itemize}
\end{definition}

In fact, special elliptic fibrations are much more common than one might think. More precisely, we have the following result of F. Cossec, which was shown by W. Lang also to hold in characteristic $2$.

\begin{proposition}{\rm(Cossec \cite[Theorem 4]{Cossec}, Lang \cite[Theorem A3]{LangE})}\label{nodalisspecial}
An Enriques surface contains a $(-2)$-curve if and only if it admits a special elliptic fibration.
\end{proposition}

Now, we recall some facts on the Jacobian fibrations of elliptic fibrations of Enriques surfaces.

\begin{proposition}{\rm (Cossec and Dolgachev \cite[Theorem 5.7.1]{CossecDolgachev})}
Let $\pi$ be an elliptic fibration of an Enriques surface. Then, the Jacobian fibration $J(\pi)$ of $\pi$ is an elliptic fibration of a rational surface. 
\end{proposition}

Since the group of sections of the Jacobian of an elliptic fibration of an Enriques surface acts on the surface, we will mostly be concerned with extremal and rational elliptic fibrations. The group of sections of an elliptic fibration $\pi$ is also called the \emph{Mordell-Weil group of $\pi$} \cite[III \S 9]{Silverman2}.

\begin{definition}
Let $\pi$ be an elliptic fibration of an Enriques surface and let $J(\pi)$ be its Jacobian. We call $J(\pi)$ and $\pi$ \emph{extremal} if the Mordell-Weil group $\MW(J(\pi))$ is finite.
\end{definition}

We will use the Kodaira-symbols $\I_n (n \geq 1),\I_n^* (n \geq 0),\II,\III,\IV,\II^*,\III^*,$ and $\IV^*$ to denote the singular fibers of an elliptic fibration (see for example \cite[p.354]{Silverman2}). The reducible fibers consist of $(-2)$-curves and their intersection behaviour will play an important role throughout this paper. 

\begin{definition}
Let $M$ be a set of $(-2)$-curves on a smooth surface $X$. 
\begin{itemize}
\item The \emph{dual graph of $M$} is the graph whose vertices are elements of $M$ and two vertices $E_i,E_j \in M$ with $i \neq j$ are joined by an $n$-tuple line if $E_i.E_j = n$. 
\item If $M$ is the set of all $(-2)$-curves on $X$, we will call the corresponding graph the \emph{dual graph of all $(-2)$-curves on $X$}.
\item If $M$ is the set of all $(-2)$-curves contained in singular fibers of an elliptic fibration $\pi$ of $X$, we call $M$ the \emph{dual graph of singular fibers of $\pi$}.
\end{itemize}
\end{definition}

The dual graphs of the singular fibers of type $\I_n (n \geq 2),\I_n^* (n \geq 0),\III,\IV,\II^*,\III^*,$ and $\IV^*$ are $\tilde{A}_{n-1},\tilde{D}_{n+4},\tilde{A}_1,\tilde{A}_2$,$\tilde{E}_8$, $\tilde{E}_7$, and $\tilde{E}_6$, respectively (see \cite[I.6]{Miranda}). Conversely, configurations of $(-2)$-curves whose dual graphs are extended Dynkin diagrams of these types give rise to elliptic fibrations.

\begin{proposition}\label{canonicaltype}{\rm(Kodaira \cite{Kodaira}, Mumford \cite{Mumford})}
A connected, reduced divisor $D$ on an Enriques surface $X$ is equal to the support of a fiber of an elliptic fibration if and only if $D$ is an irreducible genus $1$ curve or the irreducible components of $D$ are $(-2)$-curves whose dual graph is an extended Dynkin diagram of type $\tilde{A}$-$\tilde{D}$-$\tilde{E}$.
\end{proposition}

Note that one cannot always reconstruct the fiber type from the graph. Using this notation, we can give the list of extremal and rational elliptic fibrations in every characteristic due to R. Miranda, U. Persson and W. E. Lang.

\begin{proposition}{\rm (Miranda and Persson \cite{MirandaPersson}, Lang \cite{Lang1}, \cite{Lang2})}
Let $\pi$ be an extremal fibration of a rational surface. Then, the singular fibers of $\pi$ are given in Table \ref{extremalrational}.

The extremal and rational elliptic surfaces with singular fibers $(\I_0^*,\I_0^*)$ in characteristic $\neq 2$ and the ones with singular fiber $(\I_4^*)$ in characteristic $2$ form $1$-dimensional families and all other fibrations are unique.

\begin{table}[!htb]
\centering
\begin{tabular}{|>{\centering\arraybackslash}m{3.5cm}|>{\centering\arraybackslash}m{2.5cm}|>{\centering\arraybackslash}m{2.5cm}|>{\centering\arraybackslash}m{2.5cm}|}
\hline
\rm{$\Char(k) \neq 2,3,5$} & \rm{$\Char(k) = 5$} & \rm{$\Char(k) = 3$} & \rm{$\Char(k) = 2$}\\
\hline \hline
{\rm($\II^*,\II$)} & {\rm($\II^*,\II$)} & {\rm($\II^*$)} & {\rm($\II^*$)} \\
{\rm($\III^*,\III$)} & {\rm($\III^*,\III$)} & {\rm($\III^*,\III$)} & -- \\
{\rm($\IV^*,\IV$)} & {\rm($\IV^*,\IV$)} & -- & {\rm($\IV^*,\IV$)} \\ 
{\rm($\I_0^*,\I_0^*$)} & {\rm($\I_0^*,\I_0^*$)} & {\rm($\I_0^*,\I_0^*$)} & -- \\ 
{\rm($\II^*,\I_1,\I_1$)} & {\rm($\II^*,\I_1,\I_1$)} & {\rm($\II^*,\I_1$)} & {\rm($\II^*,\I_1$)}\\ 
{\rm($\III^*,\I_2,\I_1$)} & {\rm($\III^*,\I_2,\I_1$)} & {\rm($\III^*,\I_2,\I_1$)} & {\rm($\III^*,\I_2$)} \\
{\rm($\IV^*,\I_3,\I_1$)} & {\rm($\IV^*,\I_3,\I_1$)} & {\rm($\IV^*,\I_3$)} & {\rm($\IV^*,\I_3,\I_1$)}\\
{\rm($\I_4^*,\I_1,\I_1$)} & {\rm($\I_4^*,\I_1,\I_1$)} & {\rm($\I_4^*,\I_1,\I_1$)} & {\rm($\I_4^*$)}\\ 
{\rm($\I_2^*,\I_2,\I_2$)} & {\rm($\I_2^*,\I_2,\I_2$)} & {\rm($\I_2^*,\I_2,\I_2$)} & -- \\ 
{\rm($\I_1^*,\I_4,\I_1$)} & {\rm($\I_1^*,\I_4,\I_1$)} & {\rm($\I_1^*,\I_4,\I_1$)} & {\rm($\I_1^*,\I_4$)}\\ 
{\rm($\I_9,\I_1,\I_1,\I_1$)} & {\rm($\I_9,\I_1,\I_1,\I_1$)} & {\rm($\I_9,\II$)} & {\rm($\I_9,\I_1,\I_1,\I_1$)}\\
{\rm($\I_8,\I_2,\I_1,\I_1$)} & {\rm($\I_8,\I_2,\I_1,\I_1$)} & {\rm($\I_8,\I_2,\I_1,\I_1$)} & {\rm($\I_8,\III$)}\\
{\rm($\I_5,\I_5,\I_1,\I_1$)} & {\rm($\I_5,\I_5,\II$)} & {\rm($\I_5,\I_5,\I_1,\I_1$)} & {\rm($\I_5,\I_5,\I_1,\I_1$)}\\ 
{\rm($\I_6,\I_3,\I_2,\I_1$)} & {\rm($\I_6,\I_3,\I_2,\I_1$)} & {\rm($\I_6,\I_3,\III$)} & {\rm($\I_6,\IV,\I_2$)}\\ 
{\rm($\I_4,\I_4,\I_2,\I_2$)} & {\rm($\I_4,\I_4,\I_2,\I_2$)} & {\rm($\I_4,\I_4,\I_2,\I_2$)} & --\\ 
{\rm($\I_3,\I_3,\I_3,\I_3$)} & {\rm($\I_3,\I_3,\I_3,\I_3$)} & -- & {\rm($\I_3,\I_3,\I_3,\I_3$)}  \\ \hline
\end{tabular}
\caption{Extremal and rational elliptic fibrations}
\label{extremalrational}
\end{table}


\end{proposition}

\begin{remark}
From Table \ref{extremalrational} we see that the fibrations in small characterstics differ from the characteristic $0$ cases only if either a $\II^*$ fiber is involved or if the characteristic divides the number of simple components of some fiber of the fibration.
\end{remark}

In fact, the Shioda-Tate formula implies that the dual graph of $(-2)$-curves contained in singular fibers of an elliptic fibration $\pi$ determines whether $\pi$ is extremal or not.

\begin{lemma}{\rm(Shioda, \cite[Corollary 1.5]{Shioda3})}\label{shiodatate}
Let $\pi$ be an elliptic fibration of a rational surface or of an Enriques surface. Then, $\pi$ is extremal if and only if the lattice spanned by the fiber components of $\pi$ has rank $9$.
\end{lemma}

Extremal elliptic fibrations of Enriques surfaces over the complex numbers were studied by the author in \cite{Master}, where he classified those extremal fibrations with at least one reducible double fiber.
\subsection{Base Change Construction}\label{basechange}

\begin{notation}
Let $\pi: X \to \bbP^1$ be an elliptic fibration with section of a rational surface or of a K3 surface. We denote the composition in $\MW(\pi)$ with respect to some fixed zero section by $\oplus$, the inverse of a section $P$ is denoted by $\ominus P$ and the translation by a section $P$ is denoted by $t_P$. By abuse of notation, we will also use $t_P$ for the induced automorphism of $X$.
\end{notation}

Over the complex numbers, the following is due to S. Kond\=o \cite[p.199]{Kondo}. There are generalizations of this result in \cite{HulekSchütt} and \cite{Schütt2}. Since we need this construction for our classification, we will extend it to arbitrary characteristic.

\begin{lemma}\label{universal}
Let $f: \tilde{X} \to X$ be the canonical cover of an Enriques surface $X$ and let $\sigma$ be the covering involution. Let $\pi: X \to \mathbb{P}^1$ be a special elliptic fibration of $X$ with a special bisection $N$, let $F$ be a double fiber of $\pi$ and let $J(\pi): J(X) \to \mathbb{P}^1$ be the Jacobian fibration associated to $\pi$. Let $\tilde{\pi}$ be the fibration of $\tilde{X}$ induced by $|f^{-1}F|$ and denote by $\varphi:|f^{-1}F| = \bbP^1 \to \bbP^1 = |2F|$ the induced morphism on the base curve. 

Then,
\begin{enumerate}
\item $N$ splits into two sections $N^+$ and $N^-$ of $\tilde{\pi}$. In particular, the minimal proper smooth models of the base changes of $J(\pi)$ and $\pi$ along $\varphi$ are isomorphic.
\item Choose $N^+$ as the zero section of $\tilde{\pi}$. Then,
$J(\sigma) = t_{\ominus N^-} \circ \sigma$ is an involution whose quotient, after minimalizing the obtained fibration, is $J(\pi)$.
\item $N^-$ satisfies $N^-.N^+ = 0$, $J(\sigma)(N^-) = \ominus N^-$ and it does not meet the preimage of a singular double fiber of $\pi$ in the identity component.
\end{enumerate}
\end{lemma}

The main tool to establish this result in arbitrary characteristic is the following lemma, which is a close study of how automorphisms of the generic fiber of an elliptic fibration with section extend to special fibers. For lack of a reference, we will give a proof.

\begin{lemma}\label{specialauts}
Let $R$ be a discrete valuation ring and let $K = Quot(R)$. Let $(E,O)$ be an elliptic curve over $K$
and let $\calE$ be the N\'eron model of $E$ over $R$. Let $E_0$ be the identity component of the special fiber of $\calE$. Let $\rho: \Aut(E,O) \to \Aut(E_0,\overline{O}|_{E_0})$ be the natural map obtained from the N\'eron mapping property and restriction. Then,
$\rho$ is injective if and only if one of the following holds:
\begin{itemize}
\item $\Char(k) \not \in \{2,3\}$
\item $\Char(k) \in \{2,3\}$ and $E_0$ is not additive.
\end{itemize}
If $E_0$ is additive, then $\ker(\rho)$ consists of all elements of order $p^n$, where $p = \Char(k)$.
\end{lemma}

\begin{proof}
We will compute the reduction of the automorphisms explicitly using Weierstrass equations and the description of automorphisms in \cite[p.411]{Silverman} (see also \cite[p.364]{Silverman2} for an exposition of Tate's algorithm). Throughout, we denote by $\pi$ a uniformizer of $R$.

If $\Char(k) \geq 5$, then we use a minimal and integral Weierstrass equation 
\begin{equation*}
y^2 = x^3  + a_4x + a_6.
\end{equation*}
Since all $g \in \Aut(E,O)$ are of the form $g: (x,y) \mapsto (\zeta^2x,\zeta^3y)$ for some $12$-th root of unity $\zeta$, they induce non-trivial automorphisms of $E_0$ independently of $a_4$ and $a_6$.

If $\Char(k) = 3$, then we use a minimal and integral Weierstrass equation 
\begin{equation*}
y^2 = x^3  + a_2x^2 + a_4x + a_6.
\end{equation*}
If $a_2 \neq 0$, then the same argument as before works, so we may assume $a_2 = 0$. Then, an automorphism $g \in \Aut(E,O)$ is given by $g: (x,y) \mapsto (\zeta^2x+r,\zeta^3y)$, where $\zeta^4 = 1$ and $r^3 + a_4r+(1-\zeta^2)a_6 = 0$. If $\zeta \neq 1$, then $\rho(g) \neq \rm{id}$, since $\zeta$ does not depend on $a_4$ and $a_6$. But if $\zeta = 1$ and $r = \pm \sqrt{-a_4}$, then $\rho(g)$ is trivial if and only if $\pi \mid a_4$, i.e. if and only if $E_0$ is of additive type.

If $\Char(k) = 2$, then we use a minimal and integral Weierstrass equation
\begin{equation*}
y^2 + a_1xy + a_3y = x^3  + a_2x^2 + a_4x + a_6.
\end{equation*}
The inversion involution $g \in \Aut(E,O)$ is given by $(x,y) \mapsto (x, y + a_1x + a_3)$. Thus, $\rho(g)$ is trivial if and only if $\pi \mid a_1,a_3$, i.e. if and only if $E_0$ is of additive type. Now if $j(E_0) = 0$, then we can assume $a_1 = a_2 = 0$. An automorphism $g \in \Aut(E,O)$ is given by $g: (x,y) \mapsto (\zeta^2x + s^2, \zeta^3y+ \zeta^2sx + t)$, where $\zeta^3 = 1$, $s^4 + a_3s + (1-\zeta)a_4 = 0$ and $t^2 + a_3t + s^6 + a_4s^2 = 0$. If $\zeta \neq 1$, then we have $\rho(g) \neq \rm{id}$. Therefore, assume $\zeta = 1$ and $s^3 + a_3 = 0$. Now, $\rho(g) = \rm{id}$ and if and only if $\pi \mid a_3$, i.e. if and only if $E_0$ is additive.

\end{proof}

\begin{proof}[Proof of Lemma \ref{universal}]
Since $\tilde{X} \to X$ is \'etale of degree $2$, every $(-2)$-curve on $X$ splits into two disjoint $(-2)$-curves on $\tilde{X}$. In particular, $N$ splits into two $(-2)$-curves $N^+$ and $N^-$. 
We claim that a general fiber of $\pi$ also splits into two components. Indeed, suppose that a general fiber does not split into two components. Then, $\Char(k) = 2$ and $\sigma$ acts on every fiber of $\tilde{\pi}$. Since $\sigma$ is fixed point free and additive and supersingular fibers do not admit fixed point free involutions, every fiber of $\tilde{\pi}$ would have to be multiplicative or ordinary, which is absurd. Both $N^+$ and $N^-$ have to be sections of the fibration $\tilde{\pi}$, since a general fiber of $\pi$ splits into two components $F_1$ and $F_2$, both of which are fibers of $\tilde{\pi}$, and therefore $2 = N.F = 2 N^+.F_1 = 2N^-.F_1$. 

Next, we show that $J(\sigma)$ is indeed an involution. Let $F_0$ be the identity component of a fiber of $\tilde{\pi}$ which is fixed (not necessarily pointwise) by $\sigma$. Note that $F_0$ is either multiplicative or smooth by Proposition \ref{typeofdoublefiber}. Since $\sigma$ is fixed point free, it induces a translation on $F_0$ if $F_0$ is smooth. Moreover, because $J(\sigma)(N^+) = t_{\ominus N^-} \circ \sigma (N^+) = N^+$, $J(\sigma)|_{F_0}$ is the identity if $F_0$ is smooth, and it can have at most order $2$ if $F_0$ is multiplicative. Together we obtain $J(\sigma)^2|_{F_0} = \rm{id}$ in any case. Now, $J(\sigma)^2$ fixes $\tilde{\pi}$ and hence it is an automorphism of the generic fiber of $\tilde{\pi}$ fixing the zero section $N^+$. By Lemma \ref{specialauts}, $J(\sigma)^2 = \rm{id}$, because it restricts to the identity on $F_0$. Since $J(\sigma)(N^+) = N^+$, this section descends to the quotient and we obtain $J(\pi)$.

Finally, if $\pi$ has a singular double fiber $F$ of type $\I_n$, the preimage of $F$ in $\tilde{X}$ is a fiber $F'$ of $\tilde{\pi}$ of type $\I_{2n}$, since this happens with the corresponding fiber on the Jacobian. Now, $\sigma$ has to act without fixed points, hence it acts as a rotation of order $2$ on the corresponding $\tilde{A}_{2n-1}$ diagram, while $J(\sigma)$ fixes the diagram. In particular, the preimage of $N$ meets two opposite curves of the diagram, i.e. $N^-$ does not meet the identity component of $F'$ if we choose $N^+$ to be the zero section of $\tilde{\pi}$.
\end{proof}

In particular, we obtain a distinguished non-zero section of $\tilde{\pi}$ if $\tilde{\pi}$ arises as the base change of a special elliptic fibration $\pi$ of an Enriques surface with a given special bisection. Conversely, we will see that we can reconstruct $\pi$ from $J(\pi)$ by exhibiting a suitable section on a degree $2$ base change of $J(\pi)$. This has been studied by K. Hulek and M. Sch\"utt in \cite{HulekSchütt} using quadratic twists. Since in our case $J(\pi)$ is an extremal and rational elliptic fibration and extremal and rational elliptic surfaces are classified, we can approach the classification problem in a very explicit way. First, let us clarify what we mean by a "suitable section".

\begin{definition}\label{EnriquesSection}
Let $J(\pi): J \to \mathbb{P}^1$ be an elliptic fibration of a rational surface $J$ with zero section $N^+$. Let $\varphi: \mathbb{P}^1 \to \mathbb{P}^1$ be a separable degree $2$ morphism such that no branch point of $\varphi$ is a point of additive reduction of $J(\pi)$. If $\Char(k) = 2$, assume further that the branch point is not a point of good supersingular reduction of $J(\pi)$. Then, a minimal proper smooth model of the base change $\tilde{\pi}$ of $J(\pi)$ along $\varphi$ is an elliptic fibration of a K3 surface $\tilde{X}$.  Denote the zero section of $\tilde{\pi}$ also by $N^+$ and let $J(\sigma)$ be a covering involution of $\tilde{X} \to J$ such that $J(\sigma)(N^+) = N^+$. A section $N^-$ of $\tilde{\pi}$ is called a \emph{$J(\pi)$-Enriques section} of $\tilde{\pi}$ if

\begin{enumerate}
\item $N^-.N^+ = 0$,
\item $J(\sigma)(N^-) = \ominus N^-$, and
\item $N^-$ does not meet the identity component of the fiber over $\varphi^{-1}(x)$ if $\varphi$ is branched over a point $x$ with $J(\pi)^{-1}(x)$ singular.
\end{enumerate}
\end{definition}

\begin{remark}
Observe that these are exactly the properties satisfied by $N^-$ in Lemma \ref{universal} $(3)$.
\end{remark}

\begin{remark}
We will encounter several examples of such $J(\pi)$-Enriques sections throughout this paper. The quickest way to achieve conditions $(1)$ and $(2)$ is to take for $N^-$ an everywhere integral (i.e. $N^-.N^+ = 0$) $2$-torsion section of $\tilde{\pi}$, since such a section will be a base change of a $2$-torsion section of $J(\pi)$. However, this does not guarantee condition $(3)$ to hold, as we will see later.
\end{remark}

The following is the main ingredient in our approach to the classification. Over the complex numbers, this is implicitly contained in \cite{Kondo} (for a variation of this, see \cite{HulekSchütt}).

\begin{proposition}\label{Kondo}
With notation as in the above definition, let $N^-$ be a section of $\tilde{\pi}$ such that $J(\sigma)(N^-) = \ominus N^-$ and $N^+.N^- = 0$. Then, the quotient of $\tilde{X}$ by the involution $\sigma := t_{N^-} \circ J(\sigma)$ is an Enriques surface $X$ with a special elliptic fibration $\pi$ induced by $\tilde{\pi}$ if and only if $N^-$ is a $J(\pi)$-Enriques section. The Jacobian of $\pi$ is $J(\pi)$ and the double fibers of $\pi$ occur over the branch points of $\varphi$.
\end{proposition}

\begin{proof}
Let us first show that $\sigma$ is an involution. Denote by $F_0$ a fiber which is fixed (not necessarily pointwise) by $J(\sigma)$. We have $\sigma^2|_{F_0} = t_{N^-}|_{F_0} \circ J(\sigma)|_{F_0} \circ t_{N^-}|_{F_0} \circ J(\sigma)|_{F_0} = t_{N^-}|_{F_0} \circ t_{\ominus N^-}|_{F_0} = \rm{id}|_{F_0}$ and since $F_0$ is either multiplicative or smooth and $\sigma^2$ fixes $\tilde{\pi}$ and $N^+$, we obtain $\sigma^2 = \rm{id}$ by Lemma \ref{specialauts}.

Since translation by a section fixes all fibers and $J(\sigma)$ fixes at most two fibers $F_0$ and $F_1$, we have $\Fix(\sigma) \subseteq F_0 \cup F_1$. If $F \in \{F_0,F_1\}$ is smooth, we claim that $J(\sigma)$ acts trivially on $F$. In characteristic different from $2$, this follows because $J(\sigma)$ acts non-trivially on a global $2$-form, and in characteristic $2$, $J(\sigma)|_F$ is either the identity or a hyperelliptic involution, since it fixes $N^+$ and $F$ is ordinary. The latter case is impossible by \cite[Theorem 1]{DolgachevKeum}.
 Since $J(\sigma)$ acts trivially on a smooth fiber $F \in \{F_0,F_1\}$ and $N^-.N^+ = 0$, $\sigma|_F = t_{N^-}|_F$ will have no fixed points on $F$. As for a multiplicative fiber $F \in \{F_0,F_1\}$, $J(\sigma)$ fixes the components of $F$ (not necessarily pointwise), hence $\sigma$ has fixed points if and only if $N^-$ meets the identity component of this fiber, i.e. if and only if $N^-$ is not a $J(\pi)$-Enriques section.

Now, if $N^-$ is a $J(\pi)$-Enriques section, this means that the quotient of $\tilde{X}$ by $\sigma$ is an Enriques surface $X$. Moreover, the divisors $F$ and $N^++N^-$ are fixed by $\sigma$ and thus descend to $X$, giving a special elliptic fibration $\pi$ on $X$. Additionally, $F_0$ and $F_1$ descend to the two double fibers of $\pi$ and $J(\pi)$ is the Jacobian of $\pi$ by construction.
\end{proof}

\begin{remark}\label{coble}
If $\sigma$ has fixed points
, we claim that it actually has a fixed locus of dimension $1$. To see this, note that $\sigma$ fixing two points on a $(-2)$-curve in characteristic $2$ means that the whole curve is fixed (see also \cite{DolgachevKeum}). For the other characteristics, we refer the reader to \cite{Zhang}. After contracting the fixed locus, the quotient by $\sigma$ is nothing but a rational log Enriques surface of index $2$ \cite{Zhang2} and its minimal resolution is a Coble surface (see \cite{DolgachevZhang}). We will not study these surfaces here, but the attentive reader will see them occur naturally as degenerations of the models we give for the surfaces in our Main Theorem.
\end{remark}

\begin{remark}
We see from the proof that one can also obtain an Enriques surface as quotient by $\sigma$ if one weakens the assumption $N^+.N^- = 0$ to $N^+ \cap N^- \cap F_0 = N^+ \cap N^- \cap F_1 = \emptyset$. However, in general, this will not produce a \emph{smooth} bisection. For more on this, see \cite{HulekSchütt}.
\end{remark}

With this explicit and universal construction at our disposal, we can have a look at the relation between special bisections of an elliptic fibration of an Enriques surface and sections of its Jacobian.

\begin{corollary}\label{jac2}
Let $\pi$ be a special elliptic fibration of an Enriques surface $X$ with a special bisection $N$ splitting into $N^+$ and $N^-$ on the K3 cover $\tilde{X}$ of $X$. There is a map
\begin{equation*}
jac_2: \MW(J(\pi)) \to \{ \text{special bisections of } \pi \},
\end{equation*}
which is
\begin{itemize}
\item injective if $N^-$ is not $2$-torsion after fixing $N^+$ as the zero section, and
\item $2$-to-$1$ onto its image otherwise.
\end{itemize} 
Moreover, $\MW(J(\pi))$ acts transitively on the image of $jac_2$ via its action on $X$.
\end{corollary}

\begin{proof}
We use the notation of Lemma \ref{universal}. There is a natural injection $\MW(J(\pi)) \to \MW(\tilde{\pi})$ and using this, we will consider sections of $J(\pi)$ as sections of $\tilde{\pi}$ by abuse of notation. Let $P \in \MW(J(\pi))$. Since $P$ comes from $J(\pi)$, it is fixed by $J(\sigma)$. Now, we compute
\begin{equation*}
P.\sigma(P) = P.(t_{N^-} \circ J(\sigma))(P) = P.(P \oplus N^-)= N^+.N^- = 0.
\end{equation*}
Therefore, the divisor $P + \sigma(P)$ descends to a $(-2)$ curve $jac_2(P)$ on $X$, which is necessarily a bisection of $\pi$, since $2 = (P + \sigma(P)).\tilde{F} = jac_2(P).F$, where $\tilde{F}$ (resp. $F$) is a general fiber of $\tilde{\pi}$ (resp. $\pi$).
For the injectivity, observe that $\sigma(P) \in \MW(J(\pi))$ if and only if $J(\sigma)(\sigma(P)) = \sigma(P)$, i.e. if and only if
\begin{equation*}
P \oplus N^- = (t_{N^-} \circ J(\sigma))(P) = \sigma(P) = J(\sigma)(\sigma(P)) = P \ominus N^- ,
\end{equation*}
which happens if and only if $N^-$ is $2$-torsion. The statement about the action of $\MW(J(\pi))$ is clear by construction of $jac_2$.
\end{proof}

To compute the intersection behaviour of the special bisections obtained via $jac_2$, we will use the height pairing on $\MW(\tilde{\pi})$.

\begin{proposition}{\rm (Shioda \cite{Shioda2})}\label{heightpairingdef}
Let $\tilde{\pi}$ be an elliptic fibration of a K3 surface with zero section $N^+$.
The pairing
\begin{eqnarray*}
\MW(\tilde{\pi}) \times \MW(\tilde{\pi}) &\to& \bbQ \\
(P,Q) &\mapsto& \langle P,Q\rangle = 2 + P.N^+ + Q.N^+ - P.Q - \sum_{\nu \in \bbP^1} contr_{\nu} (P,Q),
\end{eqnarray*}
where the $contr_{\nu} (P,Q)$ are local correction terms depending on the intersection of $P$ and $Q$ with the fiber over $\nu$, is a symmetric, bilinear pairing on $\MW(\tilde{\pi})$, which induces the structure of a positive definite lattice on $\MW(\tilde{\pi}) / \MW(\tilde{\pi})_{tors}$. It is called the \emph{height pairing} on $\MW(\tilde{\pi})$. We write $h(P)$ for $\langle P,P \rangle$.
\end{proposition}

\begin{remark}\label{heightpairingobservation}
Note that this implies immediately that $h(P) = 0$ if and only if $P \in \MW(\tilde{\pi})_{tors}$. Moreover, $\langle P,Q \rangle = 0$ as soon as $P$ or $Q$ is in $\MW(\tilde{\pi})_{tors}$.
\end{remark}

For the reader's convenience, we recall the correction terms of the height pairing following \cite[p.52]{SchüttShioda}.
First, we have to fix a numbering of the simple components of a reducible fiber $F_\nu$ of an elliptic fibration $\pi$ with zero section $N^+$ depending on the dual graph $\Gamma$ of $F_\nu$. In any case, denote the component of $F_\nu$ which meets $N^+$ by $E_0$.
\begin{itemize}
\item If $\Gamma = \tilde{A}_{n-1}$, denote the components of $F_\nu$ by $E_0,\hdots,E_{n-1}$ such that $E_i.E_j = 1$ if and only if $i - j = \pm 1$ mod $n$.
\item If $\Gamma = \tilde{D}_{n+4}$, denote the simple components of $F_\nu$ by $E_0,E_1,E_2,$ and $E_3$ such that $E_1$ is a simple component with minimal distance to $E_0$.
\end{itemize}
Now, let $P,Q \in \MW(\pi)$ such that $P$ meets $E_i$ and $Q$ meets $E_j$ and assume $i \leq j$. If $i = 0$, the correction term is $0$. Otherwise, the value of $contr_{\nu}(P,Q)$ is given in the following Table \ref{heightpairing}.

\begin{table}[!htb]
\centering
\begin{tabular}{|>{\centering\arraybackslash}m{2.5cm}|>{\centering\arraybackslash}m{1.5cm}|>{\centering\arraybackslash}m{1.5cm}|>{\centering\arraybackslash}m{3.5cm}|>{\centering\arraybackslash}m{2.5cm}|} \hline
$\Gamma$ & $\tilde{E}_7$ & $\tilde{E}_6$ & $\tilde{D}_{n+4}$ & \vspace{1mm} $\tilde{A}_{n-1}$ \\ [1mm] \hline \hline
Case $i = j$ & $\frac{3}{2}$ & $\frac{4}{3}$ &  $\begin{cases} 1 &\text{ if } i =1 \\ 1+ \frac{n}{4} &\text{ else} \end{cases}$ & $\frac{i(n-i)}{n}$ \\
Case $i < j$ & - & $\frac{2}{3}$ & $ \begin{cases} \frac{1}{2} &\text{ if }  i =1 \\ \frac{1}{2}+ \frac{n}{4} &\text{ else} \end{cases}$ & $ \frac{i(n-j)}{n}$ \\ \hline
\end{tabular}
\caption{Correction terms for the height pairing}
\label{heightpairing}
\end{table}

\subsection{Example} \label{example}
We keep the notation introduced in the previous subsection. Since we know how sections coming from $J(\pi)$ intersect the fibers of $\tilde{\pi}$, we can compute the intersection behaviour of the corresponding bisections on $X$ once we know how $N^-$ intersects the fibers of $\tilde{\pi}$. But this is already determined by the intersection behaviour of the special bisection $N$ on $X$ with the fibers of $\pi$. We will leave these computations to the reader but give a detailed description of the procedure in the following example.

Suppose an Enriques surface contains the following dual graph of $(-2)$-curves with $N$ as indicated:

\centerline{
\xy
@={(0,0),(0,10),(0,20),(0,30),(10,0),(10,10),(10,20),(10,30),(20,0),(30,0),(40,0)}@@{*{\bullet}};
(0,0)*{};(0,30)*{}**\dir{-};
(10,0)*{};(10,30)*{}**\dir{-};
(0,0)*{};(10,0)*{}**\dir{-};
(10,30)*{};(0,30)*{}**\dir{-};
(10,0)*{};(20,0)*{}**\dir{-};
(10,10)*{};(20,0)*{}**\dir{-};
(30,0)*{};(20,0)*{}**\dir{-};
(30,0)*{};(40,0)*{}**\dir2{-};
(21,3)*{N};
\endxy}

This is the dual graph of a special elliptic fibration with a singular fiber of type $\I_8$ and a double fiber of type $\I_2$. Note that the $\I_2$ fiber has to be double, since $N$ meets its components only once and $N$ is a bisection.
On the K3 cover, this yields the following configuration:

\vspace{3mm}
\centerline{
\xy
@={(0,0),(0,10),(0,20),(0,30),(10,0),(10,10),(10,20),(10,30),(20,0),(30,20),(30,30),(50,0),(50,10),(50,20),(50,30),(60,0),(60,10),(60,20),(60,30),(25,25),(35,25),(40,10)}@@{*{\bullet}};
(0,0)*{};(0,30)*{}**\dir{-};
(10,0)*{};(10,30)*{}**\dir{-};
(0,0)*{};(10,0)*{}**\dir{-};
(10,30)*{};(0,30)*{}**\dir{-};
(10,0)*{};(20,0)*{}**\dir{-};
(30,0)*{};(20,0)*{}**\dir{-};
(25,25)*{};(30,30)*{}**\dir{-};
(25,25)*{};(30,20)*{}**\dir{-};
(35,25)*{};(30,30)*{}**\dir{-};
(35,25)*{};(30,20)*{}**\dir{-};
(17,3)*{N^+};
(43,13)*{N^-};
(10,10)*{};(50,10)*{}**\dir{-};
(20,0)*{};(25,25)*{}**\dir{-};
(50,0)*{};(60,0)*{}**\dir{-};
(50,30)*{};(50,0)*{}**\dir{-};
(60,30)*{};(60,0)*{}**\dir{-};
(50,30)*{};(60,30)*{}**\dir{-};
(50,0)*{};(30,0)*{}**\dir{-};
(40,10)*{};(35,25)*{}**\dir{-};
\endxy}
\vspace{3mm}

On the other hand, we know that the Jacobian of $\pi$ together with its four sections $P_1,P_2,P_3,$ and $P_4$ has the following dual graph:

\centerline{
\xy
(-10,33)*{P_3};
(-10,13)*{P_4};
(20,23)*{P_2};
(20,3)*{P_1};
@={(0,0),(0,10),(0,20),(0,30),(10,0),(10,10),(10,20),(10,30),(20,0),(30,0),(40,0),(-10,10),(-10,30),(20,20)}@@{*{\bullet}};
(0,0)*{};(0,30)*{}**\dir{-};
(10,0)*{};(10,30)*{}**\dir{-};
(0,0)*{};(10,0)*{}**\dir{-};
(10,30)*{};(0,30)*{}**\dir{-};
(10,0)*{};(20,0)*{}**\dir{-};
(30,0)*{};(20,0)*{}**\dir{-};
(30,0)*{};(40,0)*{}**\dir2{-};
(-10,10)*{};(0,10)*{}**\dir{-};
(-10,30)*{};(0,30)*{}**\dir{-};
(10,20)*{};(20,20)*{}**\dir{-};
(40,0)*{};(20,20)*{}**\dir{-};
(30,0)*{};(-10,30)*{}**\crv{(30,50)};
(40,0)*{};(-10,10)*{}**\crv{(0,-20)};
\endxy}
\vspace{5mm}

One can explicitly compute the dual graph of a degree $2$ base change of $J(\pi)$ ramified over the $\I_2$ fiber (and not ramified over $\I_8$):

\vspace{-3mm}
\centerline{
\xy
(18,3)*{P_1};
(20,43)*{P_3};
(40,7)*{P_2};
(40,53)*{P_4};
@={(0,0),(0,10),(0,20),(0,30),(10,0),(10,10),(10,20),(10,30),(20,0),(30,20),(30,30),(50,0),(50,10),(50,20),(50,30),(60,0),(60,10),(60,20),(60,30),(25,25),(35,25),(40,10),(20,40),(40,50)}@@{*{\bullet}};
(0,0)*{};(0,30)*{}**\dir{-};
(10,0)*{};(10,30)*{}**\dir{-};
(0,0)*{};(10,0)*{}**\dir{-};
(10,30)*{};(0,30)*{}**\dir{-};
(10,0)*{};(20,0)*{}**\dir{-};
(30,0)*{};(20,0)*{}**\dir{-};
(25,25)*{};(30,30)*{}**\dir{-};
(25,25)*{};(30,20)*{}**\dir{-};
(35,25)*{};(30,30)*{}**\dir{-};
(35,25)*{};(30,20)*{}**\dir{-};
(20,0)*{};(25,25)*{}**\dir{-};
(50,0)*{};(60,0)*{}**\dir{-};
(50,30)*{};(50,0)*{}**\dir{-};
(60,30)*{};(60,0)*{}**\dir{-};
(50,30)*{};(60,30)*{}**\dir{-};
(50,0)*{};(30,0)*{}**\dir{-};
(40,10)*{};(50,20)*{}**\dir{-};
(40,10)*{};(10,20)*{}**\dir{-};
(40,10)*{};(35,25)*{}**\dir{-};
(20,40)*{};(0,30)*{}**\dir{-};
(20,40)*{};(60,30)*{}**\dir{-};
(20,40)*{};(25,25)*{}**\dir{-};
(40,50)*{};(35,25)*{}**\dir{-};
(40,50)*{};(0,10)*{}**\crv{(-20,40)};
(40,50)*{};(60,10)*{}**\crv{(80,40)};
\endxy}
\vspace{1mm}

To put this picture together with the second one, we set $N^+ = P_1$ as the zero section, add the sections $N^- \oplus P_i$ for all $i$ to the diagram and calculate the intersection of $N^-$ with $P_i$ using the height pairing and the equality $0 = \langle P_i, N^- \rangle = 2 - N^-.P_i - \sum_\nu contr_\nu (P_i,N^-)$ which follows from
 Remark \ref{heightpairingobservation}. By using translations, we obtain the remaining intersection numbers and the following graph, where we denote $P_i$ and $P_i \oplus N^-$ by $P_i^+$ and $P_i^-$ respectively:
\begin{equation*}
P_2.N^- = P_4.N^- = 2 - \left(\frac{6}{8} + \frac{2}{8} + 1\right) = 0; \quad
P_3.N^- = 2 - \left(\frac{4}{8} + \frac{4}{8}\right) = 1
\end{equation*}
\vspace{1mm}
\centerline{
\xy
(18,3)*{N^+};
(18,-13)*{N^-};
(17,42)*{P_3^+};
(18,55)*{P_3^-};
(40,7)*{P_2^+};
(40,24)*{P_2^-};
(40,54)*{P_4^+};
(40,64)*{P_4^-};
@={(0,0),(0,10),(0,20),(0,30),(10,0),(10,10),(10,20),(10,30),(20,0),(30,20),(30,30),(50,0),(50,10),(50,20),(50,30),(60,0),(60,10),(60,20),(60,30),(25,25),(35,25),(40,10),(20,40),(40,50),(20,-10),(40,20),(40,60),(20,50)}@@{*{\bullet}};
(0,0)*{};(0,30)*{}**\dir{-};
(10,0)*{};(10,30)*{}**\dir{-};
(0,0)*{};(10,0)*{}**\dir{-};
(10,30)*{};(0,30)*{}**\dir{-};
(10,0)*{};(20,0)*{}**\dir{-};
(30,0)*{};(20,0)*{}**\dir{-};
(25,25)*{};(30,30)*{}**\dir{-};
(25,25)*{};(30,20)*{}**\dir{-};
(35,25)*{};(30,30)*{}**\dir{-};
(35,25)*{};(30,20)*{}**\dir{-};
(20,0)*{};(25,25)*{}**\dir{-};
(50,0)*{};(60,0)*{}**\dir{-};
(50,30)*{};(50,0)*{}**\dir{-};
(60,30)*{};(60,0)*{}**\dir{-};
(50,30)*{};(60,30)*{}**\dir{-};
(50,0)*{};(30,0)*{}**\dir{-};
(40,10)*{};(50,20)*{}**\dir{-};
(40,10)*{};(10,20)*{}**\dir{-};
(40,10)*{};(35,25)*{}**\dir{-};
(20,40)*{};(0,30)*{}**\dir{-};
(20,40)*{};(60,30)*{}**\dir{-};
(20,40)*{};(25,25)*{}**\dir{-};
(40,50)*{};(35,25)*{}**\dir{-};
(40,50)*{};(0,10)*{}**\crv{(-20,40)};
(40,50)*{};(60,10)*{}**\crv{(80,40)};
(20,-10)*{};(10,10)*{}**\dir{-};
(20,-10)*{};(50,10)*{}**\dir{-};
(20,-10)*{};(35,25)*{}**\dir{-};
(40,20)*{};(25,25)*{}**\crv{(25,15)};
(40,20)*{};(50,30)*{}**\dir{-};
(40,20)*{};(10,30)*{}**\crv{(25,10)};
(40,60)*{};(60,0)*{}**\crv{(90,40)};
(40,60)*{};(0,0)*{}**\crv{(-30,40)};
(40,60)*{};(25,25)*{}**\dir{-};
(20,50)*{};(0,20)*{}**\crv{(-30,40)};
(20,50)*{};(60,20)*{}**\crv{(90,40)};
(20,50)*{};(35,25)*{}**\dir{-};
(20,0)*{};(20,50)*{}**\crv{(-30,-15) & (-30,30)};
(20,-10)*{};(20,40)*{}**\crv{(-40,-15) & (-30,40)};
(40,10)*{};(40,60)*{}**\crv{(55,30)};
(40,20)*{};(40,50)*{}**\crv{(50,30)};
\endxy}
\vspace{3mm}

This gives the following configuration on the quotient Enriques surface, where we denote the special bisection corresponding to $P_i$ again by $P_i$:

\vspace{-1mm}
\centerline{
\xy
@={(0,0),(0,10),(0,20),(0,30),(10,0),(10,10),(10,20),(10,30),(20,0),(30,0),(40,0),(20,30),(-10,30),(-10,0)}@@{*{\bullet}};
(0,0)*{};(0,30)*{}**\dir{-};
(10,0)*{};(10,30)*{}**\dir{-};
(0,0)*{};(10,0)*{}**\dir{-};
(10,30)*{};(0,30)*{}**\dir{-};
(10,0)*{};(20,0)*{}**\dir{-};
(10,10)*{};(20,0)*{}**\dir{-};
(30,0)*{};(20,0)*{}**\dir{-};
(30,0)*{};(40,0)*{}**\dir2{-};
(20,4)*{N};
(19,25)*{P_2};
(-10,26)*{P_3};
(-13,0)*{P_4};
(-10,0)*{};(0,0)*{}**\dir{-};
(-10,0)*{};(0,10)*{}**\dir{-};
(-10,30)*{};(0,20)*{}**\dir{-};
(-10,30)*{};(0,30)*{}**\dir{-};
(20,30)*{};(10,20)*{}**\dir{-};
(20,30)*{};(10,30)*{}**\dir{-};
(20,30)*{};(30,0)*{}**\dir{-};
(-10,30)*{};(30,0)*{}**\crv{(30,50)};
(-10,0)*{};(30,0)*{}**\crv{(20,-10)};
(-10,30)*{};(20,0)*{}**\crv{(50,60)};
(-10,0)*{};(20,30)*{}**\crv{(-30,60)};
\endxy}
\vspace{1mm}

In fact, we can produce six more $(-2)$-curves using different fibrations with a double $\I_3$ fiber to obtain the dual graph of type $\VII$. For example, one may look at the following subgraph:

\vspace{-6mm}
\centerline{
\xy
@={(0,10),(0,20),(0,30),(10,0),(10,10),(10,30),(20,0),(30,0),(40,0),(20,30),(-10,0)}@@{*{\bullet}};
(0,10)*{};(0,30)*{}**\dir{-};
(10,0)*{};(10,10)*{}**\dir{-};
(10,30)*{};(0,30)*{}**\dir{-};
(10,0)*{};(20,0)*{}**\dir{-};
(10,10)*{};(20,0)*{}**\dir{-};
(30,0)*{};(20,0)*{}**\dir{-};
(30,0)*{};(40,0)*{}**\dir2{-};
(32,4)*{N_1};
(-10,0)*{};(0,10)*{}**\dir{-};
(20,30)*{};(10,30)*{}**\dir{-};
(20,30)*{};(30,0)*{}**\dir{-};
(-10,0)*{};(30,0)*{}**\crv{(20,-10)};
(-10,0)*{};(20,30)*{}**\crv{(-30,60)};
\endxy}
\vspace{3mm}

By Proposition \ref{canonicaltype}, the $(-2)$-curve $N_1$ is a special bisection of a fibration with fibers $\I_6,\I_3$ (not $\IV$, since it is double) and another reducible fiber. By Lemma \ref{shiodatate}, the corresponding fibration is extremal and by Table \ref{extremalrational}, the last reducible fiber is of type $\I_2$ (resp. $\III$ in characteristic $3$) and it is simple, since $N_1$ meets its reduced components twice. Hence, we can add the missing component of the $\I_2$ (resp. $\III$) fiber to the graph. Similarly, one finds five more $(-2)$-curves and finally obtains the dual graph of type $\VII$.
The configuration we started with is what we will later call the "critical subgraph of type $\VII$", since we have shown that any Enriques surface containing this graph is of type $\VII$.

\begin{remark}
Note that the crucial point in all examples is the computation of the intersection numbers of the bisections using the height pairing. The intersection of the bisections obtained via $jac_2$ with the fibers is just a "translation" of the intersection of $N$ with the fibers. In particular, the process is much easier if $N^-$ is a $2$-torsion section, since the bisections arising via $jac_2$ are disjoint.
\end{remark}

\subsection{Vinberg's criterion and numerically trivial automorphisms}
In order to check that the $(-2)$-curves in the graphs for types $\I,\hdots,\VII$ are all $(-2)$-curves on the Enriques surface, one uses Vinberg's criterion.

\begin{proposition}{\rm(Vinberg \cite[Theorem 2.6]{Vinberg})}\label{vinberg}
Let $\Gamma$ be a dual graph of finitely many $(-2)$-curves on an Enriques surface $X$. Suppose that $\Gamma$ contains no $m$-tuple lines with $m \geq 3$ and suppose that the cone $K = \{ C \in \Num(X)_{\bbR} | C.E \geq 0 \text{ for all } E \in \Gamma\}$ is strictly convex. Then, the group $W_\Gamma$ generated by reflections along $(-2)$-curves in $\Gamma$ has finite index in $O(\Num(X))$ if and only if the fibration $\pi$ induced by every subgraph $F$ of $\Gamma$ of type $\tilde{A}$-$\tilde{D}$-$\tilde{E}$ is extremal and $\Gamma$ contains the dual graph of singular fibers of $\pi$. In this case, $\Gamma$ is the dual graph of all $(-2)$-curves on $X$.
\end{proposition}

\begin{remark}
This is a reformulation of the version of Vinberg's criterion presented by Kond\=o \cite[Theorem 1.9]{Kondo}. The last statement is due to Namikawa \cite[(6.9)]{Namikawa}. The strict convexity of $K$ can be achieved, for example, if $\Gamma$ contains the dual graph of singular fibers of an elliptic fibration $\pi$ and also contains another $(-2)$-curve which is not contained in a fiber of $\pi$.
\end{remark}

The following corollary is a straightforward application of Vinberg's criterion.

\begin{corollary}\label{allcurves}
Let $X$ be an Enriques surface whose dual graph of all $(-2)$-curves contains a graph $\Gamma$ which is one of the seven dual graphs in the Main Theorem. Then, the $(-2)$-curves in $\Gamma$ are all $(-2)$-curves on $X$.
\end{corollary}

Therefore, we can check the action of $\Aut(X)$ on $\Num(X)$ directly on the dual graph of $(-2)$-curves on $X$.

\begin{definition}
An automorphism of an Enriques surface $X$ is called \emph{numerically trivial} if it acts trivially on $\Num(X)$.
It is called \emph{cohomologically trivial} if it acts trivially on $\Pic(X)$.
We denote the respective groups by $\Aut_{nt}(X)$ and $\Aut_{ct}(X)$.
\end{definition}

Recall that $\Num(X)$ is a quotient of $\Pic(X)$, hence $\Aut_{ct}(X)$ is a normal subgroup of $\Aut_{nt}(X)$. Over the complex numbers a complete classification of such automorphisms is available (see \cite{MukaiNamikawa} and \cite{Mukai}). There are three types of Enriques surfaces $X$ with numerically trivial automorphisms and they satisfy $\Aut_{nt}(X) \in \{\bbZ/2\bbZ,\bbZ/4\bbZ\}$. In positive characteristics, however, we only have bounds on the size of these groups.

\begin{proposition}{\rm(Dolgachev \cite{DolgachevNum})} \label{numtriv}
Let $X$ be an Enriques surface. Then, 
\begin{equation*}
|\Aut_{ct}(X)| \leq 2 \quad \text{  and   } \quad |\Aut_{nt}(X)/\Aut_{ct}(X)| \leq 2.
\end{equation*}
\end{proposition}

However, we will not use this result, since we are interested in the precise shape of the automorphism group. Therefore, we give explicit arguments in every case. 

\newpage
\section{Enriques surfaces of type $\I$}\label{secI}

\subsection{Main theorem for type $\I$}

\begin{theorem}
Let $X$ be an Enriques surface. The following are equivalent:

\begin{enumerate}
\item $X$ is of type $\I$.
\item The dual graph of all $(-2)$-curves on $X$ contains the graph in Figure \ref{critI}.

\begin{figure}[h!]
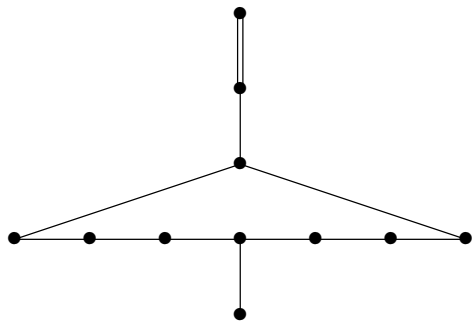

\centerline{
\xy
@={(0,10),(10,10),(20,10),(30,10),(40,10),(50,10),(60,10),(30,0),(30,20),(30,30),(30,40)}@@{*{\bullet}};
(0,10)*{};(60,10)*{}**\dir{-};
(30,10)*{};(30,0)*{}**\dir{-};
(0,10)*{};(30,20)*{}**\dir{-};
(60,10)*{};(30,20)*{}**\dir{-};
(30,20)*{};(30,30)*{}**\dir{-};
(30,40)*{};(30,30)*{}**\dir2{-};
\endxy
}
\caption{Critical subgraph for type $\I$}
\label{critI}
\end{figure}
\item The canonical cover $\tilde{X}$ of $X$ admits an elliptic fibration with a Weierstrass equation of the form
\begin{equation*}
y^2 + \beta (s^2+s) xy = x^3 + \beta^3 (s^2+s)^3 x
\end{equation*}
such that the covering morphism $\rho: \tilde{X} \to X$ is given as quotient by the involution
$
\sigma = t_{N^-} \circ J(\sigma),
$
where $J(\sigma): s \mapsto -s-1$ and $t_{N^-}$ is translation by $N^- = (0,0)$.
\end{enumerate}
\end{theorem}

\begin{proof}
First, observe that the dual graph of type $\I$ (see Table \ref{main}) contains the graph in Figure \ref{critI}.

This subgraph can be interpreted as the dual graph of a special elliptic fibration $\pi$ with singular fibers $\III^*$ and $\I_2$ (not $\III$, since this fiber is a double fiber) and special bisection $N$ as follows, where the dotted rectangles mark the fibers:

\vspace{5mm}
\centerline{
\xy
@={(0,10),(10,10),(20,10),(30,10),(40,10),(50,10),(60,10),(30,0),(30,20),(30,30),(30,40)}@@{*{\bullet}};
(0,10)*{};(60,10)*{}**\dir{-};
(30,10)*{};(30,0)*{}**\dir{-};
(0,10)*{};(30,20)*{}**\dir{-};
(60,10)*{};(30,20)*{}**\dir{-};
(30,20)*{};(30,30)*{}**\dir{-};
(30,40)*{};(30,30)*{}**\dir2{-};
(-3,13)*{};(63,13)*{}**\dir{--};
(-3,-3)*{};(63,-3)*{}**\dir{--};
(63,-3)*{};(63,13)*{}**\dir{--};
(-3,-3)*{};(-3,13)*{}**\dir{--};
(27,43)*{};(33,43)*{}**\dir{--};
(33,27)*{};(33,43)*{}**\dir{--};
(33,27)*{};(27,27)*{}**\dir{--};
(27,27)*{};(27,43)*{}**\dir{--};
(33,23)*{N};
\endxy
}
\vspace{5mm}

As explained in Lemma \ref{universal}, $N$ splits into two sections $N^+$ and $N^-$ of the elliptic fibration $\tilde{\pi}$ induced by $\pi$ on the K3 cover $\tilde{X}$. Fixing $N^+$ as the zero section, we can compute $h(N^-) = 0$ and we see that $N^-$ is a $2$-torsion section of $\tilde{\pi}$. 
Starting from the subgraph in Figure \ref{critI}, we get the last missing $(-2)$-curve from the elliptic fibration with a double fiber of type $\I_8$, which is induced by the $\tilde{A}_7$ diagram, as follows: The fibration is extremal by Lemma \ref{shiodatate}, the second reducible fiber is of type $\I_2$ (resp. $\III$ in characteristic $2$) by Table \ref{extremalrational} and the intersection behaviour can be determined from the dual graph. These are all $(-2)$-curves on $X$ by Corollary \ref{allcurves}.

Now, we pursue the converse process dictated by Proposition \ref{Kondo} and calculate all elliptic fibrations of K3 surfaces obtained as separable quadratic base changes of $J(\pi)$ together with a section having the same intersection behaviour as $N^-$ with curves obtained from $(-2)$-curves on $X$.

By \cite{Lang3} we have the following equation for the unique rational elliptic surface with singular fibers of type $\III^*$ and $\I_2$
\begin{equation*}
y^2 + txy = x^3 + t^3x,
\end{equation*}
where $t$ is a coordinate on $\bbP^1$. The $\I_2$ fiber is at $t = \infty$, while the $\III^*$ fiber is at $t = 0$. Moreover, if $\Char(k) \neq 2$, there is an $\I_1$ fiber at $t = 64$ and all other fibers are smooth. The non-trivial $2$-torsion section is $s = (0,0)$.

In every characteristic, we can write a degree $2$ morphism $\bbP^1 \to \bbP^1$ with $t = \infty$ as branch point and which is not branched over $t = 0$ in the form
\begin{equation*}
t \mapsto \beta(s^2 +s),
\end{equation*}
where $s$ is the new parameter on $\bbP^1$ and $\beta \in k - \{0\}$. We are allowed to assume that $t = 0$ is not a branch point, since the $\III^*$ fiber is not multiple. The covering involution is given by $J(\sigma): s \mapsto -s - 1$. The second branch point of this degree $2$ cover in characteristic different from $2$ is at $t = -\frac{\beta}{4}$, which corresponds to $s = -\frac{1}{2}$. Now, we get the equation
\begin{equation*}
y^2 + \beta (s^2+s) xy = x^3 + \beta^3 (s^2+s)^3 x
\end{equation*}
together with the $2$-torsion section $s' = (0,0)$ obtained by pulling back $s$. This equation defines an elliptic fibration $\tilde{\pi}$ on a K3 surface. As explained in Section \ref{basechange}, if $\tilde{\pi}$ is obtained as base change of a fibration of an Enriques surface, then $s' = N^-$ and $\sigma$ is the covering involution.
\end{proof}

\begin{remark}
Note that we have not yet claimed existence of Enriques surfaces of type $\I$. However, we have reduced this problem to the question whether $N^-$ is a $J(\pi)$-Enriques section or not. We answer this question in the subsection on degenerations and moduli.
\end{remark}

\newpage
\subsection{Automorphisms}

\begin{proposition}\label{Aut1}
Let $X$ be an Enriques surface of type $\I$. Then, $Aut(X) \cong D_4$ and this group is generated by automorphisms induced by $2$-torsion sections of the Jacobian fibrations of elliptic fibrations of $X$. Moreover, $\Aut_{nt}(X) \cong \bbZ/2\bbZ$ and $\Aut(X)/\Aut_{nt}(X) = (\bbZ/2\bbZ)^2$.
\end{proposition}

\begin{proof}
Recall that the dual graph of type $\I$ is as follows:

\centerline{
\includegraphics[width=55mm]{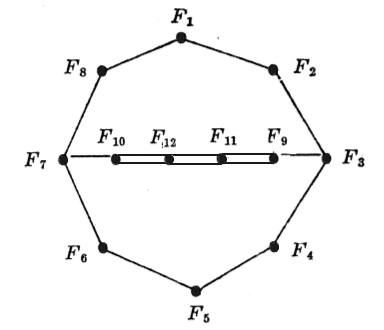}
}

As has already been explained by Kond\=o \cite[p.205]{Kondo} and Dolgachev \cite[p.175]{Dolgachev}, the symmetry group of the dual graph of $(-2)$-curves is $(\bbZ/2\bbZ)^2$ and the $2$-torsion section of the fibration $\pi$ induced by the linear system $|2(F_9+F_{11})|$ acts as a reflection along the horizontal axis, while the $2$-torsion section of the fibration induced by $|F_{11}+F_{12}|$ acts trivially on the graph. A non-trivial numerically trivial automorphism $g$ fixes $F_3$ and $F_7$ pointwise, hence $g$ fixes the fibration $\pi$ and at least one geometric point on the generic fiber of $\pi$. Since $\pi$ is non-isotrivial, $g$ is the unique hyperelliptic involution of the generic fiber of $\pi$ fixing the geometric points defined by $F_3$ and $F_7$.
%
Since $\Aut(X)$ contains a translation by a $4$-torsion section of the Jacobian of $|F_{11}+F_{12}|$, it suffices to observe that the $2$-torsion section of a fibration with $\I_4^*$ fiber acts as a reflection along the vertical axis to show that $\Aut(X) \cong D_4$. This follows from Corollary \ref{jac2}.
\end{proof}

\subsection{Degenerations and Moduli}

\begin{proposition}
Let $\beta \neq 0$ and
\begin{equation*}
y^2 + \beta (s^2+s) xy = x^3 + \beta^3 (s^2+s)^3 x
\end{equation*}
be the Weierstrass equation of an elliptic fibration $\tilde{\pi}_{\beta}$ with section on a K3 surface $\tilde{X}$. Define the involution $\sigma = t_{N^-} \circ J(\sigma)$, where $J(\sigma): s \mapsto -s-1$ and $t_{N^-}$ is translation by the section $N^- = (0,0)$. Then, the following statements are true:
\begin{enumerate}
\item $\sigma$ is fixed point free if and only if $\beta \neq -256$. If $\beta = -256$, then the fixed locus of $\sigma$ is one $(-2)$-curve.
\item Two fibrations $\tilde{\pi}_{\beta}$ and $\tilde{\pi}_{\beta'}$ are isomorphic up to automorphisms of $\bbP^1$ if and only if $\beta = \beta'$.
\end{enumerate}
\end{proposition}

\begin{proof}
For the first claim, by Lemma \ref{Kondo}, we have to check whether $N^-$ is a $J(\pi)$-Enriques section. First, observe that $N^-.N^+ = 0$, $J(\sigma)(N^-) = N^- = \ominus N^-$ and $N^-$ does not meet the $\I_4$ fiber in the identity component. Therefore, we are done if the second fiber fixed by $J(\sigma)$, namely $F_{-\frac{1}{2}}$, is smooth. This happens if and only if $\beta \neq -256$ and otherwise $F_{-\frac{1}{2}}$ is an $\I_2$ fiber. In the latter case, $N^-$ does not meet the singular point $(-2^9,2^{14})$ of the Weierstrass equation at $s = -\frac{1}{2}$ and therefore it meets the identity component of $F_{-\frac{1}{2}}$. Hence, $N^-$ is not a $J(\pi)$-Enriques section in this case and $\sigma$ is not fixed point free by Proposition \ref{Kondo}.

%

The second claim follows immediately from a comparison of $j$-invariants, since in any characteristic and independently of $\beta$, the locations of the $\III^*$ and $\I_4$ fibers are at $s = -1,0, \infty$.
\end{proof}

We have seen in the previous subsection that the two elliptic fibrations with singular fiber $\III^*$ on an Enriques surface of type $\I$ are isomorphic. Therefore, we can describe the moduli space of these Enriques surfaces using the previous proposition.

\begin{corollary}
Enriques surfaces of type $\I$ are parametrized by $\bbA^1 - \{0,-256\}$ in every characteristic.
\end{corollary}

While $\beta \in \{0,\infty\}$ leads to very degenerate surfaces, we still get an involution if $\beta = -256$, while the K3 surface acquires an additional rational double point. The minimal resolution of the quotient is a Coble surface (see also Remark \ref{coble}).
%
%
\section{Enriques surfaces of type $\II$}\label{secII}

\subsection{Main theorem for type $\II$}

\begin{theorem}
Let $X$ be an Enriques surface. The following are equivalent:

\begin{enumerate}
\item $X$ is of type $\II$.
\item The dual graph of all $(-2)$-curves on $X$ contains the graph in Figure \ref{critII}.

\begin{figure}[h!]
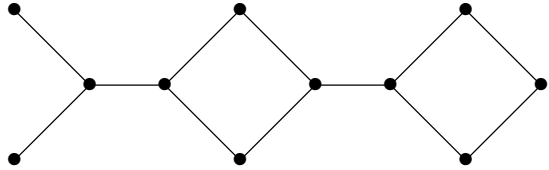

\centerline{
\xy
@={(0,20),(0,0),(10,10),(20,10),(30,20),(30,0),(40,10),(50,10),(60,20),(60,0),(70,10)}@@{*{\bullet}};
(0,20)*{};(10,10)*{}**\dir{-};
(0,0)*{};(10,10)*{}**\dir{-};
(20,10)*{};(10,10)*{}**\dir{-};
(20,10)*{};(30,20)*{}**\dir{-};
(20,10)*{};(30,0)*{}**\dir{-};
(40,10)*{};(30,20)*{}**\dir{-};
(40,10)*{};(30,0)*{}**\dir{-};
(40,10)*{};(50,10)*{}**\dir{-};
(60,20)*{};(50,10)*{}**\dir{-};
(60,0)*{};(50,10)*{}**\dir{-};
(60,20)*{};(70,10)*{}**\dir{-};
(60,0)*{};(70,10)*{}**\dir{-};
\endxy
}
\caption{Critical subgraph for type $\II$}
\label{critII}
\end{figure}
\item The canonical cover $\tilde{X}$ of $X$ admits an elliptic fibration with a Weierstrass equation of the form
\begin{equation*}
y^2 + \beta (s^2+s) xy + \beta^2 (s^2+s)^2 y  = x^3 + \beta (s^2+s) x^2
\end{equation*}
such that the covering morphism $\rho: \tilde{X} \to X$ is given as quotient by the involution
$
\sigma = t_{N^-} \circ J(\sigma),
$
where $J(\sigma): s \mapsto -s-1$ and $t_{N^-}$ is translation by $N^- = (0,0)$.
\end{enumerate}
\end{theorem}

\begin{proof}
First, observe that the dual graph of type $\II$ (see Table \ref{main}) contains the graph in Figure \ref{critII}. 

This subgraph can be interpreted as the dual graph of a special elliptic fibration $\pi$ with singular fibers $\I_1^*,\I_4$ and special bisection $N$ as follows, where the dotted rectangles mark the fibers:

\vspace{5mm}
\centerline{
\xy
@={(0,20),(0,0),(10,10),(20,10),(30,20),(30,0),(40,10),(50,10),(60,20),(60,0),(70,10)}@@{*{\bullet}};
(0,20)*{};(10,10)*{}**\dir{-};
(0,0)*{};(10,10)*{}**\dir{-};
(20,10)*{};(10,10)*{}**\dir{-};
(20,10)*{};(30,20)*{}**\dir{-};
(20,10)*{};(30,0)*{}**\dir{-};
(40,10)*{};(30,20)*{}**\dir{-};
(40,10)*{};(30,0)*{}**\dir{-};
(40,10)*{};(50,10)*{}**\dir{-};
(60,20)*{};(50,10)*{}**\dir{-};
(60,0)*{};(50,10)*{}**\dir{-};
(60,20)*{};(70,10)*{}**\dir{-};
(60,0)*{};(70,10)*{}**\dir{-};
(-3,23)*{};(33,23)*{}**\dir{--};
(33,-3)*{};(33,23)*{}**\dir{--};
(33,-3)*{};(-3,-3)*{}**\dir{--};
(-3,23)*{};(-3,-3)*{}**\dir{--};
(47,23)*{};(73,23)*{}**\dir{--};
(73,-3)*{};(73,23)*{}**\dir{--};
(73,-3)*{};(47,-3)*{}**\dir{--};
(47,23)*{};(47,-3)*{}**\dir{--};
(43,13)*{N};
\endxy
}
\vspace{5mm}

Note that the $\I_4$ fiber is a double fiber. Similarly to the case of type $\II$, we compute $h(N^-) = 0$ and find the last missing $(-2)$-curves via $jac_2$.

We found the following equation for the unique rational elliptic surface with singular fibers of type $\I_1^*$ and $\I_4$ in arbitrary characteristic
\begin{equation*}
y^2 + txy + t^2y = x^3 + tx^2,
\end{equation*}
where $t$ is a coordinate on $\bbP^1$. The $\I_4$ fiber is at $t = \infty$, while the $\I_1^*$ fiber is at $t = 0$. Moreover, if $\Char(k) \neq 2$, then there is an $\I_1$ fiber at $t = 16$ and all other fibers are smooth. The non-trivial $2$-torsion section is $s = (0,0)$.

In every characteristic, we can write every degree $2$ morphism $\bbP^1 \to \bbP^1$ with $t = \infty$ as branch point that is not branched over $t = 0$ in the following form
\begin{equation*}
t \mapsto \beta(s^2 +s),
\end{equation*}
where $s$ is the new parameter on $\bbP^1$ and $\beta \in k - \{0\}$. The covering involution is given by $s \mapsto -s - 1$. The second branch point of this degree $2$ cover in characteristic different from $2$ is at $t = -\frac{\beta}{4}$. Now, we get the equation
\begin{equation*}
y^2 + \beta (s^2+s) xy +\beta^2 (s^2+s)^2 y = x^3 + \beta (s^2+s) x^2
\end{equation*}
together with the $2$-torsion section $s' = (0,0)$ obtained by pulling back $s$. This equation defines an elliptic fibration $\tilde{\pi}$ on a K3 surface. As explained in Section \ref{basechange}, if $\tilde{\pi}$ is obtained as base change of a fibration of an Enriques surface, then $s' = N^-$ and $\sigma$ is the covering involution.
\end{proof}

\newpage
\subsection{Automorphisms}

\begin{proposition}\label{Aut2}
Let $X$ be an Enriques surface of type $\II$. Then, $Aut(X) \cong \mathfrak{S}_4$ and this group is generated by automorphisms induced by $2$-torsion sections of the Jacobian fibrations of elliptic fibrations of $X$. Moreover, $\Aut_{nt}(X) \cong \{1\}$.
\end{proposition}

\begin{proof}
Kond\=o's proof works in arbitrary characteristic \cite[p.208]{Kondo} once we show that the surface has no numerically trivial automorphisms. Recall that the dual graph of $(-2)$-curves for type $\II$ is as follows:

\centerline{
\includegraphics[width=60mm]{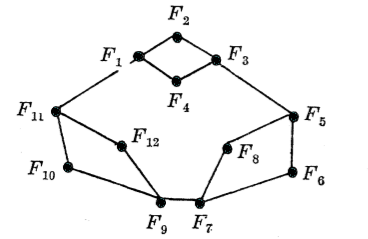}
}

A numerically trivial automorphism $g$ fixes the two bisections $F_1$ and $F_7$ of the non-isotrivial fibration $\pi$ induced by the linear system $|2(F_9+F_{10} + F_{11} + F_{12})|$ pointwise. Both $F_1$ and $F_7$ are separable (i.e. the projection to the base curve is separable) bisections of $\pi$, since they meet distinct points on the $\I_1^*$ fiber, hence $g$ fixes at least four geometric points on the generic fiber of $\pi$. If $\Char(k) = 2$, then $g$ is trivial. If $\Char(k) \neq 2$, then we may assume that $g$ is non-trivial. Then, $g$ is a hyperelliptic involution of $\pi$ and the four geometric points on the generic fiber are $2$-torsion points relative to each other. But in characteristic different from $2$, $\pi$ has an $\I_1$ fiber which has only two $2$-torsion points. Therefore, $F_1$ and $F_7$ would have to meet, but they do not. Hence, $g$ is trivial.


\end{proof}

\subsection{Degenerations and Moduli}
As in the case of type $\I$, one proves the following.

\begin{proposition}
Let $\beta \neq 0$ and
\begin{equation*}
y^2 + \beta (s^2+s) xy +\beta^2 (s^2+s)^2 y = x^3 + \beta (s^2+s) x^2
\end{equation*}
be the Weierstrass equation of an elliptic fibration $\tilde{\pi}_{\beta}$ with section on a K3 surface $\tilde{X}$. Define the involution $\sigma = t_{N^-} \circ J(\sigma)$, where $J(\sigma): s \mapsto -s-1$ and $t_{N^-}$ is translation by the section $N^- = (0,0)$. Then, the following statements are true:
\begin{enumerate}
\item $\sigma$ is fixed point free if and only if $\beta \neq -64$. If $\beta = -64$, the fixed locus of $\sigma$ is one $(-2)$-curve.
\item Two fibrations $\tilde{\pi}_{\beta}$ and $\tilde{\pi}_{\beta'}$ are isomorphic up to automorphisms of $\bbP^1$ if and only if $\beta = \beta'$.
\end{enumerate}
\end{proposition}

%

\begin{corollary}
Enriques surfaces of type $\II$ are parametrized by $\bbA^1 - \{0,-64\}$ in every characteristic.
\end{corollary}

As in the case of type $\I$, the cases where $\beta \in \{0,\infty\}$ are very degenerate surfaces and $\beta = -64$ leads to a Coble surface.
\section{Enriques surfaces of type $\III$}\label{secIII}

\subsection{Main theorem for type $\III$}

\begin{theorem}
Let $X$ be an Enriques surface. The following are equivalent:

\begin{enumerate}
\item $X$ is of type $\III$.
\item The dual graph of all $(-2)$-curves on $X$ contains the graph in Figure \ref{critIII}.
\begin{figure}[h!]
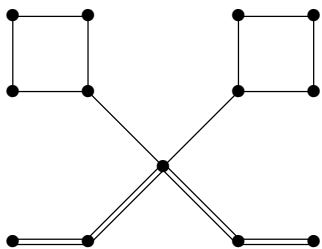

\centerline{
\xy
@={(0,40),(10,40),(30,40),(40,40),(0,30),(10,30),(30,30),(40,30),(20,20),(0,10),(10,10),(30,10),(40,10)}@@{*{\bullet}};
(0,40)*{};(10,40)*{}**\dir{-};
(30,40)*{};(40,40)*{}**\dir{-};
(0,30)*{};(10,30)*{}**\dir{-};
(30,30)*{};(40,30)*{}**\dir{-};
(0,40)*{};(0,30)*{}**\dir{-};
(10,30)*{};(10,40)*{}**\dir{-};
(30,30)*{};(30,40)*{}**\dir{-};
(40,30)*{};(40,40)*{}**\dir{-};
(10,30)*{};(20,20)*{}**\dir{-};
(30,30)*{};(20,20)*{}**\dir{-};
(10,10)*{};(20,20)*{}**\dir2{-};
(30,10)*{};(20,20)*{}**\dir2{-};
(10,10)*{};(0,10)*{}**\dir2{-};
(30,10)*{};(40,10)*{}**\dir2{-};
\endxy
}
\caption{Critical subgraph for type $\III$}
\label{critIII}
\end{figure}
\item The canonical cover $\tilde{X}$ of $X$ admits an elliptic fibration with a Weierstrass equation of the form
\begin{equation*}
y^2 +  xy  = x^3 + 4 s^4 x^2 + s^4 x 
\end{equation*}
such that the covering morphism $\rho: \tilde{X} \to X$ is given as quotient by the involution
$\sigma = t_{N^-} \circ J(\sigma)$,
where $J(\sigma): s \mapsto -s$ and $t_{N^-}$ is translation by $N^- = (0,0)$.
\end{enumerate}
Moreover, Enriques surfaces of type $\III$ do not exist in characteristic $2$.
\end{theorem}
\begin{proof}
Note that the dual graph of type $\III$ (see Table \ref{main}) contains the graph in Figure \ref{critIII}. 

The subgraph in Figure \ref{critIII} can be interpreted as the dual graph of a special elliptic fibration $\pi$ with singular fibers $(\I_4,\I_4,\I_2,\I_2)$ and special bisection $N$ as follows, where the dotted rectangles mark the fibers:

\vspace{2mm}

\centerline{
\xy
@={(0,40),(10,40),(30,40),(40,40),(0,30),(10,30),(30,30),(40,30),(20,20),(0,10),(10,10),(30,10),(40,10)}@@{*{\bullet}};
(0,40)*{};(10,40)*{}**\dir{-};
(30,40)*{};(40,40)*{}**\dir{-};
(0,30)*{};(10,30)*{}**\dir{-};
(30,30)*{};(40,30)*{}**\dir{-};
(0,40)*{};(0,30)*{}**\dir{-};
(10,30)*{};(10,40)*{}**\dir{-};
(30,30)*{};(30,40)*{}**\dir{-};
(40,30)*{};(40,40)*{}**\dir{-};
(10,30)*{};(20,20)*{}**\dir{-};
(30,30)*{};(20,20)*{}**\dir{-};
(10,10)*{};(20,20)*{}**\dir2{-};
(30,10)*{};(20,20)*{}**\dir2{-};
(10,10)*{};(0,10)*{}**\dir2{-};
(30,10)*{};(40,10)*{}**\dir2{-};
(-3,43)*{};(13,43)*{}**\dir{--};
(-3,27)*{};(13,27)*{}**\dir{--};
(13,27)*{};(13,43)*{}**\dir{--};
(-3,27)*{};(-3,43)*{}**\dir{--};
(27,43)*{};(43,43)*{}**\dir{--};
(27,27)*{};(43,27)*{}**\dir{--};
(43,27)*{};(43,43)*{}**\dir{--};
(27,27)*{};(27,43)*{}**\dir{--};
(-3,13)*{};(13,13)*{}**\dir{--};
(-3,7)*{};(13,7)*{}**\dir{--};
(13,7)*{};(13,13)*{}**\dir{--};
(-3,7)*{};(-3,13)*{}**\dir{--};
(27,13)*{};(43,13)*{}**\dir{--};
(27,7)*{};(43,7)*{}**\dir{--};
(43,7)*{};(43,13)*{}**\dir{--};
(27,7)*{};(27,13)*{}**\dir{--};
(20,24)*{N};
\endxy
}
\vspace{5mm}

As before, the bisection $N$ splits into two sections $N^+$ and $N^-$ of the elliptic fibration $\tilde{\pi}$ induced by $\pi$ on the K3 cover $\tilde{X}$. Fixing $N^+$ as the zero section, we compute $h(N^-) = 0$ and we see that $N^-$ is a $2$-torsion section of $\tilde{\pi}$ meeting the $\I_8$ fibers in a non-identity component. 

Note that the existence of this fibration already implies non-existence of this type of Enriques surfaces in characteristic $2$, since a fibration with singular fibers $(\I_4,\I_4,\I_2,\I_2)$ does not exist on rational surfaces in characteristic $2$, as can be seen in Table \ref{extremalrational}.

Now, Corollary \ref{jac2} gives three more $(-2)$-curves resulting in the following graph:

\vspace{-1mm}
\centerline{
\xy
@={(0,40),(10,40),(30,40),(40,40),(0,30),(10,30),(30,30),(40,30),(20,20),(0,10),(10,10),(30,10),(40,10),(0,20),(40,20),(20,0)}@@{*{\bullet}};
(0,40)*{};(10,40)*{}**\dir{-};
(30,40)*{};(40,40)*{}**\dir{-};
(0,30)*{};(10,30)*{}**\dir{-};
(30,30)*{};(40,30)*{}**\dir{-};
(0,40)*{};(0,30)*{}**\dir{-};
(10,30)*{};(10,40)*{}**\dir{-};
(30,30)*{};(30,40)*{}**\dir{-};
(40,30)*{};(40,40)*{}**\dir{-};
(10,30)*{};(20,20)*{}**\dir{-};
(30,30)*{};(20,20)*{}**\dir{-};
(10,10)*{};(20,20)*{}**\dir2{-};
(30,10)*{};(20,20)*{}**\dir2{-};
(10,10)*{};(0,10)*{}**\dir2{-};
(30,10)*{};(40,10)*{}**\dir2{-};
(0,20)*{};(0,10)*{}**\dir2{-};
(40,20)*{};(40,10)*{}**\dir2{-};
(0,20)*{};(30,10)*{}**\dir2{-};
(40,20)*{};(10,10)*{}**\dir2{-};
(0,20)*{};(0,40)*{}**\crv{(-5,30)};
(40,20)*{};(40,40)*{}**\crv{(45,30)};
(0,20)*{};(40,40)*{}**\crv{(-25,60)};
(40,20)*{};(00,40)*{}**\crv{(65,60)};
(20,0)*{};(0,10)*{}**\dir2{-};
(20,0)*{};(40,10)*{}**\dir2{-};
(20,0)*{};(30,30)*{}**\dir{-};
(20,0)*{};(10,30)*{}**\dir{-};
\endxy
}
\vspace{1mm}

We find a graph of an elliptic fibration with singular fibers $(\I_0^*,\I_0^*)$ and special bisection $N$: 

\vspace{-2mm}
\centerline{
\xy
(3,12)*{N};
@={(0,40),(10,40),(30,40),(0,30),(30,30),(40,30),(20,20),(0,10),(0,20),(40,20),(20,0)}@@{*{\bullet}};
(0,40)*{};(10,40)*{}**\dir{-};
(30,30)*{};(40,30)*{}**\dir{-};
(0,40)*{};(0,30)*{}**\dir{-};
(30,30)*{};(30,40)*{}**\dir{-};
(30,30)*{};(20,20)*{}**\dir{-};
(0,20)*{};(0,10)*{}**\dir2{-};
(0,20)*{};(0,40)*{}**\crv{(-5,30)};
(40,20)*{};(00,40)*{}**\crv{(65,60)};
(20,0)*{};(0,10)*{}**\dir2{-};
(20,0)*{};(30,30)*{}**\dir{-};
\endxy
}
\vspace{3mm}
%
%

With the usual notation, we compute $h(N^-) = 2$ and add two bisections coming from Corollary \ref{jac2}. In the following figure, we only added one of these bisections to maintain readability:

\vspace{-3mm}
\centerline{
\xy
@={(0,40),(10,40),(30,40),(40,40),(0,30),(10,30),(30,30),(40,30),(20,20),(0,10),(10,10),(30,10),(40,10),(0,20),(40,20),(20,0),(20,40)}@@{*{\bullet}};
(0,40)*{};(10,40)*{}**\dir{-};
(30,40)*{};(40,40)*{}**\dir{-};
(0,30)*{};(10,30)*{}**\dir{-};
(30,30)*{};(40,30)*{}**\dir{-};
(0,40)*{};(0,30)*{}**\dir{-};
(10,30)*{};(10,40)*{}**\dir{-};
(30,30)*{};(30,40)*{}**\dir{-};
(40,30)*{};(40,40)*{}**\dir{-};
(10,30)*{};(20,20)*{}**\dir{-};
(30,30)*{};(20,20)*{}**\dir{-};
(10,10)*{};(20,20)*{}**\dir2{-};
(30,10)*{};(20,20)*{}**\dir2{-};
(10,10)*{};(0,10)*{}**\dir2{-};
(30,10)*{};(40,10)*{}**\dir2{-};
(0,20)*{};(0,10)*{}**\dir2{-};
(40,20)*{};(40,10)*{}**\dir2{-};
(0,20)*{};(30,10)*{}**\dir2{-};
(40,20)*{};(10,10)*{}**\dir2{-};
(0,20)*{};(0,40)*{}**\crv{(-5,30)};
(40,20)*{};(40,40)*{}**\crv{(45,30)};
(0,20)*{};(40,40)*{}**\crv{(-25,60)};
(40,20)*{};(00,40)*{}**\crv{(65,60)};
(20,0)*{};(0,10)*{}**\dir2{-};
(20,0)*{};(40,10)*{}**\dir2{-};
(20,0)*{};(30,30)*{}**\dir{-};
(20,0)*{};(10,30)*{}**\dir{-};
(20,40)*{};(10,40)*{}**\dir2{-};
(20,40)*{};(30,40)*{}**\dir2{-};
(0,10)*{};(20,40)*{}**\dir2{-};
(40,10)*{};(20,40)*{}**\dir2{-};
(30,10)*{};(20,40)*{}**\dir2{-};
(10,10)*{};(20,40)*{}**\dir2{-};
\endxy
}
\vspace{3mm}

Note that one of the bisections arising via $jac_2$ has already been part of the graph to begin with. Hence, it remains to produce two more $(-2$)-curves using another fibration. We leave the details to the reader.


By \cite{Lang3}, we have the following equation for the unique rational elliptic surface with singular fibers of type $(\I_4,\I_4,\I_2,\I_2)$ in characteristic different from $2$ (the equation can be simplified over $\bbZ$)
\begin{equation*}
y^2 + xy = x^3 + 4t^2x^2 + t^2x,
\end{equation*}
where $t$ is a coordinate on $\bbP^1$. The $\I_4$ fibers are at $t = 0, \infty$, while the $\I_2$ fibers are at $t = \pm \frac{1}{4}$. The non-trivial $2$-torsion sections are $s_1 = (-4t^2,2t^2)$, $s_2 = (0,0)$ and $s_3 = (-\frac{1}{4},\frac{1}{8})$.

In characteristic different from $2$, we can write a degree $2$ morphism $\bbP^1 \to \bbP^1$ with $t = 0,\infty$ as branch points in the following form
\begin{equation*}
t \mapsto s^2,
\end{equation*}
where $s$ is the new parameter on $\bbP^1$. The covering involution $J(\sigma)$ is given by $s \mapsto -s$. Now, we get the equation
\begin{equation}\label{eqnIII}
y^2 +  xy  = x^3 + 4 s^4 x^2 + s^4 x 
\end{equation}
together with the $2$-torsion sections $s_1' = (-4s^4,2s^4)$, $s_2' = (0,0)$ and $s_3' = (-\frac{1}{4},\frac{1}{8})$ obtained by pulling back $s_1$,$s_2$ and $s_3$. All of them are $J(\sigma)$-(anti-)invariant. However, $s_1'$ (resp. $s_3'$) meets the identity component of the fiber at $s = \infty$ (resp. $s = 0$). Therefore, $s_2'$ is the section we are looking for.

\end{proof}

\begin{remark}\label{extraaut3}
Note that Equation (\ref{eqnIII}) has an automorphism $\iota: s \mapsto \sqrt{-1}s$ which commutes with $\sigma$. Therefore, $\iota$ induces an automorphism of the Enriques surface, which we will also denote by $\iota$. Moreover, $\iota$ fixes the $2$-torsion sections of (\ref{eqnIII}). Note also that this automorphism acts as $\sqrt{-1}$ on a non-zero global $2$-form of the K3 surface.
\end{remark}
\subsection{Automorphisms}

\begin{proposition}\label{Aut3}
Let $X$ be an Enriques surface of type $\III$. Then, $Aut(X) \cong (\bbZ/4\bbZ \times (\bbZ/2\bbZ)^2) \rtimes D_4$ and this group is generated by automorphisms induced by $2$-torsion sections of the Jacobian fibrations of non-isotrivial elliptic fibrations of $X$ and the automorphism exhibited in Remark \ref{extraaut3}. Moreover, $\Aut_{nt} \cong \bbZ/2\bbZ$ and $\Aut(X)/\Aut_{nt}(X) = (\bbZ/2\bbZ)^3 \rtimes D_4$.
\end{proposition}

\begin{proof}
Recall that the dual graph of $(-2)$-curves for type $\III$ is as follows:

\vspace{2mm}
\centerline{
\includegraphics[width=100mm]{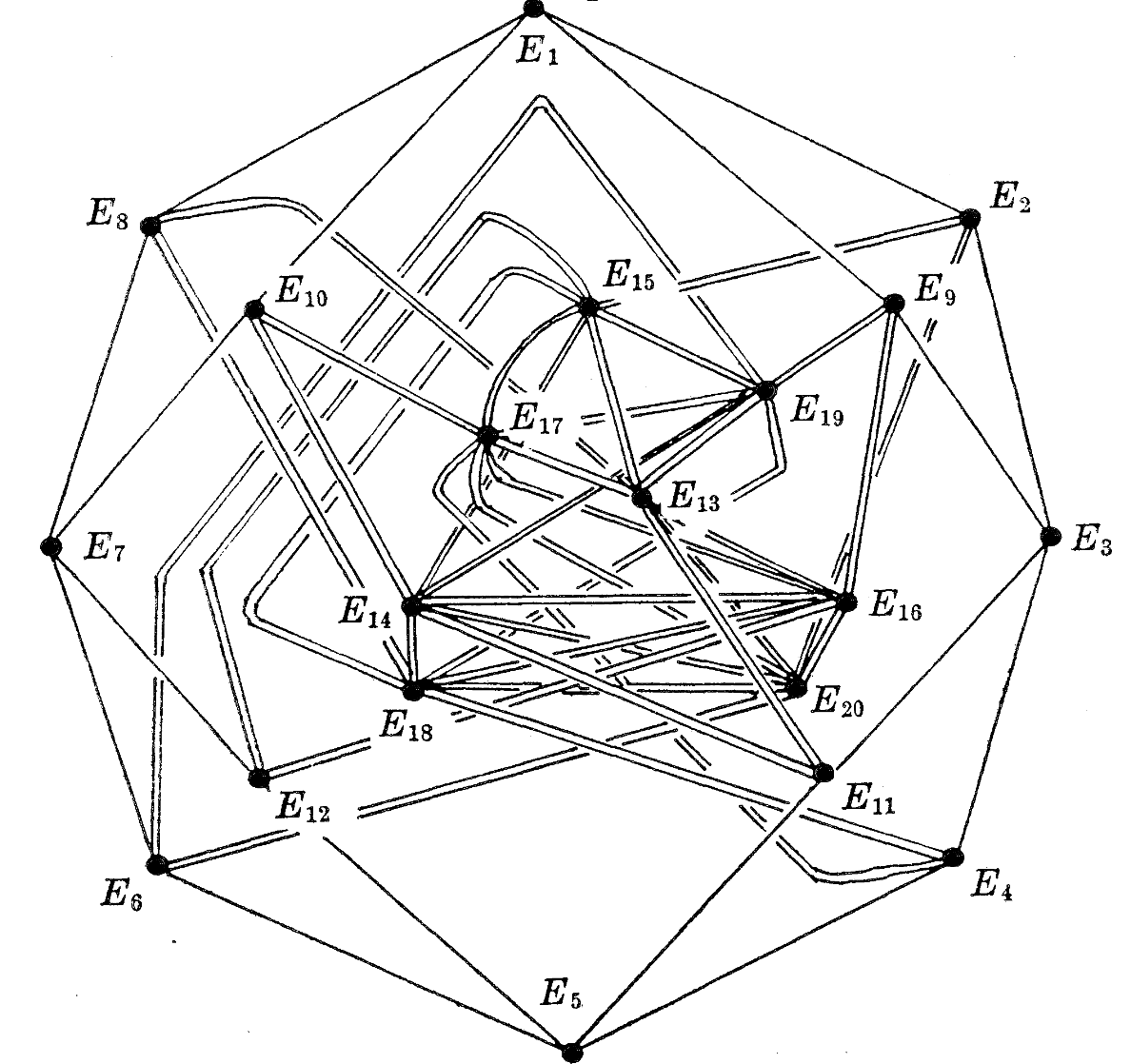}
}
\vspace{2mm}

Let us first show that $|\Aut_{nt}(X)| \leq 2$. Consider the elliptic fibration $\pi$ induced by the linear system $|2(E_3+E_4+E_5+E_{11})|$ and let $g \in \Aut_{nt}(X)$ be a non-trivial automorphism. If $g$ fixes one of the bisections $E_2,E_9,E_6$ and $E_{12}$ pointwise, then it is the hyperelliptic involution of the generic fiber of $\pi$ fixing the geometric points defined by the bisection. Moreover, $g$ induces a unique involution on such a bisection if it acts non-trivially on it. In any case, $ord(g) = 2^n$ for some $n \in \bbN$ and, since $\Char(k) \neq 2$, $g$ is tame. The fixed locus of a tame automorphism is smooth by the Lefschetz fixed point formula \cite{Iversen}. Since $g$ fixes $E_1,E_3,E_5$ and $E_7$ pointwise, it has to act non-trivially on $E_2,E_9,E_6$ and $E_{12}$. In particular, $g$ is unique.

As explained by Kond\=o \cite[p.214]{Kondo}, the automorphism group of the graph is the same as the automorphism group of the subgraph $\Sigma$ generated by $\{E_i\}_{i \in \{1,\hdots,12\}}$, which is $(\bbZ/2\bbZ)^4 \rtimes D_4 $. Moreover, since the intersection behaviour of the curves is the same in any characteristic, it is still true that only an index $2$ subgroup of $\Aut(\Sigma)$ may be realized.

As for the realization of the automorphisms, note the following:

\begin{itemize}
\item A reflection $r_d$ along a diagonal axis is realized by a $2$-torsion section of the Jacobian of $|E_2+E_9+E_6+E_{12}+2(E_1+E_7+E_8)|$.
\item A reflection $r_v$ along the vertical axis is realized by the $2$-torsion section of the Jacobian of $|E_2+E_9+E_8+E_{10}+2(E_3+E_4+E_5+E_6+E_7)|$.
\item There is a $2$-torsion section of the Jacobian of the fibration $|2(E_3+E_4+E_5+E_{11})|$ which interchanges $E_2$ and $E_9$ as well as $E_6$ and $E_{12}$ while fixing $E_4,E_{11},E_8$ and $E_{10}$. Another $2$-torsion section of the same fibration induces the numerically trivial involution.
\item After fixing $E_6$ as a special bisection $N$ of $|2(E_3+E_4+E_5+E_{11})|$, the automorphism $\iota$ of Remark \ref{extraaut3} fixes $E_6$ and $E_{12}$ and interchanges $E_2$ and $E_9$. Moreover, it acts non-trivially on exactly one of the pairs $(E_3,E_{10})$ and $(E_4,E_{11})$.
\end{itemize}

These facts are checked by using Corollary \ref{jac2} and following through the construction of $jac_2$.
Now, note that we can compute the pointwise stabilizer $G$ of the set $\{E_1,E_3,$ $E_5,E_7\}$ using Equation (\ref{eqnIII}). It is generated by $t_{s_1},t_{s_3}$ and $t_{s_2}$ as well as $\iota$ and the inversion involution. All these automorphisms commute with each other and $\iota^2 = t_{s_2}$, hence $G \cong \bbZ/4\bbZ \times (\bbZ/2\bbZ)^2$. Therefore, we have a short exact sequence
$$
\xymatrix{
0 \ar[r] & G \ar[r] & \Aut(X) \ar[r] & D_4 \ar[r] & 0.
}
$$
We claim that this sequence splits. Indeed, by \cite[Corollary 4.7 and Section 7.1]{mukaiohashi}, a tame semi-symplectic automorphism (i.e. an automorphism acting trivially on $\rm{H}^0(X,\omega_X^{\otimes 2})$) has order at most $6$. We have realized all reflections using translations by $2$-torsion sections, which are semi-symplectic, since they fix the base of an elliptic fibration and act as translation on the fibers, and tame, since we are working in characteristic different from $2$. Therefore, $r_d \circ r_v$ has order $4$ and the group generated by $r_d$ and $r_v$ is a subgroup of $\Aut(X)$ isomorphic to $D_4$. Hence, the sequence splits and the proof is finished.
\end{proof}

\begin{remark}
In particular, note that $\Aut(X)$ is not a semi-direct product $(\bbZ/2\bbZ)^4 \rtimes D_4$. This was already observed by H. Ohashi in \cite{Ohashi} and corrects a small mistake in \cite{Kondo}.
\end{remark}
\subsection{Degenerations and Moduli}
This is similar to the first two cases. However, the involution is always fixed point free, since the branch points of the degree $2$ map of $\bbP^1$s do not move.

\begin{proposition}\label{moduliIII}
Assume $\Char(k) \neq 2$. Let
\begin{equation*}
y^2 +  xy  = x^3 + 4 s^4 x^2 + s^4 x 
\end{equation*}
be the Weierstrass equation of an elliptic fibration $\tilde{\pi}$ with section on a K3 surface $\tilde{X}$. Define the involution $\sigma = t_{N^-} \circ J(\sigma)$, where $J(\sigma): s \mapsto -s$ and $t_{N^-}$ is translation by the section $N^- = (0,0)$. Then, $\sigma$ is fixed point free.
\end{proposition}

\begin{corollary}
Enriques surfaces of type $\III$ exist if and only if $\Char(k) \neq 2$. Moreover, they are unique if they exist.
\end{corollary}

\begin{remark}
The equation we took from \cite{Lang3} for $J(\pi)$ makes sense in characteristic $2$, where it defines a rational elliptic surface with singular fibers $\I_4$ at $t = 0$ and $\I_1^*$ at $t = \infty$. The degree $2$ cover $t \mapsto s^2$ given in Proposition \ref{moduliIII} is the Frobenius morphism and the base change along this morphism defines a rational elliptic surface with singular fibers $(\I_8,\III)$. This surface is the minimal resolution of singularities of a surface covering a $1$-dimensional family of classical Enriques surfaces with finite automorphism group of "type $\VIII$", as is shown by T. Katsura, S. Kond\=o and the author in \cite{KatsuraKondoMartin}. 
\end{remark}

\section{Enriques surfaces of type $\IV$}\label{secIV}

\subsection{Main theorem for type $\IV$}

\begin{theorem}
Let $X$ be an Enriques surface. The following are equivalent:

\begin{enumerate}
\item $X$ is of type $\IV$.
\item The dual graph of all $(-2)$-curves on $X$ contains the graph in Figure \ref{critIV}.

\begin{figure}[h!]
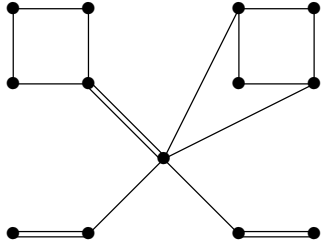

\centerline{
\xy
@={(0,40),(10,40),(30,40),(40,40),(0,30),(10,30),(30,30),(40,30),(20,20),(0,10),(10,10),(30,10),(40,10)}@@{*{\bullet}};
(0,40)*{};(10,40)*{}**\dir{-};
(30,40)*{};(40,40)*{}**\dir{-};
(0,30)*{};(10,30)*{}**\dir{-};
(30,30)*{};(40,30)*{}**\dir{-};
(0,40)*{};(0,30)*{}**\dir{-};
(10,30)*{};(10,40)*{}**\dir{-};
(30,30)*{};(30,40)*{}**\dir{-};
(40,30)*{};(40,40)*{}**\dir{-};
(10,30)*{};(20,20)*{}**\dir2{-};
(40,30)*{};(20,20)*{}**\dir{-};
(30,40)*{};(20,20)*{}**\dir{-};
(10,10)*{};(20,20)*{}**\dir{-};
(30,10)*{};(20,20)*{}**\dir{-};
(10,10)*{};(0,10)*{}**\dir2{-};
(30,10)*{};(40,10)*{}**\dir2{-};
\endxy
}
\caption{Critical subgraph for type $\IV$}
\label{critIV}
\end{figure}
\item The canonical cover $\tilde{X}$ of $X$ admits an elliptic fibration with a Weierstrass equation of the form
\begin{equation*}
y^2  = x^3 + 2 (s^4+1) x^2 + (s^4-1)^2 x
\end{equation*}
such that the covering morphism $\rho: \tilde{X} \to X$ is given as quotient by the involution
$
\sigma = t_{N^-} \circ J(\sigma),
$
where $J(\sigma): s \mapsto -s$ and $t_{N^-}$ is translation by $N^- = (-(s^2-1)^2,0)$.
\end{enumerate}
Moreover, Enriques surfaces of type $\IV$ do not exist in characteristic $2$.
\end{theorem}

\begin{proof}
First, we observe that the dual graph of type $\IV$ (see Table \ref{main}) contains the graph in Figure \ref{critIV}. 

This subgraph can be interpreted as the dual graph of a special elliptic fibration $\pi$ with singular fibers $(\I_4,\I_4,\I_2,\I_2)$ and special bisection $N$ as follows:

\vspace{3mm}

\centerline{
\xy
@={(0,40),(10,40),(30,40),(40,40),(0,30),(10,30),(30,30),(40,30),(20,20),(0,10),(10,10),(30,10),(40,10)}@@{*{\bullet}};
(0,40)*{};(10,40)*{}**\dir{-};
(30,40)*{};(40,40)*{}**\dir{-};
(0,30)*{};(10,30)*{}**\dir{-};
(30,30)*{};(40,30)*{}**\dir{-};
(0,40)*{};(0,30)*{}**\dir{-};
(10,30)*{};(10,40)*{}**\dir{-};
(30,30)*{};(30,40)*{}**\dir{-};
(40,30)*{};(40,40)*{}**\dir{-};
(10,30)*{};(20,20)*{}**\dir2{-};
(40,30)*{};(20,20)*{}**\dir{-};
(30,40)*{};(20,20)*{}**\dir{-};
(10,10)*{};(20,20)*{}**\dir{-};
(30,10)*{};(20,20)*{}**\dir{-};
(10,10)*{};(0,10)*{}**\dir2{-};
(30,10)*{};(40,10)*{}**\dir2{-};
(-3,43)*{};(13,43)*{}**\dir{--};
(-3,27)*{};(13,27)*{}**\dir{--};
(13,27)*{};(13,43)*{}**\dir{--};
(-3,27)*{};(-3,43)*{}**\dir{--};
(27,43)*{};(43,43)*{}**\dir{--};
(27,27)*{};(43,27)*{}**\dir{--};
(43,27)*{};(43,43)*{}**\dir{--};
(27,27)*{};(27,43)*{}**\dir{--};
(-3,13)*{};(13,13)*{}**\dir{--};
(-3,7)*{};(13,7)*{}**\dir{--};
(13,7)*{};(13,13)*{}**\dir{--};
(-3,7)*{};(-3,13)*{}**\dir{--};
(27,13)*{};(43,13)*{}**\dir{--};
(27,7)*{};(43,7)*{}**\dir{--};
(43,7)*{};(43,13)*{}**\dir{--};
(27,7)*{};(27,13)*{}**\dir{--};
(20,24.5)*{N};
\endxy
}
\vspace{3mm}

The bisection $N$ splits into two sections $N^+$ and $N^-$ of the elliptic fibration $\tilde{\pi}$ induced by $\pi$ on the K3 cover $\tilde{X}$. Fixing $N^+$ as the zero section, we compute $h(N^-) = 0$ and we see that $N^-$ is a $2$-torsion section of $\tilde{\pi}$ meeting the $\I_4$ fibers coming from the $\I_2$ fibers of $\pi$ in a non-identity component. 
The same argument as for type $\III$ shows that this type cannot exist in characteristic $2$.

Now, Corollary \ref{jac2} gives three more $(-2)$-curves resulting in the following graph:

\vspace{5mm}
\centerline{
\xy
@={(0,40),(10,40),(30,40),(40,40),(0,30),(10,30),(30,30),(40,30),(20,20),(0,10),(10,10),(30,10),(40,10),(50,50),(20,50),(50,20)}@@{*{\bullet}};
(0,40)*{};(10,40)*{}**\dir{-};
(30,40)*{};(40,40)*{}**\dir{-};
(0,30)*{};(10,30)*{}**\dir{-};
(30,30)*{};(40,30)*{}**\dir{-};
(0,40)*{};(0,30)*{}**\dir{-};
(10,30)*{};(10,40)*{}**\dir{-};
(30,30)*{};(30,40)*{}**\dir{-};
(40,30)*{};(40,40)*{}**\dir{-};
(10,30)*{};(20,20)*{}**\dir2{-};
(40,30)*{};(20,20)*{}**\dir{-};
(30,40)*{};(20,20)*{}**\dir{-};
(10,10)*{};(20,20)*{}**\dir{-};
(30,10)*{};(20,20)*{}**\dir{-};
(10,10)*{};(0,10)*{}**\dir2{-};
(30,10)*{};(40,10)*{}**\dir2{-};
(50,50)*{};(0,40)*{}**\dir2{-};
(50,50)*{};(30,40)*{}**\dir{-};
(50,50)*{};(40,30)*{}**\dir{-};
(50,50)*{};(30,10)*{}**\crv{(70,-10)};
(10,10)*{};(50,50)*{}**\crv{(70,-10)};
(20,50)*{};(10,40)*{}**\dir2{-};
(20,50)*{};(40,40)*{}**\dir{-};
(20,50)*{};(30,30)*{}**\dir{-};
(20,50)*{};(30,10)*{}**\dir{-};
(10,10)*{};(20,50)*{}**\dir{-};
(50,20)*{};(0,30)*{}**\dir2{-};
(50,20)*{};(40,40)*{}**\dir{-};
(50,20)*{};(30,30)*{}**\dir{-};
(50,20)*{};(30,10)*{}**\dir{-};
(10,10)*{};(50,20)*{}**\dir{-};
\endxy
}
\vspace{6mm}

Again, to produce additional $(-2)$-curves, we find a different special fibration with special bisection $N$ on this surface as follows:

\vspace{5mm}
\centerline{
\xy
(53,50)*{N};
@={(0,40),(30,40),(10,30),(30,30),(40,30),(0,10),(40,10),(50,50),(20,50),(50,20)}@@{*{\bullet}};
(30,30)*{};(40,30)*{}**\dir{-};
(30,30)*{};(30,40)*{}**\dir{-};
(50,50)*{};(0,40)*{}**\dir2{-};
(50,50)*{};(30,40)*{}**\dir{-};
(50,50)*{};(40,30)*{}**\dir{-};
(20,50)*{};(30,30)*{}**\dir{-};
(50,20)*{};(30,30)*{}**\dir{-};
\endxy
}
\vspace{3mm}

This special fibration has one $\I_0^*$ fiber and four disjoint $(-2)$-curves contained in some other fibers. Such a fibration will be extremal in any case by Lemma \ref{shiodatate}, so by Table \ref{extremalrational} the fibers are $(\I_0^*,\I_0^*)$. Hence, we obtain one more $(-2)$-curve. We leave it to the reader to find three more such diagrams and to check that the resulting graph is the one of type $\IV$.


We use the same equation as for surfaces of type $\III$
\begin{equation*}
y^2 + xy = x^3 + 4t^2x^2 + t^2x,
\end{equation*}
where $t$ is a coordinate on $\bbP^1$. Recall that the $\I_4$ fibers are at $t = 0, \infty$, while the $\I_2$ fibers are at $t = \pm \frac{1}{4}$. The non-trivial $2$-torsion sections are $s_1 = (-4t^2,2t^2)$, $s_2 = (0,0)$ and $s_3 = (-\frac{1}{4},\frac{1}{8})$.

In characteristic different from $2$, we can write a degree $2$ morphism $\bbP^1 \to \bbP^1$ with $t = \pm \frac{1}{4}$ as branch points in the following form
\begin{equation*}
t \mapsto \frac{1}{4} (\frac{s^2-1}{s^2+1}),
\end{equation*}
where $s$ is the new parameter on $\bbP^1$. The covering involution $J(\sigma)$ is given by $s \mapsto -s$. After scaling $x$ and $y$ and simplifying we get the equation
\begin{equation}\label{eqnIV}
y^2  = x^3 + 2 (s^4+1) x^2 + (s^4-1)^2 x
\end{equation}
together with the $2$-torsion sections $s_1' = (-(s^2-1)^2,0)$, $s_2' = (0,0)$ and $s_3' = (-(s^2+1)^2,0)$ obtained by pulling back $s_1$,$s_2$ and $s_3$. All of them are $J(\sigma)$-anti-invariant. However, $s_2'$ meets the identity component of the fiber at $s = 0$. Moreover, the surface defined by equation (\ref{eqnIV}) has an automorphism $\iota$ interchanging $s_1'$ and $s_3'$ given by $\iota: s \mapsto \sqrt{-1}s$.

Therefore, we can choose $s_1'$ as $N^-$.

\end{proof}

\begin{remark}\label{extrasection}
It is important to observe that the fibration $\tilde{\pi}$ defined by Equation (\ref{eqnIV}) has more torsion sections than the ones coming from the rational surface. For example, one can check that $P = (-(s-\sqrt{-1})^2(s^2-1),-2s(s-\sqrt{-1})^2(s^2-1))$ is a section satisfying $P \oplus P = N^-$. Since $t_P \circ \iota$ commutes with $\sigma$, it will induce an automorphism of the Enriques surface, which we will also denote by $t_P \circ \iota$.
 Moreover, $(t_P \circ \iota)^2 = t_Q \circ J(\sigma)$ for a $4$-torsion section $Q$ of $\tilde{\pi}$. Again, note that $t_P \circ \iota$ acts as $\sqrt{-1}$ on a non-zero global $2$-form of the K3 surface.
\end{remark}
\subsection{Automorphisms}

\begin{proposition}\label{Aut4}
Let $X$ be an Enriques surface of type $\IV$. Then, $Aut(X) \cong (\bbZ/2\bbZ)^4 \rtimes  (\bbZ/5\bbZ \rtimes \bbZ/4\bbZ)$ and this group is generated by automorphisms induced by sections of the Jacobian fibrations of elliptic fibrations of $X$ and an automorphism exhibited in Remark \ref{extrasection}. More precisely, we can choose the sections in such a way that at most one of them is not $2$-torsion and that none of them is a section of an isotrivial fibration. Moreover, $\Aut_{nt} \cong \{1\}$.
\end{proposition}

\begin{proof}
Recall that the dual graph of $(-2)$-curves for type $\IV$ is as follows:

\vspace{0.5mm}
\centerline{
\includegraphics[width=100mm]{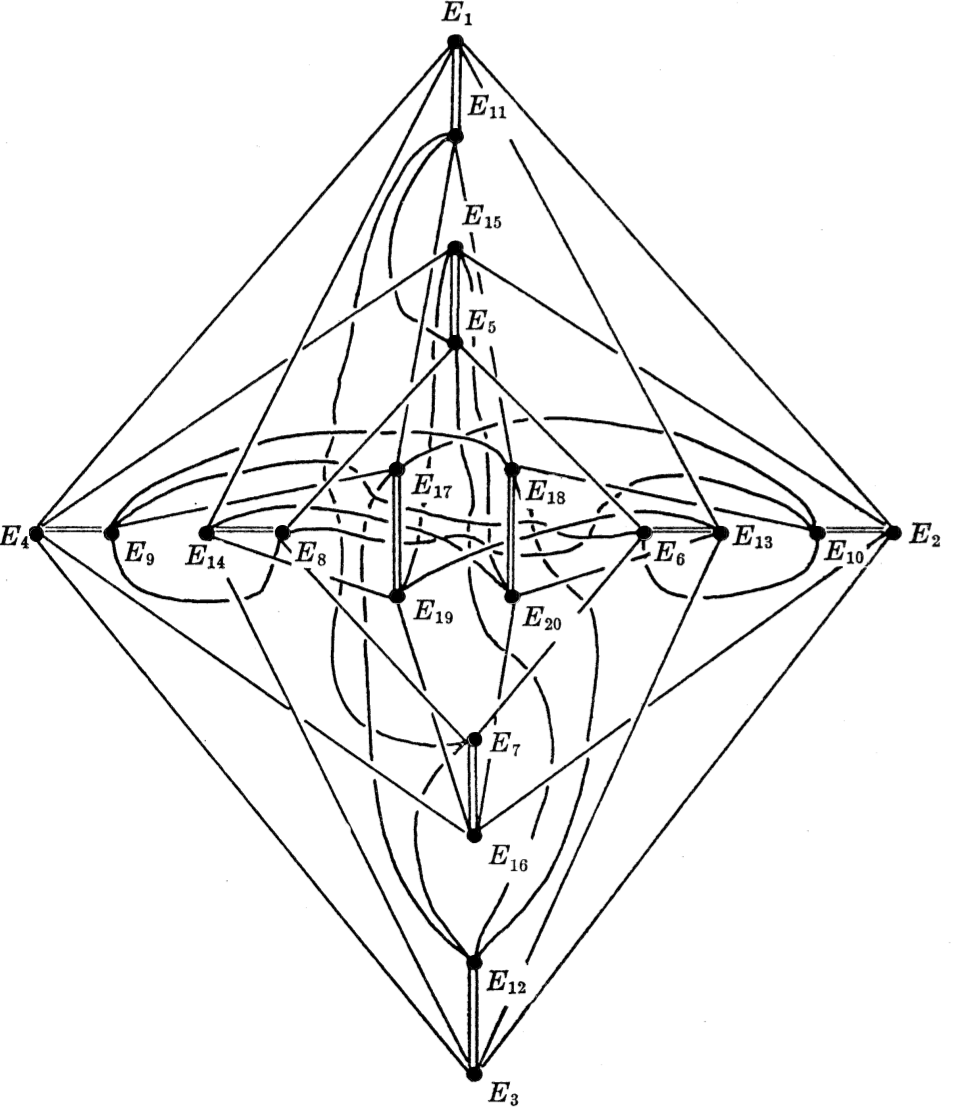}
}
\vspace{1mm}

We claim that $\Aut_{nt}(X)$ is trivial. Indeed, a numerically trivial automorphism $g$ acts trivially on the base of the fibration $|2(E_1+E_{11})|$, since this fibration has four reducible fibers and $g$ fixes the four bisections $E_2,E_4,E_{13}$ and $E_{14}$ pointwise, hence it is trivial.

Following \cite[p.217]{Kondo} we look at the action of $\Aut(X)$ on the set of five fibrations $\{ \Delta_i | i = 1,\hdots,5\}$ with $\Delta_1 = |2(E_1+E_{11})|$, $\Delta_2 = |2(E_2+E_{10})|$, $\Delta_3 = |2(E_5+E_{15})|$, $\Delta_4 = |2(E_6+E_{13})|$ and $\Delta_5 = |2(E_{17}+E_{19})|$. The kernel of the induced homomorphism $\psi: \Aut(X) \to \mathfrak{S}_5$ is isomorphic to $(\bbZ/2\bbZ)^4$ and it is generated by translations by $2$-torsion sections of the Jacobians of the $\Delta_i$ \cite[p.218]{Kondo}. From the dual graph, we see that an automorphism of $X$ cannot act as a permutation of order $3$ or as a transposition on $\{\Delta_1,\hdots,\Delta_5\}$. Now, we show that the image of $\psi$ is the group $G$ generated by
\begin{eqnarray*}
\varphi_1:& &\Delta_1 \mapsto \Delta_3 \mapsto \Delta_4 \mapsto \Delta_2 \mapsto \Delta_5 \\
\varphi_2:& &\Delta_1 \mapsto \Delta_3 \mapsto \Delta_2 \mapsto \Delta_4 .
\end{eqnarray*}
Using Corollary \ref{jac2}, these permutations are realized as follows:

\begin{itemize}
\item The Jacobian of $|E_5+E_6+E_{10}+E_{18}+E_{11}|$ has a $5$-torsion section which realizes $\varphi_1$.
\item If we fix $E_{11}$ as a special bisection of $\Delta_5$, we obtain a section $P$ by Remark \ref{extrasection} such that $\varphi_2$ is realized by the automorphism $t_P \circ \iota$. To see this, note that a $4$-torsion section of the Jacobian of $\Delta_5$ acts as $\Delta_1 \mapsto \Delta_2; \Delta_3 \mapsto \Delta_4$.
\end{itemize}

We have $G \cong \bbZ/5\bbZ \rtimes \bbZ/4\bbZ$ and, since $[\mathfrak{S}_5:G] \geq 6$, this yields the claim on the image of $\psi$. Now, note that we can compose $t_P \circ \iota$ with an involution interchanging the two $\I_2$ fibers of the $\Delta_5$ fibration to obtain an automorphism of order $4$ realizing $\varphi_2$. Hence, we obtain a splitting of
$$
\xymatrix{
0 \ar[r] & (\bbZ/2\bbZ)^4 \ar[r] & \Aut(X) \ar[r] & \bbZ/5\bbZ \rtimes \bbZ/4\bbZ \ar[r] & 0.
}
$$
This finishes the proof.
\end{proof}

\subsection{Degenerations and Moduli}
Similarly to the previous case, we obtain information about degenerations and moduli by direct calculation.

\begin{proposition}\label{moduliIV}
Assume $\Char(k) \neq 2$. Let
\begin{equation*}
y^2  = x^3 + 2 (s^4+1) x^2 + (s^4-1)^2 x
\end{equation*}
be the Weierstrass equation of an elliptic fibration $\tilde{\pi}$ with section on a K3 surface $\tilde{X}$. Define the involution $\sigma = t_{N^-} \circ J(\sigma)$, where $J(\sigma): s \mapsto -s$ and $t_{N^-}$ is translation by the section $N^- = (-(s^2-1)^2,0)$. Then, $\sigma$ is fixed point free.
\end{proposition}

\begin{corollary}
Enriques surfaces of type $\IV$ exist if and only if $\Char(k) \neq 2$. Moreover, they are unique if they exist.
\end{corollary}

\section{Enriques surfaces of type $\V$}\label{secV}

\subsection{Main theorem for type $\V$}

\begin{theorem}
Let $X$ be an Enriques surface. The following are equivalent:

\begin{enumerate}
\item $X$ is of type $\V$.
\item The dual graph of all $(-2)$-curves on $X$ contains the graph in Figure \ref{critV}.

\begin{figure}[h!]
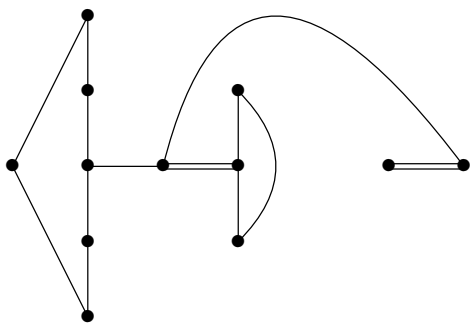

\centerline{
\xy
@={(0,20),(10,0),(10,10),(10,20),(10,30),(10,40),(20,20),(30,10),(30,20),(30,30),(50,20),(60,20)}@@{*{\bullet}};
(0,20)*{};(10,40)*{}**\dir{-};
(0,20)*{};(10,0)*{}**\dir{-};
(10,0)*{};(10,40)*{}**\dir{-};
(10,20)*{};(20,20)*{}**\dir{-};
(30,20)*{};(20,20)*{}**\dir2{-};
(30,10)*{};(30,30)*{}**\dir{-};
(30,10)*{};(30,30)*{}**\crv{(40,20)};
(50,20)*{};(60,20)*{}**\dir2{-};
(20,20)*{};(60,20)*{}**\crv{(30,60)};
\endxy
}
\caption{Critical subgraph for type $\V$}
\label{critV}
\end{figure}
\item The canonical cover $\tilde{X}$ of $X$ admits an elliptic fibration with a Weierstrass equation of the form
\begin{equation*}
y^2  + (s^2+1)xy+(s^2+1)y = x^3 + (s^2+2)x^2 + (s^2+1)x
\end{equation*}
such that the covering morphism $\rho: \tilde{X} \to X$ is given as quotient by the involution
$
\sigma = t_{N^-} \circ J(\sigma),
$
where $J(\sigma): s \mapsto -s$ and $t_{N^-}$ is translation by $N^- = (-1,0)$.
\end{enumerate}
Moreover, Enriques surfaces of type $\V$ do not exist in characteristic $2$ and $3$.
\end{theorem}

\begin{proof}
First, we observe that the dual graph of type $\V$ (see Table \ref{main}) contains the graph in Figure \ref{critV}. 

This subgraph can be interpreted as the dual graph of a special elliptic fibration $\pi$ with singular fibers $\I_6,\I_2$ (not $\III$, since it is double) and $\I_3$ (or $\IV$) and special bisection $N$ as follows, where the dotted rectangles mark the fibers:

\vspace{4mm}

\centerline{
\xy
@={(0,20),(10,0),(10,10),(10,20),(10,30),(10,40),(20,20),(30,10),(30,20),(30,30),(50,20),(60,20)}@@{*{\bullet}};
(0,20)*{};(10,40)*{}**\dir{-};
(0,20)*{};(10,0)*{}**\dir{-};
(10,0)*{};(10,40)*{}**\dir{-};
(10,20)*{};(20,20)*{}**\dir{-};
(30,20)*{};(20,20)*{}**\dir2{-};
(30,10)*{};(30,30)*{}**\dir{-};
(30,10)*{};(30,30)*{}**\crv{(40,20)};
(50,20)*{};(60,20)*{}**\dir2{-};
(20,20)*{};(60,20)*{}**\crv{(30,60)};
(-3,43)*{};(13,43)*{}**\dir{--};
(-3,-3)*{};(13,-3)*{}**\dir{--};
(13,-3)*{};(13,43)*{}**\dir{--};
(-3,-3)*{};(-3,43)*{}**\dir{--};
(27,33)*{};(43,33)*{}**\dir{--};
(27,7)*{};(43,7)*{}**\dir{--};
(43,7)*{};(43,33)*{}**\dir{--};
(27,7)*{};(27,33)*{}**\dir{--};
(47,23)*{};(63,23)*{}**\dir{--};
(47,17)*{};(63,17)*{}**\dir{--};
(63,17)*{};(63,23)*{}**\dir{--};
(47,17)*{};(47,23)*{}**\dir{--};
(17,23)*{N};
\endxy
}
\vspace{4mm}

As before, the bisection $N$ splits into two sections $N^+$ and $N^-$ of the elliptic fibration $\tilde{\pi}$ induced by $\pi$ on the K3 cover $\tilde{X}$. Fixing $N^+$ as the zero section, we can compute $h(N^-) = 0$ and we see that $N^-$ is a $2$-torsion section of $\tilde{\pi}$ meeting the $\I_6$ and $\I_2$ fibers in a non-identity component. 

Note that the existence of this fibration already gives non-existence of this type of Enriques surfaces in characteristic $2$ and $3$, since an extremal fibration with singular fibers $\I_6$ and $ \I_2$ does not exist on rational surfaces in characteristic $3$ (see Table \ref{extremalrational}) and because a fibration with two double fibers cannot exist in characteristic $2$. Therefore, we will assume $\Char(k) \neq 2,3$ from now on.

Now, Corollary \ref{jac2} gives two more $(-2)$-curves resulting in the following graph:

\vspace{4mm}
\centerline{
\xy
@={(0,20),(10,0),(10,10),(10,20),(10,30),(10,40),(20,20),(20,0),(20,40),(30,10),(30,20),(30,30),(50,20),(60,20)}@@{*{\bullet}};
(0,20)*{};(10,40)*{}**\dir{-};
(0,20)*{};(10,0)*{}**\dir{-};
(10,0)*{};(10,40)*{}**\dir{-};
(10,20)*{};(20,20)*{}**\dir{-};
(30,20)*{};(20,20)*{}**\dir2{-};
(30,30)*{};(20,40)*{}**\dir2{-};
(30,10)*{};(20,0)*{}**\dir2{-};
(10,40)*{};(20,40)*{}**\dir{-};
(10,0)*{};(20,0)*{}**\dir{-};
(60,20)*{};(20,40)*{}**\dir{-};
(60,20)*{};(20,0)*{}**\dir{-};
(30,10)*{};(30,30)*{}**\dir{-};
(30,10)*{};(30,30)*{}**\crv{(40,20)};
(50,20)*{};(60,20)*{}**\dir2{-};
(20,20)*{};(60,20)*{}**\crv{(30,60)};
\endxy
}
\vspace{4mm}

For this example, one can use a fibration with an $\I_2^*$ fiber to produce another $(-2)$-curve:

\vspace{3mm}
\centerline{
\xy
@={(0,20),(10,0),(10,10),(10,20),(10,30),(10,40),(20,20),(20,0),(20,40),(30,10),(30,20),(30,30),(50,20),(60,20),(40,20)}@@{*{\bullet}};
(0,20)*{};(10,40)*{}**\dir{-};
(0,20)*{};(10,0)*{}**\dir{-};
(10,0)*{};(10,40)*{}**\dir{-};
(10,20)*{};(20,20)*{}**\dir{-};
(30,20)*{};(20,20)*{}**\dir2{-};
(30,30)*{};(20,40)*{}**\dir2{-};
(30,10)*{};(20,0)*{}**\dir2{-};
(10,40)*{};(20,40)*{}**\dir{-};
(10,0)*{};(20,0)*{}**\dir{-};
(60,20)*{};(20,40)*{}**\dir{-};
(60,20)*{};(20,0)*{}**\dir{-};
(30,10)*{};(30,30)*{}**\dir{-};
(30,10)*{};(30,30)*{}**\crv{(40,20)};
(50,20)*{};(60,20)*{}**\dir2{-};
(20,20)*{};(60,20)*{}**\crv{(30,60)};
(40,20)*{};(50,20)*{}**\dir2{-};
(40,20)*{};(30,20)*{}**\dir2{-};
(40,20)*{};(10,40)*{}**\crv{~**\dir2{-} (30,50)};
(40,20)*{};(30,10)*{}**\dir2{-};
\endxy
}
\vspace{4mm}

As usual, the remaining curves can be found similarly.


By \cite{Lang3}, we have, after simplifying, the following equation for the unique extremal and rational elliptic surface with singular fibers $(\I_6,\I_3,\I_2,\I_1)$
\begin{equation*}
y^2 + txy +ty = x^3 + (1+t)x^2 + tx,
\end{equation*}
where $t$ is a coordinate on $\bbP^1$. The $\I_6$ fiber is at $t = \infty$, the $\I_3$ fiber is at $t = 0$, the $\I_2$ fiber is at $t = 1$ and the $\I_1$ fiber is at $t = -8$. The non-trivial $2$-torsion section is $s = (-1,0)$.

In characteristic different from $2$, we can write a degree $2$ morphism $\bbP^1 \to \bbP^1$ with $t = 1,\infty$ as branch points in the following form
\begin{equation*}
t \mapsto s^2 + 1,
\end{equation*}
where $s$ is the new parameter on $\bbP^1$. The covering involution $J(\sigma)$ is given by $s \mapsto -s$. Now, we have the equation
\begin{equation}\label{eqnV}
y^2  + (s^2+1)xy+(s^2+1)y = x^3 + (s^2+2)x^2 + (s^2+1)x
\end{equation}
together with the $2$-torsion sections $s' = (-1,0)$ obtained by pulling back $s$. Since $s'$ is $J(\sigma)$-(anti-)invariant and meets the fibers in the correct components, it is the section we are looking for. 

\end{proof}
\subsection{Automorphisms}

\begin{proposition}\label{Aut5}
Let $X$ be an Enriques surface of type $\V$. Then, $\Aut(X) \cong \mathfrak{S}_4 \times \bbZ/2\bbZ$ and this group is generated by automorphisms induced by $2$-torsion sections of the Jacobian fibrations of elliptic fibrations of $X$. Moreover, $\Aut_{nt}(X) \cong \bbZ/2\bbZ$ and $\Aut(X)/\Aut_{nt}(X) \cong \mathfrak{S}_4$.
\end{proposition}

\begin{proof}
Recall that the dual graph of $(-2)$-curves for type $\V$ is as follows:

\vspace{-1mm}
\centerline{
\includegraphics[width=110mm]{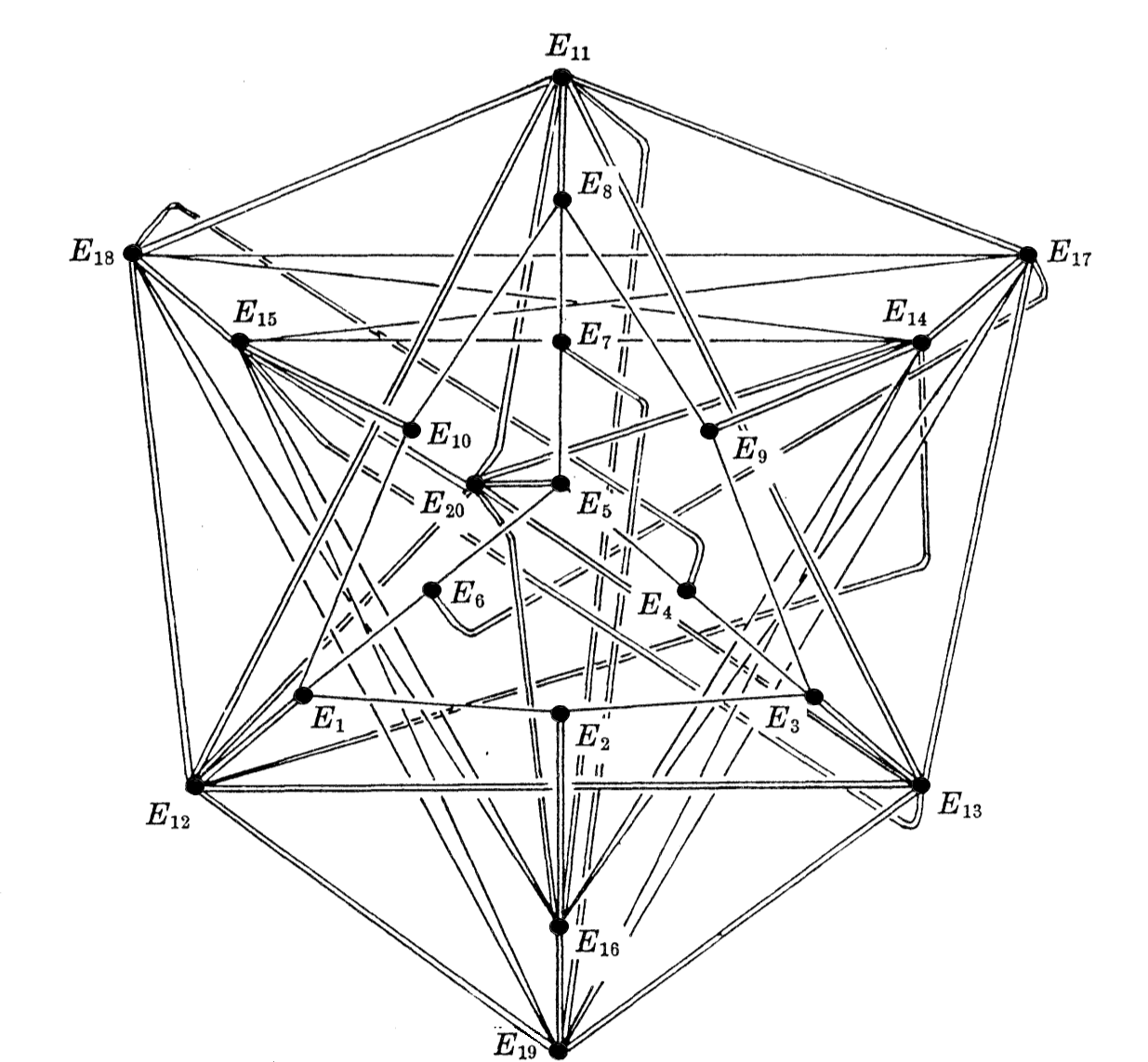}
}
\vspace{2mm}

We claim that $|\Aut_{nt}(X)| = 2$. Indeed, a numerically trivial automorphism $g$ acts trivially on the base of the fibration $|2(E_1+E_2+E_3+E_4+E_5+E_6)|$, since this fibration has at least three singular fibers and $g$ acts trivially or induces a unique involution on the three bisections $E_{10}$,$E_7$ and $E_9$. By the same argument as for type $\III$, there is at most one such $g$. Now, note that the $2$-torsion section of the Jacobian of this fibration acts identically on the graph of $(-2)$-curves.

The automorphism group of the graph is $\mathfrak{S}_4$ \cite[p.223]{Kondo}. It suffices to look at the action of $\Aut(X)$ on the set $\{E_1,E_3,E_5,E_8\}$.

\begin{itemize}
\item Transpositions of $E_5$ with another curve of the set are induced by $2$-torsion sections of fibrations with a singular fiber of type $\I_2^*$. For example, there is a $2$-torsion section of the Jacobian of $|E_2+E_6+E_7+E_9+2(E_1+E_8+E_{10})|$ which interchanges $E_3$ and $E_5$ by Corollary \ref{jac2}.
\item All transpositions of two curves different from $E_5$ are induced by $2$-torsion sections of fibrations with a singular fiber of type $\III^*$, e.g. the $2$-torsion section of the Jacobian of $|E_{10}+E_9+2E_1+2E_3+2E_7+3E_6+3E_4+4E_5|$ interchanges $E_{10}$ and $E_9$.
\end{itemize}

Finally, we claim that these transpositions generate a subgroup of $\Aut(X)$, which is isomorphic to $\mathfrak{S}_4$. Indeed, this can be checked by using Equation (\ref{eqnV}) to compute the stabilizer $G$ of $E_1$ (which is $D_6$) and by using the fact that the maximal order of a tame semi-symplectic automorphism is $6$ (see \cite{mukaiohashi}). This finishes the proof.
\end{proof}

\subsection{Degenerations and Moduli}
As in the previous cases, we prove the existence of this type by explicit calculation.
\begin{proposition}\label{moduliV}
Assume $\Char(k) \neq 2,3$. Let
\begin{equation*}
y^2  + (s^2+1)xy+(s^2+1)y = x^3 + (s^2+2)x^2 + (s^2+1)x
\end{equation*}
be the Weierstrass equation of an elliptic fibration $\tilde{\pi}$ with section on a K3 surface $\tilde{X}$. Define the involution $\sigma = t_{N^-} \circ J(\sigma)$, where $J(\sigma): s \mapsto -s$ and $t_{N^-}$ is translation by the section $N^- = (-1,0)$. Then, $\sigma$ is fixed point free.
\end{proposition}

\begin{corollary}
Enriques surfaces of type $\V$ exist if and only if $\Char(k) \neq 2,3$. Moreover, they are unique if they exist.
\end{corollary}

\begin{remark}
Again, the equation makes sense in characteristic $2$, where it defines a $K3$ surface covering a $1$-dimensional family of classical and supersingular Enriques surfaces of type $\VII$ (see \cite{KatsuraKondo}).
\end{remark}

\section{Enriques surfaces of type $\VI$}\label{secVI}

\subsection{Main theorem for type $\VI$}

\begin{remark}
In the first five cases, every base change with the correct ramification points produced an elliptic fibration of a K3 surface with $J(\pi)$-Enriques section $N^-$. This happened because the section $N^-$ was a $2$-torsion section. In the last two cases, however, we do not get this section for free.
\end{remark}

\begin{lemma}\label{lemVI}
Let $\Char(k) \neq 3$, $J(\sigma): s \mapsto -s - \beta$, and
\begin{equation*}
y^2 - 3 ( 3(s^2 + \beta s) + 1) xy+ (3(s^2+ \beta s)+1)^2y = x^3
\end{equation*}
with $\beta \in k - \{ \pm \frac{2}{\sqrt{3}} \}$ be the Weierstrass equation of an elliptic fibration of a K3 surface. Then, an everywhere integral, $J(\sigma)$-anti-invariant section $N^-$ meeting the fiber at $s = \infty$ in a non-identity component exists if and only if $\beta = \pm 1$. Moreover, it is unique up to sign if it exists. Both cases are isomorphic and if $\beta = 1$, the section is given by $N^- = (s+s^2,s^3)$.
\end{lemma}

\begin{proof}
By \cite[Lemma 1.2]{Shioda}, an everywhere integral section $N^-$ is given by $(x(s),y(s))$, where $x(s)$ and $y(s)$ are polynomials in $s$ with $\deg_s(x) \leq 4$ and $\deg_s(y) \leq 6$. Now, a lengthy, but straightforward calculation comparing coefficients gives the result.
Finally, note that the automorphism $s \mapsto -s$ exchanges both cases.
%
%
\end{proof}

\begin{theorem}
Let $X$ be an Enriques surface. The following are equivalent:

\begin{enumerate}
\item $X$ is of type $\VI$.
\item The dual graph of all $(-2)$-curves on $X$ contains the graph in Figure \ref{critVI}.

\begin{figure}[h!]
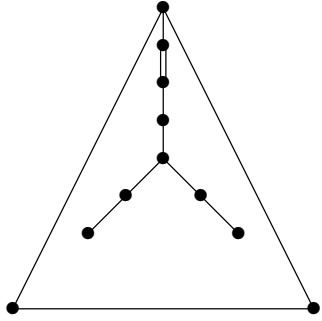

\centerline{
\xy
@={(20,20),(20,25),(20,30),(20,35),(20,40),(15,15),(10,10),(0,0),(25,15),(30,10),(40,0)}@@{*{\bullet}};
(20,20)*{};(20,30)*{}**\dir{-};
(20,20)*{};(30,10)*{}**\dir{-};
(20,20)*{};(10,10)*{}**\dir{-};
(20,35)*{};(20,30)*{}**\dir2{-};
(20,35)*{};(20,40)*{}**\dir{-};
(0,0)*{};(20,40)*{}**\dir{-};
(40,0)*{};(20,40)*{}**\dir{-};
(40,0)*{};(0,0)*{}**\dir{-};
\endxy
}
\caption{Critical subgraph for type $\VI$}
\label{critVI}
\end{figure}
\item The canonical cover $\tilde{X}$ of $X$ admits an elliptic fibration with a Weierstrass equation of the form
\begin{equation*}
y^2 - 3 ( 3s^2 + 3s + 1) xy+ (3s^2+ 3s+1)^2y = x^3
\end{equation*}
such that the covering morphism $\rho: \tilde{X} \to X$ is given as quotient by the involution
$
\sigma = t_{N^-} \circ J(\sigma),
$
where $J(\sigma): s \mapsto -s-1$ and $t_{N^-}$ is translation by $N^- = (s+s^2,s^3)$.
\end{enumerate}
Moreover, Enriques surfaces of type $\VI$ do not exist in characteristic $3$.
\end{theorem}

\begin{proof}
First, observe that the dual graph of type $\VI$ (see Table \ref{main}) contains the graph in Figure \ref{critVI}. 

This subgraph can be interpreted as the dual graph of a special elliptic fibration $\pi$ with singular fibers $\IV^*,\I_3$ (not $\III$, since it is double) and special $2$-section $N$. With the same notation as in the previous cases, we can compute $h(N^-) \neq 0$ and from Corollary \ref{jac2} we obtain two more $(-2)$-curves as follows:

\vspace{1mm}
\centerline{
\xy
@={(20,20),(20,25),(20,30),(20,35),(20,40),(15,15),(10,10),(0,0),(25,15),(30,10),(40,0),(35,5),(5,5)}@@{*{\bullet}};
(20,20)*{};(20,30)*{}**\dir{-};
(20,20)*{};(30,10)*{}**\dir{-};
(20,20)*{};(10,10)*{}**\dir{-};
(20,35)*{};(20,30)*{}**\dir2{-};
(20,35)*{};(20,40)*{}**\dir{-};
(0,0)*{};(20,40)*{}**\dir{-};
(40,0)*{};(20,40)*{}**\dir{-};
(40,0)*{};(0,0)*{}**\dir{-};
(30,10)*{};(35,5)*{}**\dir2{-};
(10,10)*{};(5,5)*{}**\dir2{-};
(40,0)*{};(35,5)*{}**\dir{-};
(0,0)*{};(5,5)*{}**\dir{-};
(5,5)*{};(35,5)*{}**\dir{-};
(20,35)*{};(5,5)*{}**\dir{-};
(20,35)*{};(35,5)*{}**\dir{-};
\endxy
}
\vspace{3mm}

There are three subgraphs of type $\tilde{A}_1$ such that the graph of $(-2)$-curves disjoint from this diagram together with a special bisection has the following form:
\vspace{5mm}

\centerline{
\xy
@={(20,20),(20,25),(20,30),(20,35),(15,15),(10,10),(0,0),(25,15),(30,10),(40,0)}@@{*{\bullet}};
(20,20)*{};(20,30)*{}**\dir{-};
(20,20)*{};(30,10)*{}**\dir{-};
(20,20)*{};(10,10)*{}**\dir{-};
(20,35)*{};(20,30)*{}**\dir2{-};
(40,0)*{};(0,0)*{}**\dir{-};
\endxy
}
\vspace{3mm}
The only rational elliptic fibration with a singular fiber of type $\I_2$ and some other singular fibers whose dual graphs contain an $A_5$ and an $A_2$ diagram is the one with singular fibers $(\I_6,\I_3,\I_2,\I_1)$ (resp. $(\I_6,\IV,\I_2)$ in characteristic $2$). Using the other $(-2)$-curves in the graph, one deduces that the $\I_6$ and $\I_3$ (resp. $\IV$) fibers are simple. These fibrations give the seven remaining $(-2)$-curves for the dual graph of type $\VI$. Observe that the existence of such a fibration excludes this case in characteristic $3$, since the $\I_2$ fiber is double.

We have found the following equation for the unique rational and extremal elliptic surface with singular fibers $\IV^*$ and $\I_3$ in any characteristic
\begin{equation*}
y^2 + txy +t^2y = x^3,
\end{equation*}
where $t$ is a coordinate on $\bbP^1$. By a change of coordinates (valid away from characteristic $3$) we obtain
\begin{equation*}
y^2 - 3(3t+1)xy + (3t+1)^2y = x^3.
\end{equation*}
 The $\IV^*$ fiber is at $t = -\frac{1}{3}$, the $\I_3$ fiber is at $t = \infty$ and there is an $\I_1$ fiber at $t = -\frac{2}{3}$. 

In characteristic $\neq 3$, we can write a degree $2$ morphism $\bbP^1 \to \bbP^1$ with $t = \infty$ as branch point and which is not branched over $-\frac{1}{3}$ as
\begin{equation*}
t \mapsto s^2+\beta s,
\end{equation*}
where $s$ is the new parameter on $\bbP^1$ and $\beta \neq \pm \frac{2}{\sqrt{3}}$. The covering involution $J(\sigma)$ is given by $s \mapsto -s-\beta$. We obtain the equation
\begin{equation*}\label{eqnVI}
y^2 - 3(3(s^2+\beta s)+1)xy + (3(s^2+\beta s)+1)^2y = x^3.
\end{equation*}
By Lemma \ref{lemVI}, if a suitable section $N^-$ exists, we can assume $\beta = 1$ and $N^- = (s+s^2,s^3)$.
\end{proof}

\subsection{Automorphisms}

\begin{proposition}\label{Aut6}
Let $X$ be an Enriques surface of type $\VI$. Then, $Aut(X) \cong \mathfrak{S}_5$ and this group is generated by automorphisms induced by $2$-torsion sections of the Jacobian fibrations of elliptic fibrations of $X$. Moreover, $\Aut_{nt}(X) \cong \{1\}$.
\end{proposition}

\begin{proof}
Recall that the dual graph of $(-2)$-curves for type $\VI$ is as follows:

\vspace{1mm}
\centerline{
\includegraphics[width=110mm]{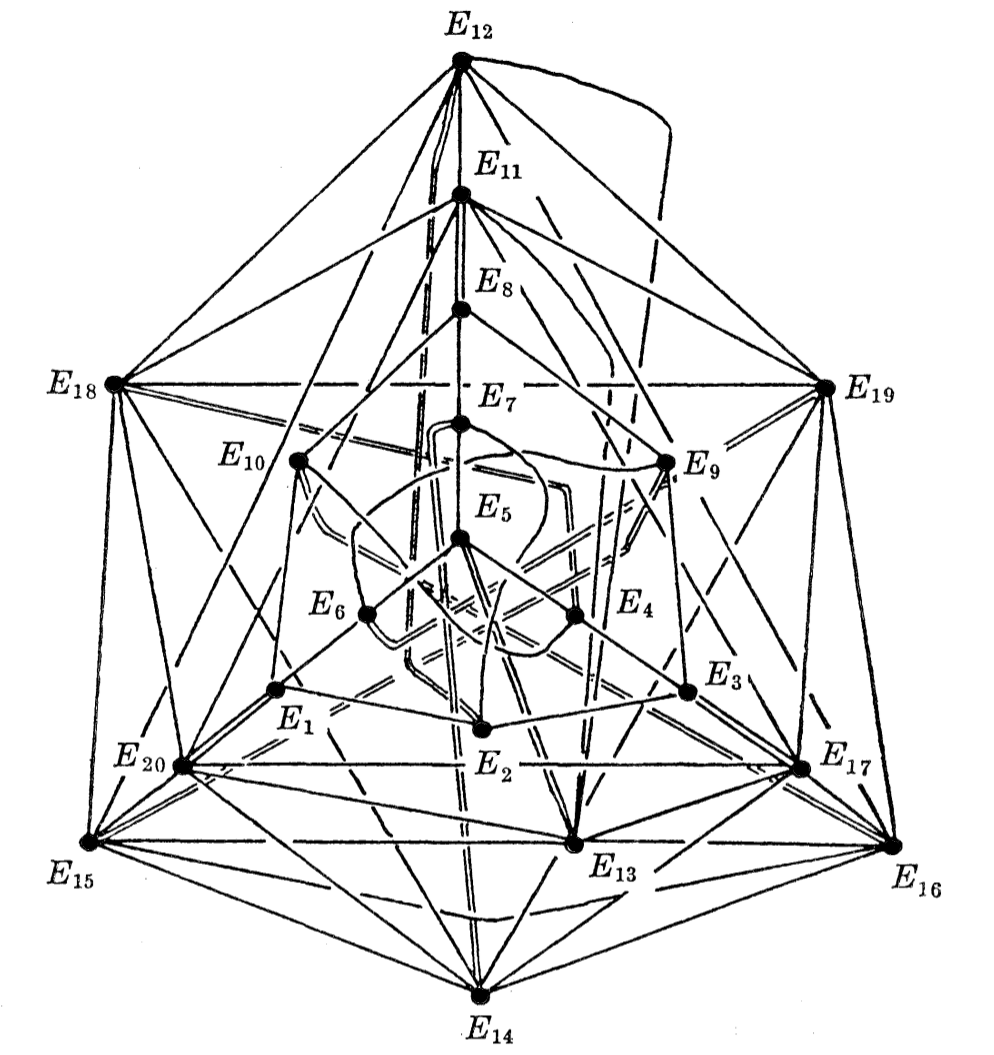}
}
\vspace{1mm}

Let us first show that $\Aut_{nt}(X)$ is trivial. Indeed, the three separable bisections $E_7,E_9$ and $E_{10}$ of $|E_1+E_2+E_3+E_4+E_5+E_6|$ are fixed pointwise by any numerically trivial automorphism, which therefore has to be the identity.

The automorphism group of the graph is $\mathfrak{S}_5$ \cite[p.223]{Kondo}. 
We look at the induced action of $\Aut(X)$ on the set $\Sigma = \{E_1,\hdots,E_{10}\}$ and note the following points:

\begin{itemize}
\item The pointwise stabilizer of the set $\Gamma_1 = \{E_4,E_5,E_6,E_7\}$ is $\bbZ/2\bbZ$. It is realized by the $2$-torsion section of the Jacobian of $|2(E_5+E_{13})|$.
\item The stabilizer of $E_5$ under the action of the automorphism group of the graph is $\mathfrak{S}_3 \times \bbZ/2\bbZ$. It is realized by the stabilizer of $\Gamma_1$ and the $2$-torsion sections of the Jacobian fibrations of fibrations with a fiber of type $\I_1^*$. For example the Jacobian of $|E_6+E_7+E_3+E_{10}+2(E_4+E_5)|$ has a $2$-torsion section which interchanges $E_6$ and $E_7$.
\item Since the stabilizer of $E_5$ has order $12$, it suffices to show that the group generated by $2$-torsion sections acts transitively on $\Sigma$. We show that we can map $E_5$ to $E_{10},E_3$ and $E_6$. The rest can be done similarly.
\item Indeed, the $2$-torsion sections of the Jacobians of $|2(E_3+E_{17})|$, $|2(E_{10}+E_{16})|$ and $|2(E_8+E_{11})|$ interchange $E_5$ and $E_{10}$, $E_5$ and $E_{3}$ and $E_3$ and $E_6$, respectively.
\end{itemize}
\end{proof}

\subsection{Degenerations and Moduli}

\begin{proposition}\label{moduliVI}
Assume $\Char(k) \neq 3$. Let
\begin{equation*}
y^2 - 3(3(s^2+ s)+1)xy + (3(s^2+ s)+1)^2y = x^3
\end{equation*}
be the Weierstrass equation of an elliptic fibration $\tilde{\pi}$ with section on a K3 surface $\tilde{X}$. Define the involution $\sigma = t_{N^-} \circ J(\sigma)$, where $J(\sigma): s \mapsto -s-1$ and $t_{N^-}$ is translation by the section $N^- = (s+s^2,s^3)$. Then, $\sigma$ is fixed point free if and only if $\Char(k) \neq 5$. If $\Char(k) = 5$, $\sigma$ has exactly one $(-2)$-curve as fixed locus.
\end{proposition}

\begin{proof}
The only possibility for $\sigma$ to have fixed points is the case where $\varphi: t \mapsto s^2 + s$ is branched over the point lying under the nodal fiber. Hence, we may assume that $\Char(k) \neq 2$. The branch points of $\varphi$ are $t = \infty$ and $t = -\frac{1}{4}$, while the $\I_1$ fiber of $\pi$ lies over $t = -\frac{2}{3}$. Hence, $\varphi$ is branched over the point lying under the nodal fiber if and only if $-\frac{2}{3} = -\frac{1}{4}$, i.e. if and only if $5 = 0$.

Now if $\Char(k) = 5$, the location of the $\I_2$ fiber of $\tilde{\pi}$ is $s = -\frac{1}{2} = 2$. The singular point of the Weierstrass equation at $s = 2$ is $(-1,1)$, while $N^-$ passes through $(1,3)$. Hence, $N^-$ meets the identity component of the $\I_2$ fiber and therefore it is not a $J(\pi)$-Enriques section and $\sigma$ fixes a $(-2)$-curve.
\end{proof}
\begin{corollary}
Enriques surfaces of type $\VI$ exist if and only if $\Char(k) \neq 3,5$. Moreover, they are unique if they exist.
\end{corollary}

Similarly to the cases of type $\I$ and $\II$, one obtains a Coble surface if $\sigma$ has a fixed curve, i.e. if $\Char(k) = 5$.
\section{Enriques surfaces of type $\VII$}\label{secVII}

\subsection{Main theorem for type $\VII$}

\begin{lemma}\label{lemVII}
Let $\Char(k) \neq 2$, $J(\sigma): s \mapsto -s$, and
\begin{equation*}
y^2 = x^3 - (s_\beta^2+s_\beta)x^2+(2s_\beta^3-3s_\beta^2+4s_\beta-2)x+(-s_\beta^3+2s_\beta^2-2s_\beta+1),
\end{equation*}
where $s_\beta = s^2 + \beta$ with $\beta \in k - \{1\}$, be the Weierstrass equation of an elliptic fibration of a K3 surface. Then, an everywhere integral, $J(\sigma)$-anti-invariant section $N^-$ meeting the fibers at $s = \infty$ and $s = \pm \sqrt{1-\beta}$ in a non-identity component exists if and only if $\beta \in \{0,2\}$. Moreover, it is unique up to sign if it exists. Both cases are isomorphic and if $\beta = 0$, the section is $N^- = (1,s-s^3)$.
\end{lemma}

\begin{proof}
Similarly to the previous case, one obtains conditions on $\beta$ by direct calculation.
Let us show the existence of the automorphism. The Weierstrass equation for the rational elliptic fibration
\begin{equation*}
y^2 = x^3 - (t^2+t)x^2+(2t^3-3t^2+4t-2)x+(-t^3+2t^2-2t+1)
\end{equation*}
has an automorphism
\begin{equation*}
t \mapsto 2-t; \quad x \mapsto x-2+2t.
\end{equation*}
This automorphism induces the desired isomorphism.
\end{proof}
\begin{theorem}
Let $X$ be an Enriques surface. The following are equivalent:

\begin{enumerate}
\item $X$ is of type $\VII$.
\item The dual graph of all $(-2)$-curves on $X$ contains the graph in Figure \ref{critVII}.

\begin{figure}[h!]
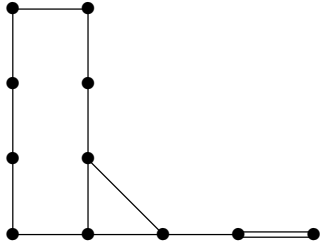

\centerline{
\xy
@={(0,0),(0,10),(0,20),(0,30),(10,0),(10,10),(10,20),(10,30),(20,0),(30,0),(40,0)}@@{*{\bullet}};
(0,0)*{};(0,30)*{}**\dir{-};
(10,0)*{};(10,30)*{}**\dir{-};
(0,0)*{};(10,0)*{}**\dir{-};
(10,30)*{};(0,30)*{}**\dir{-};
(10,0)*{};(20,0)*{}**\dir{-};
(10,10)*{};(20,0)*{}**\dir{-};
(30,0)*{};(20,0)*{}**\dir{-};
(30,0)*{};(40,0)*{}**\dir2{-};
\endxy
}
\caption{Critical subgraph for type $\VII$}
\label{critVII}
\end{figure}

\item The canonical cover $\tilde{X}$ of $X$ admits an elliptic fibration with a Weierstrass equation of the form
\begin{equation*}
y^2 = x^3 - (s^4+s^2)x^2+(2s^6-3s^4+4s^2-2)x+(-s^6+2s^4-2s^2+1)
\end{equation*}
such that the covering morphism $\rho: \tilde{X} \to X$ is given as quotient by the involution
$
\sigma = t_{N^-} \circ J(\sigma),
$
where $J(\sigma): s \mapsto -s$ and $t_{N^-}$ is translation by $N^- = (1,s-s^3)$.
\end{enumerate}
Moreover, singular Enriques surfaces of type $\VII$ do not exist in characteristic $2$.
\end{theorem}

\begin{proof}
First, observe that the dual graph of type $\VII$ (see Table \ref{main}) contains the graph in Figure \ref{critVII}. 

Conversely, we have shown in Example \ref{example} that we recover type $\VII$ from the critical subgraph and, since an elliptic fibration with singular fibers $\I_8$ and $\I_2$ (not $\III$, since it is a double fiber) does not exist in characteristic $2$, this type cannot exist in characteristic $2$.
%
%
%
%
%

We have found the following equation for the unique rational and extremal elliptic surface with singular fibers $(\I_8,\I_2,\I_1,\I_1)$ in characteristic different from $2$
\begin{equation*}
y^2 = x^3 - (t^2+t)x^2+(2t^3-3t^2+4t-2)x+(-t^3+2t^2-2t+1),
\end{equation*}
where $t$ is a coordinate on $\bbP^1$. The $\I_8$ fiber is at $t = 1$, the $\I_2$ fiber is at $t = \infty$ and there are two $\I_1$ fibers at $t = 1 \pm \frac{\sqrt{-1}}{2}$. 

In characteristic different from $2$, we can write a degree $2$ morphism $\bbP^1 \to \bbP^1$ with $t = \infty$ as branch point and which is not branched over $t = 0$ as
\begin{equation*}
t \mapsto s^2 + \beta,
\end{equation*}
where $s$ is the new parameter on $\bbP^1$ and $\beta \neq 0$. The covering involution $J(\sigma)$ is given by $s \mapsto -s$. Now, note that we are looking for a section $N^-$ which meets the $\I_4$ and $\I_8$ fibers in non-identity components. By Lemma \ref{lemVII}, if a suitable section $N^-$ exists, we can assume $\beta = 0$ and $N^- = (1,s-s^3)$. Moreover, one can check that $N^-$ has the correct intersection behaviour with the $\I_8$ fibers.
\end{proof}
\subsection{Automorphisms}

\begin{proposition}\label{Aut7}
Let $X$ be an Enriques surface of type $\VII$. Then, $Aut(X) \cong \mathfrak{S}_5$ and this group is generated by automorphisms induced by $2$-torsion sections of the Jacobian fibrations of elliptic fibrations of $X$. Moreover, $\Aut_{nt}(X) \cong \{1\}$.
\end{proposition}

\begin{proof}
Recall that the dual graph of $(-2)$-curves for type $\VII$ is as follows:

\centerline{
\includegraphics[width=110mm]{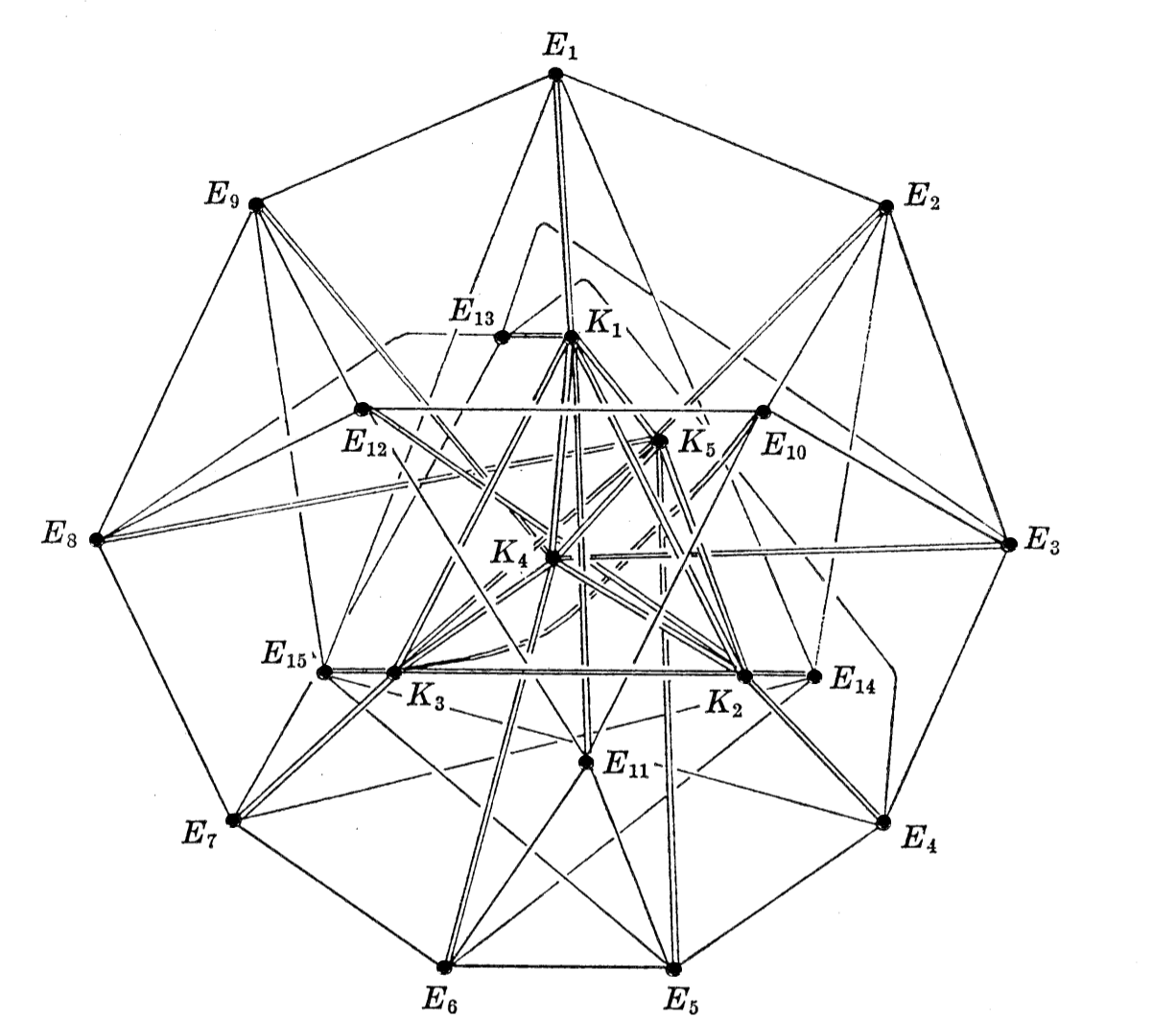}
}
\vspace{2mm}

We claim that $\Aut_{nt}(X)$ is trivial. Indeed, one can check that the bisections $E_2$, $E_3$, $E_5$, $E_6$, $E_8$ and $E_9$ of $|K_4+K_5|$ are fixed pointwise by any numerically trivial automorphism, which therefore has to be trivial.

The automorphism group of the graph is $\mathfrak{S}_5$. 
Following \cite[p.232]{Kondo}, we look at the induced action on the set $\Sigma = \{K_1,\hdots,K_5\}$ and observe that the pointwise stabilizer of $\Sigma$ is trivial.
Now, each $K_i$, $i \in \{1,\hdots,5\}$, meets exactly three $E_j$, $j \in \{1,\hdots,15\}$, twice. The $2$-torsion sections of the Jacobians of the elliptic fibrations $|2(K_i + E_j)|$ act as permutations of cycle type $(2,2)$ on $\Sigma - K_i$. Note that the $2$-torsion section of the Jacobian of $|K_4+K_5|$ interchanges $K_4$ and $K_5$ while fixing $K_1,K_2$ and $K_3$. Together, these involutions generate the full automorphism group.
\end{proof}

\subsection{Degenerations and Moduli}

\begin{proposition}\label{moduliVII}
Assume $\Char(k) \neq 2$. Let
\begin{equation*}
y^2 = x^3 - (s^4+s^2)x^2+(2s^6-3s^4+4s^2-2)x+(-s^6+2s^4-2s^2+1)
\end{equation*}
be the Weierstrass equation of an elliptic fibration $\tilde{\pi}$ with section on a K3 surface $\tilde{X}$. Define the involution $\sigma = t_{N^-} \circ J(\sigma)$, where $J(\sigma): s \mapsto -s$ and $t_{N^-}$ is translation by the section $N^- = (1,s-s^3)$. Then, $\sigma$ is fixed point free if and only if $\Char(k) \neq 5$. If $\Char(k) = 5$, $\sigma$ has exactly one $(-2)$-curve as fixed locus.
\end{proposition}

\begin{proof}
The branch points of $\varphi$ are $t = \infty$ and $t = 0$, while the $\I_1$ fibers of $\pi$ lie over $t = 1 \pm \frac{\sqrt{-1}}{2}$. Hence, $\varphi$ is branched over a point lying under a nodal fiber if and only if $1 \pm \frac{\sqrt{-1}}{2} = 0$, i.e. if and only if $5 = 0$.

Now, if $\Char(k) = 5$, the location of the $\I_2$ fiber of $\tilde{\pi}$ is $s = 0$. The singular point of the Weierstrass equation at $s = 0$ is $(2,0)$, while $N^-$ passes through $(1,0)$. Hence, $N^-$ meets the identity component of the $\I_2$ fiber and therefore it is not a $J(\pi)$-Enriques section and $\sigma$ fixes a $(-2)$-curve.
\end{proof}
\begin{corollary}
Enriques surfaces of type $\VII$ with smooth K3 cover exist if and only if $\Char(k) \neq 2,5$. Moreover, they are unique if they exist.
\end{corollary}

\begin{remark}
Here, it is important to recall our convention on Enriques surfaces in characteristic $2$. In fact, by \cite{KatsuraKondo}, there is a $1$-dimensional family of classical and supersingular Enriques surfaces of type $\VII$ in characteristic $2$. Note also that the involution $\sigma$ produces a Coble surface in characteristic $5$.
\end{remark}
\section{The classification-theorem}\label{graphs}

Now that we have completed the construction of the seven types of Enriques surfaces with finite automorphism group, it remains to show that these seven types are indeed all possible Enriques surfaces with finite automorphism group. Hence, the goal of this chapter is to prove the following classification-theorem, finishing the proof of the Main Theorem. Recall that all our Enriques surfaces are assumed to have a smooth canonical cover.

\begin{theorem}\label{graphtheorem}
Let $X$ be an Enriques surface. The following are equivalent:
\begin{enumerate}
\item $X$ has finite automorphism group.
\item Every elliptic fibration of $X$ is extremal.
\item Every special elliptic fibration of $X$ is extremal and $X$ contains a $(-2)$-curve.
\item The dual graph of all $(-2)$-curves on $X$ contains a critical subgraph for one of the types $\I,\hdots,\VII$.
\item The dual graph of all $(-2)$-curves on $X$ is one of the seven types $\I,\hdots,\VII$.
\item $X$ contains only finitely many, but at least one, $(-2)$-curves.
\end{enumerate}
\end{theorem}

Before giving the proof of Theorem \ref{graphtheorem}, we need to introduce the tools for the classification of dual graphs. 

\subsection{Preparations for the proof of the classification-theorem}
Corollary \ref{jac2} and the height pairing of sections of elliptic fibrations of the K3 cover will play an important role. More precisely, we have the following lemma.

\begin{lemma}\label{admissible}
Let $\pi: X \to \bbP^1$ be a special and extremal elliptic fibration of an Enriques surface $X$ with special bisection $N$. Let $\tilde{\pi}$ be the corresponding elliptic fibration of the K3 cover $\tilde{X}$ of $X$. Denote the irreducible curves on $\tilde{X}$ mapping surjectively onto $N$ by $N^+$ and $N^-$. Let $J(\pi)$ be the Jacobian of $\pi$. We choose $N^+$ as the zero section of $\tilde{\pi}$.

Then,
\begin{itemize}
\item either $h(N^-) = 0$ and $N^-$ is a $2$-torsion section in $\MW(J(\pi)) \subseteq \MW(\tilde{\pi})$
\item or $N^-$ satisfies
\begin{equation*}
 \sum_\nu contr_\nu (N^-) < 4 \quad \text{ and } \quad
 \sum_\nu contr_\nu (N^-,P) \in \{0,1,2\}
\end{equation*}
for all $P \in \MW(J(\pi)) \subseteq \MW(\tilde{\pi})$ with $P \neq N^-$.
\end{itemize}
\end{lemma}

\begin{proof}
Since 
\begin{equation*}
0 \leq h(N^-) = 4 + 2N^-.N^+ -  \sum_\nu contr_\nu (N^-) = 4 -  \sum_\nu contr_\nu (N^-)
\end{equation*}
and $N^-$ restricts to a $2$-torsion section on a fiber $F$ of $\tilde{\pi}$ lying over a double fiber of $\pi$, we either have $h(N^-) = 0$ and we claim that $N^-$ is $2$-torsion or $h(N^-) > 0$ and therefore $\sum_\nu contr_\nu (N^-) < 4$. 

Indeed, suppose $h(N^-) = 0$ and $N^-$ is not $2$-torsion. Then, $N^- \oplus N^-$ meets the zero section in $F$, hence its order is divisible by $\Char(k) = 2$ by \cite[Proposition 2.4]{ItoLiedtke}. But if $\Char(k) = 2$, the fiber $F$ is either multiplicative or ordinary by Proposition \ref{typeofdoublefiber}, contradicting \cite[Proposition 2.1]{ItoLiedtke}.


Since every $P \in \MW(J(\pi)) \subseteq \MW(\tilde{\pi})$ is disjoint from $N^+$, we have
\begin{eqnarray*}
0 = \langle P,N^- \rangle  &=& 2 + P.N^+ + N^-.N^+ - P.N^- - \sum_\nu contr_\nu (N^-,P) = \\ &=& 2 - P.N^- - \sum_\nu contr_\nu (N^-,P),
\end{eqnarray*}
which yields the second claim.
\end{proof}

\begin{remark}\label{careful}
By Table \ref{heightpairing}, the local contributions to the height pairing can be read off almost completely from the dual graph of singular fibers. However, a remark about the cases where $\pi$ has a double fiber of type $\I_1$ is in order. Since sections $P \in \MW(J(\pi))$ meet the corresponding $\I_2$ fiber of $\tilde{\pi}$ in the identity component, $N^-$ cannot be $2$-torsion. Moreover, $\sum_\nu contr_\nu (N^-)$ will decrease by $\frac{1}{2}$, while  $\sum_\nu contr_\nu (N^-,P)$ stays the same, hence $N^-$ can only satisfy the conditions of the lemma if it does so, when we ignore the double $\I_1$ fiber. We will do this from now on.
\end{remark}

\begin{definition}
Let $\Gamma_1$ be the dual graph of singular fibers of a rational and extremal elliptic fibration. A graph $\Gamma \supseteq \Gamma_1$ is called a fiber-bisection configuration for $\Gamma_1$ if the following two conditions hold:
\begin{enumerate}
\item $\Gamma - \Gamma_1$ consists of one vertex $N$ called the special bisection.
\item $N$ meets every connected component of $\Gamma_1$ of type $\tilde{D}$ and $\tilde{E}$ exactly twice and every component of type $\tilde{A}$ at least once and at most twice. Moreover, $N$ meets at most two connected components of $\Gamma_1$ exactly once.
\end{enumerate}
\end{definition}

Given such a fiber-bisection configuration $\Gamma$, we can check whether it could be the dual graph of a special elliptic fibration $\pi$ on an Enriques surface as follows:
Suppose it is the dual graph of $\pi$. Then, we can pass to the canonical cover, add the sections coming from the Jacobian $J(\pi)$ of $\pi$ and check the conditions of Lemma \ref{admissible}. By Remark \ref{careful}, it makes sense to say that a fiber-bisection configuration satisfies the conditions of Lemma \ref{admissible}.

\begin{definition}
A fiber-bisection configuration is called \emph{admissible} if it satisfies the conditions of Lemma \ref{admissible}
\end{definition}

\subsection{Outline of proof}
In this section, we outline the proof of the following lemma, which is the main ingredient in the proof of Theorem \ref{graphtheorem}.

\begin{lemma}\label{graphlemma}
Let $X$ be an Enriques surface such that every special elliptic fibration of $X$ is extremal and $X$ contains a $(-2)$-curve. Then, the dual graph of $(-2)$-curves on $X$ contains a critical subgraph (see Figures \ref{critI}$,\hdots,$ \ref{critVII}) for one of the types $\I,\hdots,\VII$.
\end{lemma}

\begin{proof}[Proof of Theorem \ref{graphtheorem} (assuming Lemma \ref{graphlemma})]
As observed by Dolgachev \cite[\S 4]{Dolgachev}, if $X$ has finite automorphism group, then every elliptic fibration $\pi$ on $X$ is extremal, since the Mordell-Weil group of $J(\pi)$ acts faithfully on $X$. In particular, since $X$ admits an elliptic fibration by Proposition \ref{ellipticpencil}, $X$ contains a $(-2)$-curve by Lemma \ref{shiodatate} and every special elliptic fibration of $X$ is extremal. From Lemma \ref{graphlemma}, we deduce that $X$ contains a critical subgraph, which, by the earlier chapters, implies that the dual graph of $(-2)$-curves on $X$ is one of the types $\I,\hdots,\VII$.

The seven dual graphs in Table \ref{main} consist of $12$ (resp. $20$) vertices, hence $X$ contains finitely many and at least one $(-2)$-curve. Moreover, we have computed the automorphism groups of these surfaces. They are finite. Finally, by Corollary \ref{jac2}, the only special elliptic fibrations of Enriques surfaces with finitely many, but at least one, $(-2)$-curves are the extremal ones.
\end{proof}

Since we have constructed all seven types in the previous chapters, Theorem \ref{graphtheorem} will finish the classification.
The strategy for the proof of Lemma \ref{graphlemma} can be summarized as follows:


\begin{enumerate}
\item Let $X$ be an Enriques surface with a $(-2)$-curve such that every special elliptic fibration of $X$ is extremal. By Proposition \ref{nodalisspecial}, $X$ admits such a special elliptic fibration $\pi$. 
\item Pick a dual graph $\Gamma_1$ of singular fibers of a rational and extremal elliptic fibration and some admissible fiber-bisection configuration $\Gamma \supseteq \Gamma_1$.
Suppose that $\Gamma$ is the dual graph of fibers and special bisection of $\pi$.
\item Apply Corollary \ref{jac2} to find additional $(-2)$-curves and obtain a bigger graph $\Gamma_2$.
\item If $\Gamma_2$ contains one of the critical subgraphs, we have shown in the previous chapters that $X$ is of one of the seven types.
\item If not, find a different subgraph $\Gamma_3$ of $\Gamma_2$ of type $\tilde{A}_n$ together with a vertex $N$ meeting $\Gamma_3$ exactly once. By Proposition \ref{canonicaltype}, $\Gamma_3$ is the dual graph of a singular fiber of a special elliptic fibration $\pi_1$ and $N$ is a special bisection of $\pi_1$.
By the assumption on $X$, $\pi_1$ is extremal, i.e. we can extend $\Gamma_3$ to a dual graph $\Gamma_4$ of singular fibers of an extremal elliptic fibration such that $\Gamma_4 \cup N$ is an admissible fiber-bisection configuration for $\Gamma_4$. Now, go back to step $(3)$.

\end{enumerate}

We will show that the above process will terminate at some point for every choice of $\Gamma_1$, either with a contradiction or with step $(4)$. 

\subsection{Proof of the classification-theorem}
The following lemma shows that the number of admissible fiber-bisection configurations we have to check is "not too big".

\begin{lemma}\label{reductionstep1}
Let $X$ be an Enriques surface with a special and extremal elliptic fibration $\pi$. Then, $X$ admits a special elliptic fibration with a double fiber of type ${\rm \I_n}$ with $n \geq 2$.
Moreover, if $\pi$ has double fibers of type $\rm{\I_{n_1}}$ and $\rm{\I_{n_2}}$, then $n_1 + n_2 \leq 8$.
\end{lemma}

\begin{proof}
For the first claim, let $\pi$ be a special and extremal elliptic fibration of $X$ and let $N$ be a special bisection of $\pi$. If $\pi$ has a fiber of type $\IV^*,\III^*,\II^*,\I_n^*$, or $\I_n$ with $n \geq 5$, then $N$ and fiber components form a fiber of type $\I_n$ and a component of the fiber takes the role of a special bisection.
The remaining possibilities for $\pi$ are the one with fibers $(\I_4,\I_4,\I_2,\I_2)$ and the one with fibers $(\I_3,\I_3,\I_3,\I_3)$. These are checked similarly, using more than one fiber.


For the second claim, let $\pi$ be a special elliptic fibration of $X$ with double fibers of type $\I_{n_1}$ and $\I_{n_2}$. Denote a special bisection by $N$ and the corresponding curves on the K3 cover by $N^+$ and $N^-$ as usual. Then, we compute $\sum_\nu contr_\nu(N^-) \geq (n_1+n_2)/2$ using Table \ref{heightpairing}. Since $\sum_\nu contr_\nu(N^-) \leq 4$, this gives the second claim.
\end{proof}

It is straightforward to give a complete list of admissible fiber-bisection configurations for dual graphs of singular fibers of extremal elliptic fibrations. We leave the details to the reader. Note that it follows immediately from the classification of extremal and rational elliptic surfaces (see Table \ref{extremalrational}) that we do not have to take special care of small characteristics with this method.

\begin{lemma}
Let $Adm_p$ be the set of admissible fiber-bisection configurations for dual graphs of extremal elliptic fibrations over an algebraically closed field of characteristic $p$. Then $Adm_p \subseteq Adm_0$.
\end{lemma}

\begin{lemma}\label{admissibletypes}
Table \ref{admissiblefiberbisection} shows the list of all admissible fiber-bisection configurations for dual graphs of singular fibers of extremal elliptic fibrations, where the special bisection meets at least one $\tilde{A}$ subgraph (marked with a $2$ in front) only once. 

\vspace{-3cm}
\begin{table}[!htb]
\begin{tabular}{|>{\centering\arraybackslash}m{3.5cm}|>{\centering\arraybackslash}m{10cm}|} \hline \vspace{1mm}
\rm{Dual graph of fibers} & \vspace{1mm} \rm{Admissible fiber-bisection configurations} \\ \hline \hline
$\tilde{E}_7 \oplus 2\tilde{A}_1$ & \vspace{1mm}
\resizebox{3cm}{!}{
\xy
@={(0,10),(10,10),(20,10),(30,10),(40,10),(50,10),(60,10),(30,0),(30,20),(30,30),(30,40)}@@{*{\bullet}};
(0,10)*{};(60,10)*{}**\dir{-};
(30,10)*{};(30,0)*{}**\dir{-};
(0,10)*{};(30,20)*{}**\dir{-};
(60,10)*{};(30,20)*{}**\dir{-};
(30,20)*{};(30,30)*{}**\dir{-};
(30,40)*{};(30,30)*{}**\dir2{-};
\endxy} \hspace{4mm}
\resizebox{3cm}{!}{
\xy
@={(0,10),(10,10),(20,10),(30,10),(40,10),(50,10),(60,10),(30,0),(30,20),(30,30),(30,40)}@@{*{\bullet}};
(0,10)*{};(60,10)*{}**\dir{-};
(30,10)*{};(30,0)*{}**\dir{-};
(60,10)*{};(30,20)*{}**\dir2{-};
(30,20)*{};(30,30)*{}**\dir{-};
(30,40)*{};(30,30)*{}**\dir2{-};
\endxy} \\ \hline
$\tilde{E}_6 \oplus 2\tilde{A}_2$ &  \vspace{1mm}
\resizebox{4cm}{!}{
\xy
@={(0,0),(10,0),(20,0),(30,0),(40,0),(50,0),(60,0),(70,0),(65,10),(20,10),(20,20)}@@{*{\bullet}};
(0,0)*{};(40,0)*{}**\dir{-};
(20,0)*{};(20,20)*{}**\dir{-};
(50,0)*{};(40,0)*{}**\dir2{-};
(50,0)*{};(70,0)*{}**\dir{-};
(65,10)*{};(70,0)*{}**\dir{-};
(60,0)*{};(65,10)*{}**\dir{-};
\endxy
} \\ \hline
$\tilde{D}_5 \oplus 2\tilde{A}_3$ & \vspace{1mm}
\resizebox{4cm}{!}{
\xy
@={(0,20),(0,0),(10,10),(20,10),(30,20),(30,0),(40,10),(50,10),(60,20),(60,0),(70,10)}@@{*{\bullet}};
(0,20)*{};(10,10)*{}**\dir{-};
(0,0)*{};(10,10)*{}**\dir{-};
(20,10)*{};(10,10)*{}**\dir{-};
(20,10)*{};(30,20)*{}**\dir{-};
(20,10)*{};(30,0)*{}**\dir{-};
(40,10)*{};(30,20)*{}**\dir2{-};
(40,10)*{};(50,10)*{}**\dir{-};
(60,20)*{};(50,10)*{}**\dir{-};
(60,0)*{};(50,10)*{}**\dir{-};
(60,20)*{};(70,10)*{}**\dir{-};
(60,0)*{};(70,10)*{}**\dir{-};
\endxy
} \hspace{4mm}
\resizebox{4cm}{!}{
\xy
@={(0,20),(0,0),(10,10),(20,10),(30,20),(30,0),(40,10),(50,10),(60,20),(60,0),(70,10)}@@{*{\bullet}};
(0,20)*{};(10,10)*{}**\dir{-};
(0,0)*{};(10,10)*{}**\dir{-};
(20,10)*{};(10,10)*{}**\dir{-};
(20,10)*{};(30,20)*{}**\dir{-};
(20,10)*{};(30,0)*{}**\dir{-};
(40,10)*{};(30,20)*{}**\dir{-};
(40,10)*{};(30,0)*{}**\dir{-};
(40,10)*{};(50,10)*{}**\dir{-};
(60,20)*{};(50,10)*{}**\dir{-};
(60,0)*{};(50,10)*{}**\dir{-};
(60,20)*{};(70,10)*{}**\dir{-};
(60,0)*{};(70,10)*{}**\dir{-};
\endxy} \\ \hline
$\tilde{D}_6 \oplus 2\tilde{A}_1 \oplus 2\tilde{A}_1$ & \vspace{1mm}
\resizebox{4cm}{!}{
\xy
@={(0,0),(10,10),(0,20),(20,10),(30,10),(40,0),(40,20),(60,20),(60,0),(50,10),(70,20),(70,0)}@@{*{\bullet}};
(0,0)*{};(10,10)*{}**\dir{-};
(0,20)*{};(10,10)*{}**\dir{-};
(30,10)*{};(10,10)*{}**\dir{-};
(30,10)*{};(40,20)*{}**\dir{-};
(30,10)*{};(40,0)*{}**\dir{-};
(50,10)*{};(40,20)*{}**\dir2{-};
(50,10)*{};(60,20)*{}**\dir{-};
(50,10)*{};(60,0)*{}**\dir{-};
(70,20)*{};(60,20)*{}**\dir2{-};
(70,0)*{};(60,0)*{}**\dir2{-};
\endxy} \hspace{4mm}
\resizebox{4cm}{!}{
\xy
@={(0,0),(10,10),(0,20),(20,10),(30,10),(40,0),(40,20),(60,20),(60,0),(50,10),(70,20),(70,0)}@@{*{\bullet}};
(0,0)*{};(10,10)*{}**\dir{-};
(0,20)*{};(10,10)*{}**\dir{-};
(30,10)*{};(10,10)*{}**\dir{-};
(30,10)*{};(40,20)*{}**\dir{-};
(30,10)*{};(40,0)*{}**\dir{-};
(50,10)*{};(40,20)*{}**\dir{-};
(50,10)*{};(40,0)*{}**\dir{-};
(50,10)*{};(60,20)*{}**\dir{-};
(50,10)*{};(60,0)*{}**\dir{-};
(70,20)*{};(60,20)*{}**\dir2{-};
(70,0)*{};(60,0)*{}**\dir2{-};
\endxy} \\ \hline
$\tilde{D}_6 \oplus 2\tilde{A}_1 \oplus \tilde{A}_1$ &
\resizebox{4cm}{!}{
\xy
@={(0,0),(10,10),(0,20),(20,10),(30,10),(40,0),(40,20),(60,20),(60,0),(50,10),(70,20),(70,0)}@@{*{\bullet}};
(0,0)*{};(10,10)*{}**\dir{-};
(0,20)*{};(10,10)*{}**\dir{-};
(30,10)*{};(10,10)*{}**\dir{-};
(30,10)*{};(40,20)*{}**\dir{-};
(30,10)*{};(40,0)*{}**\dir{-};
(50,10)*{};(40,20)*{}**\dir{-};
(50,10)*{};(0,20)*{}**\crv{(50,35)};
(50,10)*{};(60,20)*{}**\dir{-};
(50,10)*{};(60,0)*{}**\dir2{-};
(70,20)*{};(60,20)*{}**\dir2{-};
(70,0)*{};(60,0)*{}**\dir2{-};
\endxy} \hspace{4mm}
\resizebox{4cm}{!}{
\xy
@={(0,0),(10,10),(0,20),(20,10),(30,10),(40,0),(40,20),(60,20),(60,0),(50,10),(70,20),(70,0)}@@{*{\bullet}};
(0,0)*{};(10,10)*{}**\dir{-};
(0,20)*{};(10,10)*{}**\dir{-};
(30,10)*{};(10,10)*{}**\dir{-};
(30,10)*{};(40,20)*{}**\dir{-};
(30,10)*{};(40,0)*{}**\dir{-};
(50,10)*{};(40,20)*{}**\dir{-};
(50,10)*{};(40,0)*{}**\dir{-};
(50,10)*{};(60,20)*{}**\dir{-};
(50,10)*{};(60,0)*{}**\dir{-};
(50,10)*{};(70,0)*{}**\dir{-};
(70,20)*{};(60,20)*{}**\dir2{-};
(70,0)*{};(60,0)*{}**\dir2{-};
\endxy}

\vspace{3mm}
\resizebox{4cm}{!}{
\xy
@={(0,0),(10,10),(0,20),(20,10),(30,10),(40,0),(40,20),(60,20),(60,0),(50,10),(70,20),(70,0)}@@{*{\bullet}};
(0,0)*{};(10,10)*{}**\dir{-};
(0,20)*{};(10,10)*{}**\dir{-};
(30,10)*{};(10,10)*{}**\dir{-};
(30,10)*{};(40,20)*{}**\dir{-};
(30,10)*{};(40,0)*{}**\dir{-};
(50,10)*{};(40,20)*{}**\dir2{-};
(50,10)*{};(60,20)*{}**\dir{-};
(50,10)*{};(60,0)*{}**\dir{-};
(50,10)*{};(70,0)*{}**\dir{-};
(70,20)*{};(60,20)*{}**\dir2{-};
(70,0)*{};(60,0)*{}**\dir2{-};
\endxy} \hspace{4mm}
\resizebox{4cm}{!}{
\xy
@={(0,0),(10,10),(0,20),(20,10),(30,10),(40,0),(40,20),(60,20),(60,0),(50,10),(70,20),(70,0)}@@{*{\bullet}};
(0,0)*{};(10,10)*{}**\dir{-};
(0,20)*{};(10,10)*{}**\dir{-};
(30,10)*{};(10,10)*{}**\dir{-};
(30,10)*{};(40,20)*{}**\dir{-};
(30,10)*{};(40,0)*{}**\dir{-};
(50,10)*{};(40,20)*{}**\dir2{-};
(50,10)*{};(60,20)*{}**\dir{-};
(50,10)*{};(60,0)*{}**\dir2{-};
(70,20)*{};(60,20)*{}**\dir2{-};
(70,0)*{};(60,0)*{}**\dir2{-};
\endxy} \\ \hline
$2\tilde{A}_7 \oplus \tilde{A}_1$ & \vspace{1mm}
\resizebox{!}{1.5cm}{
\xy
@={(0,0),(0,10),(0,20),(0,30),(10,0),(10,10),(10,20),(10,30),(20,0),(30,0),(40,0)}@@{*{\bullet}};
(0,0)*{};(0,30)*{}**\dir{-};
(10,0)*{};(10,30)*{}**\dir{-};
(0,0)*{};(10,0)*{}**\dir{-};
(10,30)*{};(0,30)*{}**\dir{-};
(10,0)*{};(20,0)*{}**\dir{-};
(30,0)*{};(20,0)*{}**\dir2{-};
(30,0)*{};(40,0)*{}**\dir2{-};
\endxy}\\ \hline
$\tilde{A}_7 \oplus 2\tilde{A}_1$ & \vspace{1mm}
\resizebox{!}{1.5cm}{
\xy
@={(0,0),(0,10),(0,20),(0,30),(10,0),(10,10),(10,20),(10,30),(20,0),(30,0),(40,0)}@@{*{\bullet}};
(0,0)*{};(0,30)*{}**\dir{-};
(10,0)*{};(10,30)*{}**\dir{-};
(0,0)*{};(10,0)*{}**\dir{-};
(10,30)*{};(0,30)*{}**\dir{-};
(10,0)*{};(20,0)*{}**\dir{-};
(10,10)*{};(20,0)*{}**\dir{-};
(30,0)*{};(20,0)*{}**\dir{-};
(30,0)*{};(40,0)*{}**\dir2{-};
\endxy} \hspace{4mm}
\resizebox{!}{1.5cm}{
\xy
@={(0,0),(0,10),(0,20),(0,30),(10,0),(10,10),(10,20),(10,30),(20,0),(30,0),(40,0)}@@{*{\bullet}};
(0,0)*{};(0,30)*{}**\dir{-};
(10,0)*{};(10,30)*{}**\dir{-};
(0,0)*{};(10,0)*{}**\dir{-};
(10,30)*{};(0,30)*{}**\dir{-};
(10,0)*{};(20,0)*{}**\dir2{-};
(30,0)*{};(20,0)*{}**\dir{-};
(30,0)*{};(40,0)*{}**\dir2{-};
\endxy} \\ \hline
$2\tilde{A}_4 \oplus \tilde{A}_4$ & \vspace{1mm}
\resizebox{!}{1.2cm}{
\xy
@={(0,0),(0,10),(10,20),(20,10),(20,0),(40,0),(40,10),(50,20),(60,10),(60,0),(30,0)}@@{*{\bullet}};
(0,0)*{};(30,0)*{}**\dir{-};
(0,0)*{};(0,10)*{}**\dir{-};
(20,10)*{};(20,0)*{}**\dir{-};
(10,20)*{};(0,10)*{}**\dir{-};
(20,10)*{};(10,20)*{}**\dir{-};
(40,0)*{};(30,0)*{}**\dir2{-};
(40,0)*{};(60,0)*{}**\dir{-};
(60,0)*{};(60,10)*{}**\dir{-};
(40,0)*{};(40,10)*{}**\dir{-};
(50,20)*{};(60,10)*{}**\dir{-};
(50,20)*{};(40,10)*{}**\dir{-};
\endxy} \\ \hline
$2\tilde{A}_5 \oplus \tilde{A}_2 \oplus 2\tilde{A}_1$ & \vspace{1mm}
\resizebox{!}{1.2cm}{
\xy
@={(0,0),(0,10),(0,20),(10,0),(10,10),(10,20),(20,0),(30,0),(40,0),(20,10),(20,20),(30,20)}@@{*{\bullet}};
(0,0)*{};(0,20)*{}**\dir{-};
(10,0)*{};(10,20)*{}**\dir{-};
(0,0)*{};(10,0)*{}**\dir{-};
(0,20)*{};(10,20)*{}**\dir{-};
(0,0)*{};(30,0)*{}**\dir{-};
(20,10)*{};(20,20)*{}**\dir{-};
(20,0)*{};(20,10)*{}**\dir2{-};
(20,20)*{};(30,20)*{}**\dir{-};
(30,20)*{};(20,10)*{}**\dir{-};
(30,0)*{};(40,0)*{}**\dir2{-};
\endxy
} \\ \hline
$\tilde{A}_5 \oplus 2\tilde{A}_2 \oplus 2\tilde{A}_1$ & \vspace{1mm}
\resizebox{!}{1.2cm}{
\xy
@={(0,0),(0,10),(0,20),(10,0),(10,10),(10,20),(20,0),(30,0),(40,0),(20,10),(20,20),(30,20)}@@{*{\bullet}};
(0,0)*{};(0,20)*{}**\dir{-};
(10,0)*{};(10,20)*{}**\dir{-};
(0,0)*{};(10,0)*{}**\dir{-};
(0,20)*{};(10,20)*{}**\dir{-};
(0,0)*{};(10,0)*{}**\dir{-};
(20,0)*{};(30,0)*{}**\dir{-};
(20,0)*{};(10,0)*{}**\dir2{-};
(20,10)*{};(20,20)*{}**\dir{-};
(20,0)*{};(20,10)*{}**\dir{-};
(20,20)*{};(30,20)*{}**\dir{-};
(30,20)*{};(20,10)*{}**\dir{-};
(30,0)*{};(40,0)*{}**\dir2{-};
\endxy
}  \\ \hline
$2\tilde{A}_5 \oplus \tilde{A}_2 \oplus \tilde{A}_1$ & \vspace{1mm}
\resizebox{!}{1.2cm}{
\xy
@={(0,0),(0,10),(0,20),(10,0),(10,10),(10,20),(20,0),(30,0),(40,0),(20,10),(20,20),(30,20)}@@{*{\bullet}};
(0,0)*{};(0,20)*{}**\dir{-};
(10,0)*{};(10,20)*{}**\dir{-};
(0,0)*{};(10,0)*{}**\dir{-};
(0,20)*{};(10,20)*{}**\dir{-};
(0,0)*{};(30,0)*{}**\dir{-};
(40,0)*{};(20,0)*{}**\crv{(30,10)};
(20,10)*{};(20,20)*{}**\dir{-};
(20,0)*{};(20,10)*{}**\dir2{-};
(20,20)*{};(30,20)*{}**\dir{-};
(30,20)*{};(20,10)*{}**\dir{-};
(30,0)*{};(40,0)*{}**\dir2{-};
\endxy
} \\ \hline
\end{tabular}
\end{table}
\clearpage
\vspace{-2cm}

\begin{table}[!htbp]
\begin{tabular}{|>{\centering\arraybackslash}m{3.5cm}|>{\centering\arraybackslash}m{10cm}|} \hline
$\tilde{A}_5 \oplus 2\tilde{A}_2 \oplus \tilde{A}_1$ & \vspace{1mm}
\resizebox{!}{1.2cm}{
\xy
@={(0,0),(0,10),(0,20),(10,0),(10,10),(10,20),(20,0),(30,0),(40,0),(20,10),(20,20),(30,20)}@@{*{\bullet}};
(0,0)*{};(0,20)*{}**\dir{-};
(10,0)*{};(10,20)*{}**\dir{-};
(0,0)*{};(10,0)*{}**\dir{-};
(0,20)*{};(10,20)*{}**\dir{-};
(0,0)*{};(10,0)*{}**\dir{-};
(20,0)*{};(10,0)*{}**\dir2{-};
(30,0)*{};(20,0)*{}**\dir2{-};
(20,10)*{};(20,20)*{}**\dir{-};
(20,0)*{};(20,10)*{}**\dir{-};
(20,20)*{};(30,20)*{}**\dir{-};
(30,20)*{};(20,10)*{}**\dir{-};
(30,0)*{};(40,0)*{}**\dir2{-};
\endxy} \hspace{4mm}
\resizebox{!}{1.2cm}{
\xy
(-3,0)*{};
@={(0,0),(0,10),(0,20),(10,0),(10,10),(10,20),(20,0),(30,0),(40,0),(20,10),(20,20),(30,20)}@@{*{\bullet}};
(0,0)*{};(0,20)*{}**\dir{-};
(10,0)*{};(10,20)*{}**\dir{-};
(0,0)*{};(10,0)*{}**\dir{-};
(0,20)*{};(10,20)*{}**\dir{-};
(0,0)*{};(10,0)*{}**\dir{-};
(20,0)*{};(10,0)*{}**\dir2{-};
(30,0)*{};(20,0)*{}**\dir{-};
(40,0)*{};(20,0)*{}**\crv{(30,10)};
(20,10)*{};(20,20)*{}**\dir{-};
(20,0)*{};(20,10)*{}**\dir{-};
(20,20)*{};(30,20)*{}**\dir{-};
(30,20)*{};(20,10)*{}**\dir{-};
(30,0)*{};(40,0)*{}**\dir2{-};
\endxy} \hspace{4mm}
\resizebox{!}{1.2cm}{
\xy
@={(0,0),(0,10),(0,20),(10,0),(10,10),(10,20),(20,0),(30,0),(40,0),(20,10),(20,20),(30,20)}@@{*{\bullet}};
(0,0)*{};(0,20)*{}**\dir{-};
(10,0)*{};(10,20)*{}**\dir{-};
(0,0)*{};(10,0)*{}**\dir{-};
(0,20)*{};(10,20)*{}**\dir{-};
(0,0)*{};(10,0)*{}**\dir{-};
(20,0)*{};(10,0)*{}**\dir{-};
(20,0)*{};(10,10)*{}**\dir{-};
(30,0)*{};(20,0)*{}**\dir2{-};
(20,10)*{};(20,20)*{}**\dir{-};
(20,0)*{};(20,10)*{}**\dir{-};
(20,20)*{};(30,20)*{}**\dir{-};
(30,20)*{};(20,10)*{}**\dir{-};
(30,0)*{};(40,0)*{}**\dir2{-};
\endxy} \\ \hline
$\tilde{A}_5 \oplus \tilde{A}_2 \oplus 2\tilde{A}_1$ & \vspace{1mm}
\resizebox{!}{1.4cm}{
\xy
(-2,0)*{};
@={(0,0),(0,10),(0,20),(10,0),(10,10),(10,20),(20,0),(30,0),(40,0),(20,10),(20,20),(30,20)}@@{*{\bullet}};
(0,0)*{};(0,20)*{}**\dir{-};
(10,0)*{};(10,20)*{}**\dir{-};
(0,0)*{};(10,0)*{}**\dir{-};
(0,20)*{};(10,20)*{}**\dir{-};
(0,0)*{};(10,0)*{}**\dir{-};
(20,0)*{};(10,0)*{}**\dir{-};
(30,0)*{};(20,0)*{}**\dir{-};
(20,10)*{};(20,20)*{}**\dir{-};
(20,0)*{};(20,10)*{}**\dir2{-};
(20,20)*{};(30,20)*{}**\dir{-};
(30,20)*{};(20,10)*{}**\dir{-};
(30,0)*{};(40,0)*{}**\dir2{-};
(20,0)*{};(0,20)*{}**\crv{(15,30)};
\endxy
} \hspace{4mm}
\resizebox{!}{1.2cm}{
\xy
@={(0,0),(0,10),(0,20),(10,0),(10,10),(10,20),(20,0),(30,0),(40,0),(20,10),(20,20),(30,20)}@@{*{\bullet}};
(0,0)*{};(0,20)*{}**\dir{-};
(10,0)*{};(10,20)*{}**\dir{-};
(0,0)*{};(10,0)*{}**\dir{-};
(0,20)*{};(10,20)*{}**\dir{-};
(0,0)*{};(10,0)*{}**\dir{-};
(20,0)*{};(10,0)*{}**\dir{-};
(20,0)*{};(10,10)*{}**\dir{-};
(30,0)*{};(20,0)*{}**\dir{-};
(20,10)*{};(20,20)*{}**\dir{-};
(20,0)*{};(20,10)*{}**\dir2{-};
(20,20)*{};(30,20)*{}**\dir{-};
(30,20)*{};(20,10)*{}**\dir{-};
(30,0)*{};(40,0)*{}**\dir2{-};
\endxy} \vspace{1mm}

\resizebox{!}{1.2cm}{
\xy
@={(0,0),(0,10),(0,20),(10,0),(10,10),(10,20),(20,0),(30,0),(40,0),(20,10),(20,20),(30,20)}@@{*{\bullet}};
(0,0)*{};(0,20)*{}**\dir{-};
(10,0)*{};(10,20)*{}**\dir{-};
(0,0)*{};(10,0)*{}**\dir{-};
(0,20)*{};(10,20)*{}**\dir{-};
(0,0)*{};(10,0)*{}**\dir{-};
(20,0)*{};(10,0)*{}**\dir2{-};
(20,0)*{};(30,20)*{}**\dir{-};
(30,0)*{};(20,0)*{}**\dir{-};
(20,10)*{};(20,20)*{}**\dir{-};
(20,0)*{};(20,10)*{}**\dir{-};
(20,20)*{};(30,20)*{}**\dir{-};
(30,20)*{};(20,10)*{}**\dir{-};
(30,0)*{};(40,0)*{}**\dir2{-};
\endxy
} \hspace{4mm}
\resizebox{!}{1.2cm}{
\xy
@={(0,0),(0,10),(0,20),(10,0),(10,10),(10,20),(20,0),(30,0),(40,0),(20,10),(20,20),(30,20)}@@{*{\bullet}};
(0,0)*{};(0,20)*{}**\dir{-};
(10,0)*{};(10,20)*{}**\dir{-};
(0,0)*{};(10,0)*{}**\dir{-};
(0,20)*{};(10,20)*{}**\dir{-};
(0,0)*{};(10,0)*{}**\dir{-};
(20,0)*{};(10,0)*{}**\dir2{-};
(30,0)*{};(20,0)*{}**\dir{-};
(20,10)*{};(20,20)*{}**\dir{-};
(20,0)*{};(20,10)*{}**\dir2{-};
(20,20)*{};(30,20)*{}**\dir{-};
(30,20)*{};(20,10)*{}**\dir{-};
(30,0)*{};(40,0)*{}**\dir2{-};
\endxy}
 \\ \hline
$ 2\tilde{A}_3 \oplus 2\tilde{A}_3 \oplus \tilde{A}_1 \oplus \tilde{A}_1$ & \vspace{1mm}
\resizebox{!}{1.6cm}{
\xy
@={(0,40),(10,40),(30,40),(40,40),(0,30),(10,30),(30,30),(40,30),(20,20),(0,10),(10,10),(30,10),(40,10)}@@{*{\bullet}};
(0,40)*{};(10,40)*{}**\dir{-};
(30,40)*{};(40,40)*{}**\dir{-};
(0,30)*{};(10,30)*{}**\dir{-};
(30,30)*{};(40,30)*{}**\dir{-};
(0,40)*{};(0,30)*{}**\dir{-};
(10,30)*{};(10,40)*{}**\dir{-};
(30,30)*{};(30,40)*{}**\dir{-};
(40,30)*{};(40,40)*{}**\dir{-};
(10,30)*{};(20,20)*{}**\dir{-};
(30,30)*{};(20,20)*{}**\dir{-};
(10,10)*{};(20,20)*{}**\dir2{-};
(30,10)*{};(20,20)*{}**\dir2{-};
(10,10)*{};(0,10)*{}**\dir2{-};
(30,10)*{};(40,10)*{}**\dir2{-};
\endxy
} \\ \hline
$\tilde{A}_3 \oplus \tilde{A}_3 \oplus 2\tilde{A}_1 \oplus 2\tilde{A}_1$ & \vspace{1mm}
\resizebox{!}{1.6cm}{
\xy
(-7,10)*{,};
@={(0,40),(10,40),(30,40),(40,40),(0,30),(10,30),(30,30),(40,30),(20,20),(0,10),(10,10),(30,10),(40,10)}@@{*{\bullet}};
(0,40)*{};(10,40)*{}**\dir{-};
(30,40)*{};(40,40)*{}**\dir{-};
(0,30)*{};(10,30)*{}**\dir{-};
(30,30)*{};(40,30)*{}**\dir{-};
(0,40)*{};(0,30)*{}**\dir{-};
(10,30)*{};(10,40)*{}**\dir{-};
(30,30)*{};(30,40)*{}**\dir{-};
(40,30)*{};(40,40)*{}**\dir{-};
(10,30)*{};(20,20)*{}**\dir2{-};
(40,30)*{};(20,20)*{}**\dir{-};
(30,40)*{};(20,20)*{}**\dir{-};
(10,10)*{};(20,20)*{}**\dir{-};
(30,10)*{};(20,20)*{}**\dir{-};
(10,10)*{};(0,10)*{}**\dir2{-};
(30,10)*{};(40,10)*{}**\dir2{-};
\endxy}  \hspace{4mm}
\resizebox{!}{1.6cm}{
\xy
@={(0,40),(10,40),(30,40),(40,40),(0,30),(10,30),(30,30),(40,30),(20,20),(0,10),(10,10),(30,10),(40,10)}@@{*{\bullet}};
(0,40)*{};(10,40)*{}**\dir{-};
(30,40)*{};(40,40)*{}**\dir{-};
(0,30)*{};(10,30)*{}**\dir{-};
(30,30)*{};(40,30)*{}**\dir{-};
(0,40)*{};(0,30)*{}**\dir{-};
(10,30)*{};(10,40)*{}**\dir{-};
(30,30)*{};(30,40)*{}**\dir{-};
(40,30)*{};(40,40)*{}**\dir{-};
(10,30)*{};(20,20)*{}**\dir2{-};
(30,30)*{};(20,20)*{}**\dir2{-};
(10,10)*{};(20,20)*{}**\dir{-};
(30,10)*{};(20,20)*{}**\dir{-};
(10,10)*{};(0,10)*{}**\dir2{-};
(30,10)*{};(40,10)*{}**\dir2{-};
\endxy}
\\ \hline
$2\tilde{A}_3 \oplus \tilde{A}_3 \oplus 2\tilde{A}_1 \oplus \tilde{A}_1$ & \vspace{1mm}
\resizebox{!}{1.6cm}{
\xy
@={(0,40),(10,40),(30,40),(40,40),(0,30),(10,30),(30,30),(40,30),(20,20),(0,10),(10,10),(30,10),(40,10)}@@{*{\bullet}};
(0,40)*{};(10,40)*{}**\dir{-};
(30,40)*{};(40,40)*{}**\dir{-};
(0,30)*{};(10,30)*{}**\dir{-};
(30,30)*{};(40,30)*{}**\dir{-};
(0,40)*{};(0,30)*{}**\dir{-};
(10,30)*{};(10,40)*{}**\dir{-};
(30,30)*{};(30,40)*{}**\dir{-};
(40,30)*{};(40,40)*{}**\dir{-};
(10,30)*{};(20,20)*{}**\dir{-};
(30,30)*{};(20,20)*{}**\dir2{-};
(10,10)*{};(20,20)*{}**\dir{-};
(30,10)*{};(20,20)*{}**\dir{-};
(40,10)*{};(20,20)*{}**\dir{-};
(10,10)*{};(0,10)*{}**\dir2{-};
(30,10)*{};(40,10)*{}**\dir2{-};
\endxy
} \\ \hline
$2\tilde{A}_3 \oplus \tilde{A}_3 \oplus \tilde{A}_1 \oplus \tilde{A}_1$ & \vspace{1mm}
\resizebox{!}{1.6cm}{
\xy
@={(0,40),(10,40),(30,40),(40,40),(0,30),(10,30),(30,30),(40,30),(20,20),(0,10),(10,10),(30,10),(40,10)}@@{*{\bullet}};
(0,40)*{};(10,40)*{}**\dir{-};
(30,40)*{};(40,40)*{}**\dir{-};
(0,30)*{};(10,30)*{}**\dir{-};
(30,30)*{};(40,30)*{}**\dir{-};
(0,40)*{};(0,30)*{}**\dir{-};
(10,30)*{};(10,40)*{}**\dir{-};
(30,30)*{};(30,40)*{}**\dir{-};
(40,30)*{};(40,40)*{}**\dir{-};
(10,30)*{};(20,20)*{}**\dir{-};
(30,30)*{};(20,20)*{}**\dir2{-};
(10,10)*{};(20,20)*{}**\dir2{-};
(30,10)*{};(20,20)*{}**\dir2{-};
(10,10)*{};(0,10)*{}**\dir2{-};
(30,10)*{};(40,10)*{}**\dir2{-};
\endxy} \hspace{4mm}
\resizebox{!}{1.6cm}{
\xy
@={(0,40),(10,40),(30,40),(40,40),(0,30),(10,30),(30,30),(40,30),(20,20),(0,10),(10,10),(30,10),(40,10)}@@{*{\bullet}};
(0,40)*{};(10,40)*{}**\dir{-};
(30,40)*{};(40,40)*{}**\dir{-};
(0,30)*{};(10,30)*{}**\dir{-};
(30,30)*{};(40,30)*{}**\dir{-};
(0,40)*{};(0,30)*{}**\dir{-};
(10,30)*{};(10,40)*{}**\dir{-};
(30,30)*{};(30,40)*{}**\dir{-};
(40,30)*{};(40,40)*{}**\dir{-};
(10,30)*{};(20,20)*{}**\dir{-};
(40,30)*{};(20,20)*{}**\dir{-};
(30,40)*{};(20,20)*{}**\dir{-};
(10,10)*{};(20,20)*{}**\dir2{-};
(30,10)*{};(20,20)*{}**\dir2{-};
(10,10)*{};(0,10)*{}**\dir2{-};
(30,10)*{};(40,10)*{}**\dir2{-};
\endxy} \hspace{4mm}
\resizebox{!}{1.6cm}{
\xy
@={(0,40),(10,40),(30,40),(40,40),(0,30),(10,30),(30,30),(40,30),(20,20),(0,10),(10,10),(30,10),(40,10)}@@{*{\bullet}};
(0,40)*{};(10,40)*{}**\dir{-};
(30,40)*{};(40,40)*{}**\dir{-};
(0,30)*{};(10,30)*{}**\dir{-};
(30,30)*{};(40,30)*{}**\dir{-};
(0,40)*{};(0,30)*{}**\dir{-};
(10,30)*{};(10,40)*{}**\dir{-};
(30,30)*{};(30,40)*{}**\dir{-};
(40,30)*{};(40,40)*{}**\dir{-};
(10,30)*{};(20,20)*{}**\dir{-};
(30,30)*{};(20,20)*{}**\dir2{-};
(40,10)*{};(20,20)*{}**\dir{-};
(0,10)*{};(20,20)*{}**\dir{-};
(10,10)*{};(20,20)*{}**\dir{-};
(30,10)*{};(20,20)*{}**\dir{-};
(10,10)*{};(0,10)*{}**\dir2{-};
(30,10)*{};(40,10)*{}**\dir2{-};
\endxy
} \\ \hline
$\tilde{A}_3 \oplus \tilde{A}_3 \oplus 2\tilde{A}_1 \oplus \tilde{A}_1$ & \vspace{1mm}
\resizebox{!}{1.6cm}{
\xy
@={(0,40),(10,40),(30,40),(40,40),(0,30),(10,30),(30,30),(40,30),(20,20),(0,10),(10,10),(30,10),(40,10)}@@{*{\bullet}};
(0,40)*{};(10,40)*{}**\dir{-};
(30,40)*{};(40,40)*{}**\dir{-};
(0,30)*{};(10,30)*{}**\dir{-};
(30,30)*{};(40,30)*{}**\dir{-};
(0,40)*{};(0,30)*{}**\dir{-};
(10,30)*{};(10,40)*{}**\dir{-};
(30,30)*{};(30,40)*{}**\dir{-};
(40,30)*{};(40,40)*{}**\dir{-};
(0,30)*{};(20,20)*{}**\dir{-};
(30,30)*{};(20,20)*{}**\dir2{-};
(10,10)*{};(20,20)*{}**\dir{-};
(30,10)*{};(20,20)*{}**\dir{-};
(40,10)*{};(20,20)*{}**\dir{-};
(10,40)*{};(20,20)*{}**\dir{-};
(10,10)*{};(0,10)*{}**\dir2{-};
(30,10)*{};(40,10)*{}**\dir2{-};
\endxy} \hspace{4mm}
\resizebox{!}{1.6cm}{
\xy
@={(0,40),(10,40),(30,40),(40,40),(0,30),(10,30),(30,30),(40,30),(20,20),(0,10),(10,10),(30,10),(40,10)}@@{*{\bullet}};
(0,40)*{};(10,40)*{}**\dir{-};
(30,40)*{};(40,40)*{}**\dir{-};
(0,30)*{};(10,30)*{}**\dir{-};
(30,30)*{};(40,30)*{}**\dir{-};
(0,40)*{};(0,30)*{}**\dir{-};
(10,30)*{};(10,40)*{}**\dir{-};
(30,30)*{};(30,40)*{}**\dir{-};
(40,30)*{};(40,40)*{}**\dir{-};
(10,30)*{};(20,20)*{}**\dir2{-};
(30,30)*{};(20,20)*{}**\dir2{-};
(10,10)*{};(20,20)*{}**\dir{-};
(30,10)*{};(20,20)*{}**\dir{-};
(40,10)*{};(20,20)*{}**\dir{-};
(10,10)*{};(0,10)*{}**\dir2{-};
(30,10)*{};(40,10)*{}**\dir2{-};
\endxy}  \vspace{2mm}

\resizebox{!}{1.6cm}{
\xy
@={(0,40),(10,40),(30,40),(40,40),(0,30),(10,30),(30,30),(40,30),(20,20),(0,10),(10,10),(30,10),(40,10)}@@{*{\bullet}};
(0,40)*{};(10,40)*{}**\dir{-};
(30,40)*{};(40,40)*{}**\dir{-};
(0,30)*{};(10,30)*{}**\dir{-};
(30,30)*{};(40,30)*{}**\dir{-};
(0,40)*{};(0,30)*{}**\dir{-};
(10,30)*{};(10,40)*{}**\dir{-};
(30,30)*{};(30,40)*{}**\dir{-};
(40,30)*{};(40,40)*{}**\dir{-};
(10,30)*{};(20,20)*{}**\dir{-};
(30,30)*{};(20,20)*{}**\dir2{-};
(10,10)*{};(20,20)*{}**\dir{-};
(30,10)*{};(20,20)*{}**\dir2{-};
(10,40)*{};(20,20)*{}**\dir{-};
(10,10)*{};(0,10)*{}**\dir2{-};
(30,10)*{};(40,10)*{}**\dir2{-};
\endxy} \hspace{4mm}
\resizebox{!}{1.6cm}{
\xy
@={(0,40),(10,40),(30,40),(40,40),(0,30),(10,30),(30,30),(40,30),(20,20),(0,10),(10,10),(30,10),(40,10)}@@{*{\bullet}};
(0,40)*{};(10,40)*{}**\dir{-};
(30,40)*{};(40,40)*{}**\dir{-};
(0,30)*{};(10,30)*{}**\dir{-};
(30,30)*{};(40,30)*{}**\dir{-};
(0,40)*{};(0,30)*{}**\dir{-};
(10,30)*{};(10,40)*{}**\dir{-};
(30,30)*{};(30,40)*{}**\dir{-};
(40,30)*{};(40,40)*{}**\dir{-};
(10,30)*{};(20,20)*{}**\dir2{-};
(30,30)*{};(20,20)*{}**\dir2{-};
(10,10)*{};(20,20)*{}**\dir{-};
(30,10)*{};(20,20)*{}**\dir2{-};
(10,10)*{};(0,10)*{}**\dir2{-};
(30,10)*{};(40,10)*{}**\dir2{-};
\endxy}
\\ \hline
$2\tilde{A}_2 \oplus 2\tilde{A}_2 \oplus \tilde{A}_2 \oplus \tilde{A}_2$ & \vspace{1mm}
\resizebox{!}{1.6cm}{
\xy
@={(5,40),(35,40),(0,30),(10,30),(30,30),(40,30),(20,20),(0,10),(10,10),(30,10),(40,10),(5,20),(35,20)}@@{*{\bullet}};
(0,30)*{};(10,30)*{}**\dir{-};
(10,30)*{};(5,40)*{}**\dir{-};
(0,30)*{};(5,40)*{}**\dir{-};
(30,30)*{};(40,30)*{}**\dir{-};
(40,30)*{};(35,40)*{}**\dir{-};
(30,30)*{};(35,40)*{}**\dir{-};
(30,10)*{};(40,10)*{}**\dir{-};
(40,10)*{};(35,20)*{}**\dir{-};
(30,10)*{};(35,20)*{}**\dir{-};
(0,10)*{};(10,10)*{}**\dir{-};
(10,10)*{};(5,20)*{}**\dir{-};
(0,10)*{};(5,20)*{}**\dir{-};
(10,10)*{};(20,20)*{}**\dir{-};
(30,10)*{};(20,20)*{}**\dir{-};
(30,30)*{};(20,20)*{}**\dir2{-};
(10,30)*{};(20,20)*{}**\dir2{-};
\endxy
} \\ \hline
$2\tilde{A}_2 \oplus \tilde{A}_2 \oplus \tilde{A}_2 \oplus \tilde{A}_2$ & \vspace{1mm}
\resizebox{!}{1.6cm}{
\xy
@={(5,40),(35,40),(0,30),(10,30),(30,30),(40,30),(20,20),(0,10),(10,10),(30,10),(40,10),(5,20),(35,20)}@@{*{\bullet}};
(0,30)*{};(10,30)*{}**\dir{-};
(10,30)*{};(5,40)*{}**\dir{-};
(0,30)*{};(5,40)*{}**\dir{-};
(30,30)*{};(40,30)*{}**\dir{-};
(40,30)*{};(35,40)*{}**\dir{-};
(30,30)*{};(35,40)*{}**\dir{-};
(30,10)*{};(40,10)*{}**\dir{-};
(40,10)*{};(35,20)*{}**\dir{-};
(30,10)*{};(35,20)*{}**\dir{-};
(0,10)*{};(10,10)*{}**\dir{-};
(10,10)*{};(5,20)*{}**\dir{-};
(0,10)*{};(5,20)*{}**\dir{-};
(10,10)*{};(20,20)*{}**\dir{-};
(30,10)*{};(20,20)*{}**\dir2{-};
(30,30)*{};(20,20)*{}**\dir2{-};
(10,30)*{};(20,20)*{}**\dir2{-};
\endxy} \hspace{4mm}
\resizebox{!}{1.6cm}{
\xy
@={(5,40),(35,40),(0,30),(10,30),(30,30),(40,30),(20,20),(0,10),(10,10),(30,10),(40,10),(5,20),(35,20)}@@{*{\bullet}};
(0,30)*{};(10,30)*{}**\dir{-};
(10,30)*{};(5,40)*{}**\dir{-};
(0,30)*{};(5,40)*{}**\dir{-};
(30,30)*{};(40,30)*{}**\dir{-};
(40,30)*{};(35,40)*{}**\dir{-};
(30,30)*{};(35,40)*{}**\dir{-};
(30,10)*{};(40,10)*{}**\dir{-};
(40,10)*{};(35,20)*{}**\dir{-};
(30,10)*{};(35,20)*{}**\dir{-};
(0,10)*{};(10,10)*{}**\dir{-};
(10,10)*{};(5,20)*{}**\dir{-};
(0,10)*{};(5,20)*{}**\dir{-};
(10,10)*{};(20,20)*{}**\dir{-};
(30,10)*{};(20,20)*{}**\dir{-};
(35,20)*{};(20,20)*{}**\dir{-};
(30,30)*{};(20,20)*{}**\dir2{-};
(10,30)*{};(20,20)*{}**\dir2{-};
\endxy} \\ \hline
\end{tabular}
\caption{Admissible fiber-bisection configurations for extremal fibrations}
\label{admissiblefiberbisection}
\end{table}
\clearpage

\end{lemma}

\begin{remark}
In fact, many of these admissible fiber-bisection configurations are realizable over the complex numbers (see \cite{Master}).
\end{remark}

From these tables, we can deduce the following improvement of Lemma \ref{reductionstep1}.

\begin{corollary}
If an Enriques surface $X$ admits a special and extremal elliptic fibration, then $X$ is either of type $\II$ or it admits a special elliptic fibration with a double fiber of type $\I_2$.
\end{corollary}

\begin{proof}\label{reductionstep2}
By Lemma \ref{reductionstep1}, we know that $X$ admits an elliptic fibration with a double fiber of type $\I_n$ for some $n$. Almost every graph in Lemma \ref{admissibletypes} admits an $\tilde{A}_1$ subgraph and a vertex meeting this subgraph exactly once; the only exception is the critical subgraph for type $\II$. Hence, the claim follows.
\end{proof}

Before we start with the proof of Lemma \ref{graphlemma}, we need the following auxiliary result.

\begin{lemma}\label{noHesse}
There is no Enriques surface with a special elliptic fibration with singular fibers 
\begin{itemize}
\item $(\I_3,\I_3,\I_3,\I_3)$ such that two of the $\I_3$ fibers are multiple or
\item $(\I_6,\I_3,\I_2,\I_1)$ such that the $\I_3$ and $\I_2$ fibers are multiple.
\end{itemize}
\end{lemma}

\begin{proof}
We will only show the first claim; the second one is similar.
The claim is true if $\Char(k) \in \{2,3\}$, since there is no rational elliptic surface with singular fibers $(\I_3,\I_3,\I_3,\I_3)$ in characteristic $3$ and an elliptic fibration of an Enriques surface in characteristic $2$ cannot have two multiplicative double fibers.

Let us assume $\Char(k) \not \in \{2,3\}$. The rational elliptic surface $J(\pi)$ with singular fibers $(\I_3,\I_3,\I_3,\I_3)$ has the Weierstrass equation
\begin{equation}\label{jacobianHesse}
y^2 = x^3 + (-3t^4 +24t)x + 2t^6+40t^3-16.
\end{equation}
If an Enriques surface with this Jacobian and two double $\I_3$ fibers exists, it is covered by the base change of (\ref{jacobianHesse}) via $t \mapsto s^2 -1$.

A $J(\pi)$-Enriques section $N^- = (x(s),y(s))$ meets the fibers of $J(\pi)$ at $s = 0$ and at $s = \infty$ in a non-identity component and is $J(\sigma)$-anti-invariant, where $J(\sigma): s \mapsto -s$.
Since the singular point of the fiber at $s = 0$ (resp. $s  = \infty$) is $(-3,0)$ (resp. $(1,0)$), $N^-$ has the form
\begin{eqnarray*}
x &= -3 + x_2 s^2 + s^4 \\
y &= y_1s +y_3 s^3 + y_5 s^5.
\end{eqnarray*}
Plugging this into the base change of equation (\ref{jacobianHesse}), we additionally obtain $y_1 = y_5 = 0$, $y_3 = \pm 8$, $x_2 = -2$ and finally $144 = 0$, which is not allowed, since $\Char(k) \neq 2,3$.
\end{proof}

\begin{proof}[Proof of Lemma \ref{graphlemma}] (For a detailed explanation of how to add $(-2)$-curves using $jac_2$, see Section \ref{example}.)
By Corollary \ref{reductionstep2}, it suffices to check the admissible fiber-bisection configurations with a $2\tilde{A}_1$ component. We will treat them in the following order:

\vspace{3mm}
\centerline{
\begin{tabular}{|>{\centering\arraybackslash}m{3.5cm}|>{\centering\arraybackslash}m{8.5cm}|}
\hline
$\Gamma_1$ & \text{\# Admissible fiber-bisection configurations } \\ \hline \hline
\vspace{0.5mm} $\tilde{D}_6 \oplus 2\tilde{A}_1 \oplus 2\tilde{A}_1$ & \vspace{1mm} $2$ \\ [0.5mm] \hline
\vspace{0.5mm} $\tilde{D}_6 \oplus 2\tilde{A}_1 \oplus \tilde{A}_1$ &  \vspace{1mm}  $4$\\ [0.5mm] \hline
\vspace{0.5mm} $\tilde{E}_7 \oplus 2\tilde{A}_1$  & \vspace{1mm} $2$\\ [0.5mm] \hline
\vspace{0.5mm} $\tilde{A}_3 \oplus \tilde{A}_3 \oplus 2\tilde{A}_1 \oplus 2\tilde{A}_1$ & \vspace{1mm} $2$\\ [0.5mm] \hline
\vspace{0.5mm} $2\tilde{A}_3 \oplus \tilde{A}_3 \oplus 2\tilde{A}_1 \oplus \tilde{A}_1$ &\vspace{1mm} $1$\\ [0.5mm] \hline
\vspace{0.5mm} $\tilde{A}_3 \oplus \tilde{A}_3 \oplus 2\tilde{A}_1 \oplus \tilde{A}_1$ &\vspace{1mm} $4$\\[0.5mm]  \hline
\vspace{0.5mm} $2\tilde{A}_5 \oplus \tilde{A}_2 \oplus 2\tilde{A}_1$ & \vspace{1mm}$1$ \\ [0.5mm]\hline
\vspace{0.5mm} $\tilde{A}_5 \oplus \tilde{A}_2 \oplus 2\tilde{A}_1$ & \vspace{1mm}$4$ \\ [0.5mm]\hline
\vspace{0.5mm} $\tilde{A}_7 \oplus 2\tilde{A}_1$ & \vspace{1mm}$2$\\ [0.5mm]\hline
\end{tabular}
}

\begin{itemize}
\item $\Gamma_1 = \tilde{D}_6 \oplus 2\tilde{A}_1 \oplus 2\tilde{A}_1$
\begin{enumerate}[label = \alph*)]
\item Fiber-bisection configuration:

\centerline{
\xy
(30,10)*{};
@={(0,0),(10,10),(0,20),(20,10),(30,10),(40,0),(40,20),(60,20),(60,0),(50,10),(70,20),(70,0)}@@{*{\bullet}};
(0,0)*{};(10,10)*{}**\dir{-};
(0,20)*{};(10,10)*{}**\dir{-};
(30,10)*{};(10,10)*{}**\dir{-};
(30,10)*{};(40,20)*{}**\dir{-};
(30,10)*{};(40,0)*{}**\dir{-};
(50,10)*{};(40,20)*{}**\dir2{-};
(50,10)*{};(60,20)*{}**\dir{-};
(50,10)*{};(60,0)*{}**\dir{-};
(70,20)*{};(60,20)*{}**\dir2{-};
(70,0)*{};(60,0)*{}**\dir2{-};
\endxy}
After adding a bisection with $jac_2$, we find another special fibration with two double $\I_2$ fibers and bisection $N$ as follows, where the dotted rectangles mark the fibers:

\centerline{
\xy
(30,10)*{};
@={(0,0),(10,10),(0,20),(20,10),(30,10),(40,0),(40,20),(60,20),(60,0),(50,20),(50,0),(70,20),(70,0)}@@{*{\bullet}};
(0,0)*{};(10,10)*{}**\dir{-};
(0,20)*{};(10,10)*{}**\dir{-};
(30,10)*{};(10,10)*{}**\dir{-};
(30,10)*{};(40,20)*{}**\dir{-};
(30,10)*{};(40,0)*{}**\dir{-};
(50,20)*{};(40,20)*{}**\dir2{-};
(50,20)*{};(60,20)*{}**\dir{-};
(50,20)*{};(60,0)*{}**\dir{-};
(50,0)*{};(40,0)*{}**\dir2{-};
(50,0)*{};(60,20)*{}**\dir{-};
(50,0)*{};(60,0)*{}**\dir{-};
(70,20)*{};(60,20)*{}**\dir2{-};
(70,0)*{};(60,0)*{}**\dir2{-};
(37,23)*{};(53,23)*{}**\dir{--};
(37,17)*{};(53,17)*{}**\dir{--};
(37,17)*{};(37,23)*{}**\dir{--};
(53,17)*{};(53,23)*{}**\dir{--};
(37,3)*{};(53,3)*{}**\dir{--};
(37,-3)*{};(53,-3)*{}**\dir{--};
(37,-3)*{};(37,3)*{}**\dir{--};
(53,-3)*{};(53,3)*{}**\dir{--};
(30,13)*{N};
\endxy}

There is a $D_4$ diagram which is disjoint from the two $\tilde{A}_1$ subgraphs. By Table \ref{extremalrational}, the only extremal fibration with two singular fibers of type $\I_2$ and one singular fiber whose dual graph contains a $D_4$ is the one with singular fibers $(\I_2^*,\I_2,\I_2)$. However, the bisection $N$ cannot meet the $\I_2^*$ fiber in an admissible way, hence this fiber-bisection configuration does not occur.
\item Fiber-bisection configuration:

\centerline{
\xy
@={(0,0),(10,10),(0,20),(20,10),(30,10),(40,0),(40,20),(60,20),(60,0),(50,10),(70,20),(70,0)}@@{*{\bullet}};
(0,0)*{};(10,10)*{}**\dir{-};
(0,20)*{};(10,10)*{}**\dir{-};
(30,10)*{};(10,10)*{}**\dir{-};
(30,10)*{};(40,20)*{}**\dir{-};
(30,10)*{};(40,0)*{}**\dir{-};
(50,10)*{};(40,20)*{}**\dir{-};
(50,10)*{};(40,0)*{}**\dir{-};
(50,10)*{};(60,20)*{}**\dir{-};
(50,10)*{};(60,0)*{}**\dir{-};
(70,20)*{};(60,20)*{}**\dir2{-};
(70,0)*{};(60,0)*{}**\dir2{-};
\endxy}

After adding a bisection with $jac_2$, we find another special fibration with two double $\I_4$ fibers and bisection $N$ as follows:

\centerline{
\xy
@={(-10,10),(0,0),(10,10),(0,20),(20,10),(30,10),(40,0),(40,20),(60,20),(60,0),(50,10),(70,20),(70,0)}@@{*{\bullet}};
(0,0)*{};(-10,10)*{}**\dir{-};
(0,20)*{};(-10,10)*{}**\dir{-};
(0,0)*{};(10,10)*{}**\dir{-};
(0,20)*{};(10,10)*{}**\dir{-};
(30,10)*{};(10,10)*{}**\dir{-};
(30,10)*{};(40,20)*{}**\dir{-};
(30,10)*{};(40,0)*{}**\dir{-};
(50,10)*{};(40,20)*{}**\dir{-};
(50,10)*{};(40,0)*{}**\dir{-};
(50,10)*{};(60,20)*{}**\dir{-};
(50,10)*{};(60,0)*{}**\dir{-};
(70,20)*{};(60,20)*{}**\dir2{-};
(70,0)*{};(60,0)*{}**\dir2{-};
(60,20)*{};(-10,10)*{}**\crv{(0,40)};
(60,0)*{};(-10,10)*{}**\crv{(0,-20)};
(-13,-3)*{};(-13,23)*{}**\dir{--};
(-13,-3)*{};(13,-3)*{}**\dir{--};
(13,23)*{};(-13,23)*{}**\dir{--};
(13,-3)*{};(13,23)*{}**\dir{--};
(27,-3)*{};(27,23)*{}**\dir{--};
(27,-3)*{};(53,-3)*{}**\dir{--};
(53,23)*{};(27,23)*{}**\dir{--};
(53,-3)*{};(53,23)*{}**\dir{--};
(23,13)*{N};
\endxy}
\vspace{5mm}

By Table \ref{extremalrational}, the only extremal fibration with two singular fibers of type $\I_4$ is the one with singular fibers $(\I_4,\I_4,\I_2,\I_2)$ and the only admissible fiber-bisection configuration with $\Gamma_1 = 2\tilde{A}_3 \oplus 2\tilde{A}_3 \oplus \tilde{A}_1 \oplus \tilde{A}_1$ is the critical subgraph for type $\III$. 
\end{enumerate}

\item $\Gamma_1 = \tilde{D}_6 \oplus 2\tilde{A}_1 \oplus \tilde{A}_1$

\begin{enumerate}[label = \alph*)]
\item
Fiber-bisection configuration:

\centerline{
\xy
@={(0,0),(10,10),(0,20),(20,10),(30,10),(40,0),(40,20),(60,20),(60,0),(50,10),(70,20),(70,0)}@@{*{\bullet}};
(0,0)*{};(10,10)*{}**\dir{-};
(0,20)*{};(10,10)*{}**\dir{-};
(30,10)*{};(10,10)*{}**\dir{-};
(30,10)*{};(40,20)*{}**\dir{-};
(30,10)*{};(40,0)*{}**\dir{-};
(50,10)*{};(40,20)*{}**\dir{-};
(50,10)*{};(0,20)*{}**\crv{(50,35)};
(50,10)*{};(60,20)*{}**\dir{-};
(50,10)*{};(60,0)*{}**\dir2{-};
(70,20)*{};(60,20)*{}**\dir2{-};
(70,0)*{};(60,0)*{}**\dir2{-};
\endxy}

After adding a bisection corresponding to a $2$-torsion section via $jac_2$, we obtain another special fibration with double singular fibers $\I_6$ and $\I_2$ and bisection $N$ as follows:

\centerline{
\xy
@={(0,0),(10,10),(0,20),(20,10),(30,10),(40,0),(40,20),(60,20),(60,0),(50,10),(70,20),(70,0),(20,-10)}@@{*{\bullet}};
(0,0)*{};(10,10)*{}**\dir{-};
(0,20)*{};(10,10)*{}**\dir{-};
(30,10)*{};(10,10)*{}**\dir{-};
(30,10)*{};(40,20)*{}**\dir{-};
(30,10)*{};(40,0)*{}**\dir{-};
(50,10)*{};(40,20)*{}**\dir{-};
(50,10)*{};(0,20)*{}**\crv{(50,35)};
(50,10)*{};(60,20)*{}**\dir{-};
(50,10)*{};(60,0)*{}**\dir2{-};
(70,20)*{};(60,20)*{}**\dir2{-};
(70,0)*{};(60,0)*{}**\dir2{-};
(20,-10)*{};(60,20)*{}**\crv{(60,0)};
(20,-10)*{};(40,0)*{}**\dir{-};
(20,-10)*{};(0,0)*{}**\dir{-};
(20,-10)*{};(70,0)*{}**\crv{~**\dir2{-} (60,-10)};
(-3,13)*{};(43,13)*{}**\dir{--};
(-3,-13)*{};(43,-13)*{}**\dir{--};
(43,13)*{};(43,-13)*{}**\dir{--};
(-3,13)*{};(-3,-13)*{}**\dir{--};
(63,13)*{};(47,13)*{}**\dir{--};
(63,-3)*{};(47,-3)*{}**\dir{--};
(47,13)*{};(47,-3)*{}**\dir{--};
(63,13)*{};(63,-3)*{}**\dir{--};
(37,22)*{N};
\endxy}
\vspace{5mm}

The only admissible fiber-bisection configuration for such a fibration is the critical subgraph for type $\V$.

\item Fiber-bisection configuration:

\centerline{
\xy
@={(0,0),(10,10),(0,20),(20,10),(30,10),(40,0),(40,20),(60,20),(60,0),(50,10),(70,20),(70,0)}@@{*{\bullet}};
(0,0)*{};(10,10)*{}**\dir{-};
(0,20)*{};(10,10)*{}**\dir{-};
(30,10)*{};(10,10)*{}**\dir{-};
(30,10)*{};(40,20)*{}**\dir{-};
(30,10)*{};(40,0)*{}**\dir{-};
(50,10)*{};(40,20)*{}**\dir{-};
(50,10)*{};(40,0)*{}**\dir{-};
(50,10)*{};(60,20)*{}**\dir{-};
(50,10)*{};(60,0)*{}**\dir{-};
(50,10)*{};(70,0)*{}**\dir{-};
(70,20)*{};(60,20)*{}**\dir2{-};
(70,0)*{};(60,0)*{}**\dir2{-};
\endxy}

Adding another bisection corresponding to a $2$-torsion section via $jac_2$, we obtain another special fibration with two singular double fibers of type $\I_4$, giving the critical subgraph for type $\III$:

\centerline{
\xy
@={(0,0),(10,10),(0,20),(20,10),(30,10),(40,0),(40,20),(60,20),(60,0),(50,10),(70,20),(70,0),(-10,10)}@@{*{\bullet}};
(0,0)*{};(10,10)*{}**\dir{-};
(0,20)*{};(10,10)*{}**\dir{-};
(30,10)*{};(10,10)*{}**\dir{-};
(30,10)*{};(40,20)*{}**\dir{-};
(30,10)*{};(40,0)*{}**\dir{-};
(50,10)*{};(40,20)*{}**\dir{-};
(50,10)*{};(40,0)*{}**\dir{-};
(50,10)*{};(60,20)*{}**\dir{-};
(50,10)*{};(60,0)*{}**\dir{-};
(50,10)*{};(70,0)*{}**\dir{-};
(70,20)*{};(60,20)*{}**\dir2{-};
(70,0)*{};(60,0)*{}**\dir2{-};
(-10,10)*{};(60,20)*{}**\crv{(-10,40)};
(-10,10)*{};(0,20)*{}**\dir{-};
(-10,10)*{};(0,0)*{}**\dir{-};
(-10,10)*{};(60,0)*{}**\crv{(0,-20)};
(-10,10)*{};(70,0)*{}**\crv{(-10,-30)};
(-13,-3)*{};(-13,23)*{}**\dir{--};
(-13,-3)*{};(13,-3)*{}**\dir{--};
(13,23)*{};(-13,23)*{}**\dir{--};
(13,-3)*{};(13,23)*{}**\dir{--};
(27,-3)*{};(27,23)*{}**\dir{--};
(27,-3)*{};(53,-3)*{}**\dir{--};
(53,23)*{};(27,23)*{}**\dir{--};
(53,-3)*{};(53,23)*{}**\dir{--};
(20,13)*{N};
\endxy}
\vspace{5mm}

\item Fiber-bisection configuration:

\centerline{
\xy
@={(0,0),(10,10),(0,20),(20,10),(30,10),(40,0),(40,20),(60,20),(60,0),(50,10),(70,20),(70,0)}@@{*{\bullet}};
(0,0)*{};(10,10)*{}**\dir{-};
(0,20)*{};(10,10)*{}**\dir{-};
(30,10)*{};(10,10)*{}**\dir{-};
(30,10)*{};(40,20)*{}**\dir{-};
(30,10)*{};(40,0)*{}**\dir{-};
(50,10)*{};(40,20)*{}**\dir2{-};
(50,10)*{};(60,20)*{}**\dir{-};
(50,10)*{};(60,0)*{}**\dir{-};
(50,10)*{};(70,0)*{}**\dir{-};
(70,20)*{};(60,20)*{}**\dir2{-};
(70,0)*{};(60,0)*{}**\dir2{-};
\endxy}

Adding another bisection corresponding to a $2$-torsion section via $jac_2$, we obtain another special fibration with two singular double fibers of type $\I_2$, bisection $N$, and some fiber whose dual graph contains a $D_4$. The only extremal fibration satisfying this is the one with fibers $(\I_2^*,\I_2,\I_2)$ and we have already treated the cases where both $\I_2$ fibers are double.

\vspace{1mm}
\centerline{
\xy
@={(0,0),(10,10),(0,20),(20,10),(30,10),(40,0),(40,20),(60,20),(60,0),(50,20),(50,0),(70,20),(70,0)}@@{*{\bullet}};
(0,0)*{};(10,10)*{}**\dir{-};
(0,20)*{};(10,10)*{}**\dir{-};
(30,10)*{};(10,10)*{}**\dir{-};
(30,10)*{};(40,20)*{}**\dir{-};
(30,10)*{};(40,0)*{}**\dir{-};
(50,20)*{};(40,20)*{}**\dir2{-};
(50,20)*{};(60,20)*{}**\dir{-};
(50,20)*{};(60,0)*{}**\dir{-};
(50,20)*{};(70,0)*{}**\dir{-};
(70,20)*{};(60,20)*{}**\dir2{-};
(70,0)*{};(60,0)*{}**\dir2{-};
(50,0)*{};(40,0)*{}**\dir2{-};
(50,0)*{};(60,20)*{}**\dir{-};
(50,0)*{};(60,0)*{}**\dir{-};
(50,0)*{};(70,0)*{}**\crv{(60,-10)};
(37,23)*{};(53,23)*{}**\dir{--};
(37,17)*{};(53,17)*{}**\dir{--};
(37,17)*{};(37,23)*{}**\dir{--};
(53,17)*{};(53,23)*{}**\dir{--};
(37,3)*{};(53,3)*{}**\dir{--};
(37,-3)*{};(53,-3)*{}**\dir{--};
(37,-3)*{};(37,3)*{}**\dir{--};
(53,-3)*{};(53,3)*{}**\dir{--};
(30,13)*{N};
\endxy}

\item Fiber-bisection configuration:

\centerline{
\xy
@={(0,0),(10,10),(0,20),(20,10),(30,10),(40,0),(40,20),(60,20),(60,0),(50,10),(70,20),(70,0)}@@{*{\bullet}};
(0,0)*{};(10,10)*{}**\dir{-};
(0,20)*{};(10,10)*{}**\dir{-};
(30,10)*{};(10,10)*{}**\dir{-};
(30,10)*{};(40,20)*{}**\dir{-};
(30,10)*{};(40,0)*{}**\dir{-};
(50,10)*{};(40,20)*{}**\dir2{-};
(50,10)*{};(60,20)*{}**\dir{-};
(50,10)*{};(60,0)*{}**\dir2{-};
(70,20)*{};(60,20)*{}**\dir2{-};
(70,0)*{};(60,0)*{}**\dir2{-};
\endxy }

There is another special elliptic fibration with double fiber of type $\I_2$ as in the following figure:

\centerline{
\xy
@={(0,0),(10,10),(0,20),(20,10),(30,10),(40,0),(40,20),(60,20),(60,0),(50,10),(70,20),(70,0)}@@{*{\bullet}};
(0,0)*{};(10,10)*{}**\dir{-};
(0,20)*{};(10,10)*{}**\dir{-};
(30,10)*{};(10,10)*{}**\dir{-};
(30,10)*{};(40,20)*{}**\dir{-};
(30,10)*{};(40,0)*{}**\dir{-};
(50,10)*{};(40,20)*{}**\dir2{-};
(50,10)*{};(60,20)*{}**\dir{-};
(50,10)*{};(60,0)*{}**\dir2{-};
(70,20)*{};(60,20)*{}**\dir2{-};
(70,0)*{};(60,0)*{}**\dir2{-};
(37,23)*{};(53,23)*{}**\dir{--};
(37,23)*{};(37,7)*{}**\dir{--};
(37,7)*{};(53,7)*{}**\dir{--};
(53,23)*{};(53,7)*{}**\dir{--};
(30,13)*{N};
\endxy }

There is a $D_4$ diagram and three disjoint vertices, which are disjoint from the marked subgraph. The only extremal fibration whose dual graph of singular fibers contains these diagrams is the one with singular fibers $(\I_2^*,\I_2,\I_2)$. But the bisection $N$ meets the fibers in such a way, that the fiber-bisection configuration will be one of the configurations we have already treated.

\end{enumerate}

\item $\Gamma_1 = \tilde{E}_7 \oplus 2\tilde{A}_1$

\begin{enumerate}[label = \alph*)]

\item Fiber-bisection configuration:

\centerline{
\xy
@={(0,10),(10,10),(20,10),(30,10),(40,10),(50,10),(60,10),(30,0),(30,20),(30,30),(30,40)}@@{*{\bullet}};
(0,10)*{};(60,10)*{}**\dir{-};
(30,10)*{};(30,0)*{}**\dir{-};
(0,10)*{};(30,20)*{}**\dir{-};
(60,10)*{};(30,20)*{}**\dir{-};
(30,20)*{};(30,30)*{}**\dir{-};
(30,40)*{};(30,30)*{}**\dir2{-};
\endxy}
\vspace{1mm}

This is the critical subgraph for type $\I$.

\item Fiber-bisection configuration:

\centerline{
\xy
@={(0,10),(10,10),(20,10),(30,10),(40,10),(50,10),(60,10),(30,0),(30,20),(30,30),(30,40)}@@{*{\bullet}};
(0,10)*{};(60,10)*{}**\dir{-};
(30,10)*{};(30,0)*{}**\dir{-};
(60,10)*{};(30,20)*{}**\dir2{-};
(30,20)*{};(30,30)*{}**\dir{-};
(30,40)*{};(30,30)*{}**\dir2{-};
\endxy}
\vspace{1mm}

There is another special elliptic fibration with a double fiber of type $\I_2$ and a bisection $N$ as follows:

\centerline{
\xy
@={(0,10),(10,10),(20,10),(30,10),(40,10),(50,10),(60,10),(30,0),(30,20),(30,30),(30,40)}@@{*{\bullet}};
(0,10)*{};(60,10)*{}**\dir{-};
(30,10)*{};(30,0)*{}**\dir{-};
(60,10)*{};(30,20)*{}**\dir2{-};
(30,20)*{};(30,30)*{}**\dir{-};
(30,40)*{};(30,30)*{}**\dir2{-};
(60,8)*{};(63,10)*{}**\dir{--};
(30,22)*{};(27,20)*{}**\dir{--};
(60,8)*{};(27,20)*{}**\dir{--};
(30,22)*{};(63,10)*{}**\dir{--};
(50,7)*{N};
\endxy}
\vspace{1mm}

There is a $D_6$ diagram and an isolated vertex which are disjoint from the marked subgraph. Moreover, from the intersection behaviour of $N$, we can exclude the case that the new fibration has a singular fiber of type $\III^*$. The only extremal fibration satisfying these conditions is the one with singular fibers $(\I_2^*,\I_2,\I_2)$. We have already treated all fiber-bisection configurations for this fibration.

\end{enumerate}

\item $\Gamma_1 = \tilde{A}_3 \oplus \tilde{A}_3 \oplus 2\tilde{A}_1 \oplus 2\tilde{A}_1$

\begin{enumerate}[label = \alph*)]
\item Fiber-bisection configuration:

\vspace{1mm}
\centerline{
\xy
@={(0,40),(10,40),(30,40),(40,40),(0,30),(10,30),(30,30),(40,30),(20,20),(0,10),(10,10),(30,10),(40,10)}@@{*{\bullet}};
(0,40)*{};(10,40)*{}**\dir{-};
(30,40)*{};(40,40)*{}**\dir{-};
(0,30)*{};(10,30)*{}**\dir{-};
(30,30)*{};(40,30)*{}**\dir{-};
(0,40)*{};(0,30)*{}**\dir{-};
(10,30)*{};(10,40)*{}**\dir{-};
(30,30)*{};(30,40)*{}**\dir{-};
(40,30)*{};(40,40)*{}**\dir{-};
(10,30)*{};(20,20)*{}**\dir2{-};
(40,30)*{};(20,20)*{}**\dir{-};
(30,40)*{};(20,20)*{}**\dir{-};
(10,10)*{};(20,20)*{}**\dir{-};
(30,10)*{};(20,20)*{}**\dir{-};
(10,10)*{};(0,10)*{}**\dir2{-};
(30,10)*{};(40,10)*{}**\dir2{-};
\endxy}

This is the critical subgraph for type $\IV$.

\item Fiber-bisection configuration:

\vspace{1mm}
\centerline{
\xy
@={(0,40),(10,40),(30,40),(40,40),(0,30),(10,30),(30,30),(40,30),(20,20),(0,10),(10,10),(30,10),(40,10)}@@{*{\bullet}};
(0,40)*{};(10,40)*{}**\dir{-};
(30,40)*{};(40,40)*{}**\dir{-};
(0,30)*{};(10,30)*{}**\dir{-};
(30,30)*{};(40,30)*{}**\dir{-};
(0,40)*{};(0,30)*{}**\dir{-};
(10,30)*{};(10,40)*{}**\dir{-};
(30,30)*{};(30,40)*{}**\dir{-};
(40,30)*{};(40,40)*{}**\dir{-};
(10,30)*{};(20,20)*{}**\dir2{-};
(30,30)*{};(20,20)*{}**\dir2{-};
(10,10)*{};(20,20)*{}**\dir{-};
(30,10)*{};(20,20)*{}**\dir{-};
(10,10)*{};(0,10)*{}**\dir2{-};
(30,10)*{};(40,10)*{}**\dir2{-};
\endxy}

After adding bisections coming from $2$-torsion sections via $jac_2$, we obtain another (maybe non-special) fibration with two double $\I_2$ fibers as follows:

\centerline{
\xy
@={(0,40),(10,40),(30,40),(40,40),(0,30),(10,30),(30,30),(40,30),(20,25),(0,10),(10,10),(30,10),(40,10),(20,50),(-10,20),(50,20)}@@{*{\bullet}};
(0,40)*{};(10,40)*{}**\dir{-};
(30,40)*{};(40,40)*{}**\dir{-};
(0,30)*{};(10,30)*{}**\dir{-};
(30,30)*{};(40,30)*{}**\dir{-};
(0,40)*{};(0,30)*{}**\dir{-};
(10,30)*{};(10,40)*{}**\dir{-};
(30,30)*{};(30,40)*{}**\dir{-};
(40,30)*{};(40,40)*{}**\dir{-};
(10,30)*{};(20,25)*{}**\dir2{-};
(30,30)*{};(20,25)*{}**\dir2{-};
(10,10)*{};(20,25)*{}**\dir{-};
(30,10)*{};(20,25)*{}**\dir{-};
(10,10)*{};(0,10)*{}**\dir2{-};
(30,10)*{};(40,10)*{}**\dir2{-};
(-10,20)*{};(50,20)*{}**\crv{~**\dir2{-} (20,15)};
(20,25)*{};(20,50)*{}**\dir2{-};
(0,40)*{};(20,50)*{}**\dir2{-};
(40,40)*{};(20,50)*{}**\dir2{-};
(10,30)*{};(50,20)*{}**\dir2{-};
(30,30)*{};(-10,20)*{}**\dir2{-};
(0,40)*{};(-10,20)*{}**\dir2{-};
(40,40)*{};(50,20)*{}**\dir2{-};
(10,10)*{};(-10,20)*{}**\dir{-};
(30,10)*{};(50,20)*{}**\dir{-};
(30,10)*{};(-10,20)*{}**\dir{-};
(10,10)*{};(50,20)*{}**\dir{-};
(10,10)*{};(20,50)*{}**\dir{-};
(30,10)*{};(20,50)*{}**\dir{-};
(-13,23)*{};(53,23)*{}**\crv{~**\dir{--} (20,15)};
(-13,17)*{};(53,17)*{}**\crv{~**\dir{--} (20,14)};
(-13,23)*{};(-13,17)*{}**\dir{--};
(53,23)*{};(53,17)*{}**\dir{--};
(23,22)*{};(23,53)*{}**\dir{--};
(17,22)*{};(17,53)*{}**\dir{--};
(23,22)*{};(17,22)*{}**\dir{--};
(23,53)*{};(17,53)*{}**\dir{--};
\endxy}

There are six disjoint vertices which are disjoint from the two $\I_2$ fibers. There is no extremal elliptic fibration whose dual graph of singular fibers contains two $\tilde{A}_1$ diagrams and six disjoint vertices.

\end{enumerate}

\item $\Gamma_1 = 2\tilde{A}_3 \oplus \tilde{A}_3 \oplus 2\tilde{A}_1 \oplus \tilde{A}_1$

Fiber-bisection configuration:

\vspace{1mm}
\centerline{
\xy
@={(0,40),(10,40),(30,40),(40,40),(0,30),(10,30),(30,30),(40,30),(20,20),(0,10),(10,10),(30,10),(40,10)}@@{*{\bullet}};
(0,40)*{};(10,40)*{}**\dir{-};
(30,40)*{};(40,40)*{}**\dir{-};
(0,30)*{};(10,30)*{}**\dir{-};
(30,30)*{};(40,30)*{}**\dir{-};
(0,40)*{};(0,30)*{}**\dir{-};
(10,30)*{};(10,40)*{}**\dir{-};
(30,30)*{};(30,40)*{}**\dir{-};
(40,30)*{};(40,40)*{}**\dir{-};
(10,30)*{};(20,20)*{}**\dir{-};
(30,30)*{};(20,20)*{}**\dir2{-};
(10,10)*{};(20,20)*{}**\dir{-};
(30,10)*{};(20,20)*{}**\dir{-};
(40,10)*{};(20,20)*{}**\dir{-};
(10,10)*{};(0,10)*{}**\dir2{-};
(30,10)*{};(40,10)*{}**\dir2{-};
\endxy
}

Adding a bisection corresponding to a $2$-torsion section via $jac_2$, we find another special fibration with two double $\I_2$ fibers and special bisection $N$.

\vspace{2mm}
\centerline{
\xy
@={(0,40),(10,40),(30,40),(40,40),(0,30),(10,30),(30,30),(40,30),(20,20),(0,10),(10,10),(30,10),(40,10),(20,40)}@@{*{\bullet}};
(0,40)*{};(10,40)*{}**\dir{-};
(30,40)*{};(40,40)*{}**\dir{-};
(0,30)*{};(10,30)*{}**\dir{-};
(30,30)*{};(40,30)*{}**\dir{-};
(0,40)*{};(0,30)*{}**\dir{-};
(10,30)*{};(10,40)*{}**\dir{-};
(30,30)*{};(30,40)*{}**\dir{-};
(40,30)*{};(40,40)*{}**\dir{-};
(10,30)*{};(20,20)*{}**\dir{-};
(30,30)*{};(20,20)*{}**\dir2{-};
(10,10)*{};(20,20)*{}**\dir{-};
(30,10)*{};(20,20)*{}**\dir{-};
(40,10)*{};(20,20)*{}**\dir{-};
(10,10)*{};(0,10)*{}**\dir2{-};
(30,10)*{};(40,10)*{}**\dir2{-};
(10,30)*{};(20,40)*{}**\dir{-};
(40,40)*{};(20,40)*{}**\crv{~**\dir2{-} (30,50)};
(43,43)*{};(17,43)*{}**\crv{~**\dir{--} (30,50)};
(43,35)*{};(17,35)*{}**\crv{~**\dir{--} (30,50)};
(17,35)*{};(17,43)*{}**\dir{--};
(43,35)*{};(43,43)*{}**\dir{--};
(10,10)*{};(20,40)*{}**\dir{-};
(40,10)*{};(20,40)*{}**\dir{-};
(30,10)*{};(20,40)*{}**\dir{-};
(16,20)*{};(20,16)*{}**\dir{--};
(30,34)*{};(34,30)*{}**\dir{--};
(16,20)*{};(30,34)*{}**\dir{--};
(20,16)*{};(34,30)*{}**\dir{--};
(43,30)*{N};
\endxy
}

 Since we have already treated all cases with two double $\I_2$ fibers, we are done with this case.
 
\item $\Gamma_1 = \tilde{A}_3 \oplus \tilde{A}_3 \oplus 2\tilde{A}_1 \oplus \tilde{A}_1$

\begin{enumerate}[label = \alph*)]

\item Fiber-bisection configuration:

\centerline{
\xy
@={(0,40),(10,40),(30,40),(40,40),(0,30),(10,30),(30,30),(40,30),(20,20),(0,10),(10,10),(30,10),(40,10)}@@{*{\bullet}};
(0,40)*{};(10,40)*{}**\dir{-};
(30,40)*{};(40,40)*{}**\dir{-};
(0,30)*{};(10,30)*{}**\dir{-};
(30,30)*{};(40,30)*{}**\dir{-};
(0,40)*{};(0,30)*{}**\dir{-};
(10,30)*{};(10,40)*{}**\dir{-};
(30,30)*{};(30,40)*{}**\dir{-};
(40,30)*{};(40,40)*{}**\dir{-};
(0,30)*{};(20,20)*{}**\dir{-};
(30,30)*{};(20,20)*{}**\dir2{-};
(10,10)*{};(20,20)*{}**\dir{-};
(30,10)*{};(20,20)*{}**\dir{-};
(40,10)*{};(20,20)*{}**\dir{-};
(10,40)*{};(20,20)*{}**\dir{-};
(10,10)*{};(0,10)*{}**\dir2{-};
(30,10)*{};(40,10)*{}**\dir2{-};
\endxy}

After adding a bisection corresponding to a different $2$-torsion section via $jac_2$, we obtain another special elliptic fibration with two double fibers of type $\I_2$ and a special bisection $N$ as follows:

\centerline{
\xy
@={(0,40),(10,40),(30,40),(40,40),(20,50),(0,30),(10,30),(30,30),(40,30),(20,20),(0,10),(10,10),(30,10),(40,10)}@@{*{\bullet}};
(0,40)*{};(10,40)*{}**\dir{-};
(30,40)*{};(40,40)*{}**\dir{-};
(0,30)*{};(10,30)*{}**\dir{-};
(30,30)*{};(40,30)*{}**\dir{-};
(0,40)*{};(0,30)*{}**\dir{-};
(10,30)*{};(10,40)*{}**\dir{-};
(30,30)*{};(30,40)*{}**\dir{-};
(40,30)*{};(40,40)*{}**\dir{-};
(0,30)*{};(20,20)*{}**\dir{-};
(30,30)*{};(20,20)*{}**\dir2{-};
(10,10)*{};(20,20)*{}**\dir{-};
(30,10)*{};(20,20)*{}**\dir{-};
(40,10)*{};(20,20)*{}**\dir{-};
(10,40)*{};(20,20)*{}**\dir{-};
(10,10)*{};(0,10)*{}**\dir2{-};
(30,10)*{};(40,10)*{}**\dir2{-};
(20,50)*{};(10,40)*{}**\dir{-};
(20,50)*{};(0,30)*{}**\crv{(-10,50)};
(20,50)*{};(40,40)*{}**\dir2{-};
(20,50)*{};(40,10)*{}**\crv{(30,20)};
(20,50)*{};(30,10)*{}**\dir{-};
(20,50)*{};(10,10)*{}**\dir{-};
(19,53)*{};(41,43)*{}**\dir{--};
(41,37)*{};(41,43)*{}**\dir{--};
(19,47)*{};(41,37)*{}**\dir{--};
(19,53)*{};(19,47)*{}**\dir{--};
(31,33)*{};(19,22)*{}**\dir{--};
(19,17)*{};(19,22)*{}**\dir{--};
(31,33)*{};(31,28)*{}**\dir{--};
(31,28)*{};(19,17)*{}**\dir{--};
(43,30)*{N};
\endxy}

Since we have treated all fibrations with two double $\I_2$ fibers, we are done.

\item  The other fiber-bisection configuration where the bisection meets both components of the simple $\I_2$ fiber is treated similarly to case a).

\item Fiber-bisection configuration:

\vspace{1mm}
\centerline{
\xy
@={(0,40),(10,40),(30,40),(40,40),(0,30),(10,30),(30,30),(40,30),(20,20),(0,10),(10,10),(30,10),(40,10)}@@{*{\bullet}};
(0,40)*{};(10,40)*{}**\dir{-};
(30,40)*{};(40,40)*{}**\dir{-};
(0,30)*{};(10,30)*{}**\dir{-};
(30,30)*{};(40,30)*{}**\dir{-};
(0,40)*{};(0,30)*{}**\dir{-};
(10,30)*{};(10,40)*{}**\dir{-};
(30,30)*{};(30,40)*{}**\dir{-};
(40,30)*{};(40,40)*{}**\dir{-};
(10,30)*{};(20,20)*{}**\dir{-};
(30,30)*{};(20,20)*{}**\dir2{-};
(10,10)*{};(20,20)*{}**\dir{-};
(30,10)*{};(20,20)*{}**\dir2{-};
(10,40)*{};(20,20)*{}**\dir{-};
(10,10)*{};(0,10)*{}**\dir2{-};
(30,10)*{};(40,10)*{}**\dir2{-};
\endxy}

We add another bisection arising via $jac_2$ and find a special elliptic fibration with double fibers of type $\I_3$ and $\I_2$ and bisection $N$ as follows:

\vspace{4mm}
\centerline{
\xy
@={(0,40),(10,40),(30,40),(40,40),(0,30),(10,30),(30,30),(40,30),(20,20),(0,10),(10,10),(30,10),(40,10),(-10,20)}@@{*{\bullet}};
(0,40)*{};(10,40)*{}**\dir{-};
(30,40)*{};(40,40)*{}**\dir{-};
(0,30)*{};(10,30)*{}**\dir{-};
(30,30)*{};(40,30)*{}**\dir{-};
(0,40)*{};(0,30)*{}**\dir{-};
(10,30)*{};(10,40)*{}**\dir{-};
(30,30)*{};(30,40)*{}**\dir{-};
(40,30)*{};(40,40)*{}**\dir{-};
(10,30)*{};(20,20)*{}**\dir{-};
(30,30)*{};(20,20)*{}**\dir2{-};
(10,10)*{};(20,20)*{}**\dir{-};
(30,10)*{};(20,20)*{}**\dir2{-};
(10,40)*{};(20,20)*{}**\dir{-};
(10,10)*{};(0,10)*{}**\dir2{-};
(30,10)*{};(40,10)*{}**\dir2{-};
(-10,20)*{};(0,30)*{}**\dir{-};
(-10,20)*{};(0,40)*{}**\dir{-};
(-10,20)*{};(10,10)*{}**\dir{-};
(-10,20)*{};(40,10)*{}**\dir2{-};
(-10,20)*{};(30,30)*{}**\dir2{-};
(-12,18)*{};(-12,42)*{}**\dir{--};
(2,18)*{};(2,42)*{}**\dir{--};
(2,18)*{};(-12,18)*{}**\dir{--};
(-12,42)*{};(2,42)*{}**\dir{--};
(18,8)*{};(18,22)*{}**\dir{--};
(32,8)*{};(32,22)*{}**\dir{--};
(18,8)*{};(32,8)*{}**\dir{--};
(18,22)*{};(32,22)*{}**\dir{--};
(13,41)*{N};
\endxy}

The only extremal fibration with these fibers is the one with singular fibers $(\I_6,\I_3,\I_2,\I_1)$. But the $\I_3$ and $\I_2$ fibers cannot both be double by Lemma \ref{noHesse}. Therefore, this fiber-bisection configuration does not occur.

\item Fiber-bisection configuration:

\vspace{1mm}
\centerline{
\xy
@={(0,40),(10,40),(30,40),(40,40),(0,30),(10,30),(30,30),(40,30),(20,20),(0,10),(10,10),(30,10),(40,10)}@@{*{\bullet}};
(0,40)*{};(10,40)*{}**\dir{-};
(30,40)*{};(40,40)*{}**\dir{-};
(0,30)*{};(10,30)*{}**\dir{-};
(30,30)*{};(40,30)*{}**\dir{-};
(0,40)*{};(0,30)*{}**\dir{-};
(10,30)*{};(10,40)*{}**\dir{-};
(30,30)*{};(30,40)*{}**\dir{-};
(40,30)*{};(40,40)*{}**\dir{-};
(10,30)*{};(20,20)*{}**\dir2{-};
(30,30)*{};(20,20)*{}**\dir2{-};
(10,10)*{};(20,20)*{}**\dir{-};
(30,10)*{};(20,20)*{}**\dir2{-};
(10,10)*{};(0,10)*{}**\dir2{-};
(30,10)*{};(40,10)*{}**\dir2{-};
\endxy}

There is another special elliptic fibration with a double singular fiber of type $\I_2$ and special bisection $N$ as in the following figure:

\vspace{1mm}
\centerline{
\xy
@={(0,40),(10,40),(30,40),(40,40),(0,30),(10,30),(30,30),(40,30),(20,20),(0,10),(10,10),(30,10),(40,10)}@@{*{\bullet}};
(0,40)*{};(10,40)*{}**\dir{-};
(30,40)*{};(40,40)*{}**\dir{-};
(0,30)*{};(10,30)*{}**\dir{-};
(30,30)*{};(40,30)*{}**\dir{-};
(0,40)*{};(0,30)*{}**\dir{-};
(10,30)*{};(10,40)*{}**\dir{-};
(30,30)*{};(30,40)*{}**\dir{-};
(40,30)*{};(40,40)*{}**\dir{-};
(10,30)*{};(20,20)*{}**\dir2{-};
(30,30)*{};(20,20)*{}**\dir2{-};
(10,10)*{};(20,20)*{}**\dir{-};
(30,10)*{};(20,20)*{}**\dir2{-};
(10,10)*{};(0,10)*{}**\dir2{-};
(30,10)*{};(40,10)*{}**\dir2{-};
(9,33)*{};(21,22)*{}**\dir{--};
(9,28)*{};(21,17)*{}**\dir{--};
(9,33)*{};(9,28)*{}**\dir{--};
(21,17)*{};(21,22)*{}**\dir{--};
(-3,30)*{N};
\endxy}

There is an $A_3$ diagram and three disjoint vertices which are disjoint from the $\I_2$ fiber. The extremal fibrations whose dual graphs of singular fibers satisfy these conditions are the ones with singular fibers $(\I_2^*,\I_2,\I_2)$ and $(\I_4,\I_4,\I_2,\I_2)$. Since we have already treated the first fibration, we can assume that the second one occurs. But the bisection $N$ and the fibers form a fiber-bisection configuration which we have already treated, hence this case is settled.

\end{enumerate}

\item $\Gamma_1 = 2\tilde{A}_5 \oplus \tilde{A}_2 \oplus 2\tilde{A}_1$

Fiber-bisection configuration:

\vspace{1mm}
\centerline{
\xy
@={(0,0),(0,10),(0,20),(10,0),(10,10),(10,20),(20,0),(30,0),(40,0),(20,10),(20,20),(30,20)}@@{*{\bullet}};
(0,0)*{};(0,20)*{}**\dir{-};
(10,0)*{};(10,20)*{}**\dir{-};
(0,0)*{};(10,0)*{}**\dir{-};
(0,20)*{};(10,20)*{}**\dir{-};
(0,0)*{};(30,0)*{}**\dir{-};
(20,10)*{};(20,20)*{}**\dir{-};
(20,0)*{};(20,10)*{}**\dir2{-};
(20,20)*{};(30,20)*{}**\dir{-};
(30,20)*{};(20,10)*{}**\dir{-};
(30,0)*{};(40,0)*{}**\dir2{-};
\endxy
}
\vspace{1mm}

This is the critical subgraph for type $\V$.

\item $\Gamma_1 = \tilde{A}_5 \oplus \tilde{A}_2 \oplus 2\tilde{A}_1$
 \begin{enumerate}[label = \alph*)]

\item Fiber-bisection configuration:

\centerline{\xy
@={(0,0),(0,10),(0,20),(10,0),(10,10),(10,20),(20,0),(30,0),(40,0),(20,10),(20,20),(30,20)}@@{*{\bullet}};
(0,0)*{};(0,20)*{}**\dir{-};
(10,0)*{};(10,20)*{}**\dir{-};
(0,0)*{};(10,0)*{}**\dir{-};
(0,20)*{};(10,20)*{}**\dir{-};
(0,0)*{};(10,0)*{}**\dir{-};
(20,0)*{};(10,0)*{}**\dir{-};
(30,0)*{};(20,0)*{}**\dir{-};
(20,10)*{};(20,20)*{}**\dir{-};
(20,0)*{};(20,10)*{}**\dir2{-};
(20,20)*{};(30,20)*{}**\dir{-};
(30,20)*{};(20,10)*{}**\dir{-};
(30,0)*{};(40,0)*{}**\dir2{-};
(20,0)*{};(0,20)*{}**\crv{(15,30)};
\endxy
}
\vspace{1mm}

There is another special elliptic fibration with a double fiber of type $\I_5$ and bisection $N$. We leave it to the reader to check that one obtains the critical subgraph for type $\VI$ from a fibration with singular fibers $\I_5,\I_5$ where one of the $\I_5$ fibers is double.

\vspace{1mm}
\centerline{
\xy
@={(0,0),(0,10),(0,20),(10,0),(10,10),(10,20),(20,0),(30,0),(40,0),(20,10),(20,20),(30,20)}@@{*{\bullet}};
(0,0)*{};(0,20)*{}**\dir{-};
(10,0)*{};(10,20)*{}**\dir{-};
(0,0)*{};(10,0)*{}**\dir{-};
(0,20)*{};(10,20)*{}**\dir{-};
(0,0)*{};(10,0)*{}**\dir{-};
(20,0)*{};(10,0)*{}**\dir{-};
(30,0)*{};(20,0)*{}**\dir{-};
(20,10)*{};(20,20)*{}**\dir{-};
(20,0)*{};(20,10)*{}**\dir2{-};
(20,20)*{};(30,20)*{}**\dir{-};
(30,20)*{};(20,10)*{}**\dir{-};
(30,0)*{};(40,0)*{}**\dir2{-};
(20,0)*{};(0,20)*{}**\crv{(15,30)};
(-3,17)*{};(7,17)*{}**\dir{--};
(7,-3)*{};(7,17)*{}**\dir{--};
(7,-3)*{};(23,-3)*{}**\dir{--};
(13,25)*{};(23,-3)*{}**\dir{--};
(13,25)*{};(-3,25)*{}**\dir{--};
(-3,17)*{};(-3,25)*{}**\dir{--};
(-3,0)*{N};
\endxy}
\vspace{1mm}

\item Fiber-bisection configuration:

\vspace{1mm}
\centerline{
\xy
@={(0,0),(0,10),(0,20),(10,0),(10,10),(10,20),(20,0),(30,0),(40,0),(20,10),(20,20),(30,20)}@@{*{\bullet}};
(0,0)*{};(0,20)*{}**\dir{-};
(10,0)*{};(10,20)*{}**\dir{-};
(0,0)*{};(10,0)*{}**\dir{-};
(0,20)*{};(10,20)*{}**\dir{-};
(0,0)*{};(10,0)*{}**\dir{-};
(20,0)*{};(10,0)*{}**\dir{-};
(20,0)*{};(10,10)*{}**\dir{-};
(30,0)*{};(20,0)*{}**\dir{-};
(20,10)*{};(20,20)*{}**\dir{-};
(20,0)*{};(20,10)*{}**\dir2{-};
(20,20)*{};(30,20)*{}**\dir{-};
(30,20)*{};(20,10)*{}**\dir{-};
(30,0)*{};(40,0)*{}**\dir2{-};
\endxy}
\vspace{1mm}

Adding another special bisection corresponding to the $2$-torsion section via $jac_2$, we obtain another special elliptic fibration with two double singular fibers of type $\I_3$ and bisection $N$ as follows:

\centerline{
\xy
@={(0,0),(0,10),(0,20),(10,0),(10,10),(10,20),(20,0),(30,0),(40,0),(20,10),(20,20),(30,20),(-10,20)}@@{*{\bullet}};
(7,17)*{N};
(0,0)*{};(0,20)*{}**\dir{-};
(10,0)*{};(10,20)*{}**\dir{-};
(0,0)*{};(10,0)*{}**\dir{-};
(0,20)*{};(10,20)*{}**\dir{-};
(0,0)*{};(10,0)*{}**\dir{-};
(20,0)*{};(10,0)*{}**\dir{-};
(20,0)*{};(10,10)*{}**\dir{-};
(30,0)*{};(20,0)*{}**\dir{-};
(20,10)*{};(20,20)*{}**\dir{-};
(20,0)*{};(20,10)*{}**\dir2{-};
(20,20)*{};(30,20)*{}**\dir{-};
(30,20)*{};(20,10)*{}**\dir{-};
(30,0)*{};(40,0)*{}**\dir2{-};
(-10,20)*{};(0,10)*{}**\dir{-};
(-10,20)*{};(0,20)*{}**\dir{-};
(-10,20)*{};(20,10)*{}**\crv{~**\dir2{-} (10,35)};
(-10,20)*{};(30,0)*{}**\crv{(-10,-20)};
(7,-3)*{};(23,-3)*{}**\dir{--};
(7,-3)*{};(7,13)*{}**\dir{--};
(23,3)*{};(23,-3)*{}**\dir{--};
(13,13)*{};(7,13)*{}**\dir{--};
(13,13)*{};(23,3)*{}**\dir{--};
(-13,23)*{};(-13,17)*{}**\dir{--};
(-3,7)*{};(3,7)*{}**\dir{--};
(-3,7)*{};(-13,17)*{}**\dir{--};
(3,23)*{};(3,7)*{}**\dir{--};
(3,23)*{};(-13,23)*{}**\dir{--};
\endxy}
\vspace{5mm}

The only extremal and rational elliptic fibration with two fibers of type $\I_3$ is the fibration with fibers $(\I_3,\I_3,\I_3,\I_3)$. By Lemma \ref{noHesse}, there is no such fibration with two double $\I_3$ fibers.

\item Fiber-bisection configuration:

\vspace{1mm}
\centerline{
\xy
@={(0,0),(0,10),(0,20),(10,0),(10,10),(10,20),(20,0),(30,0),(40,0),(20,10),(20,20),(30,20)}@@{*{\bullet}};
(0,0)*{};(0,20)*{}**\dir{-};
(10,0)*{};(10,20)*{}**\dir{-};
(0,0)*{};(10,0)*{}**\dir{-};
(0,20)*{};(10,20)*{}**\dir{-};
(0,0)*{};(10,0)*{}**\dir{-};
(20,0)*{};(10,0)*{}**\dir2{-};
(20,0)*{};(30,20)*{}**\dir{-};
(30,0)*{};(20,0)*{}**\dir{-};
(20,10)*{};(20,20)*{}**\dir{-};
(20,0)*{};(20,10)*{}**\dir{-};
(20,20)*{};(30,20)*{}**\dir{-};
(30,20)*{};(20,10)*{}**\dir{-};
(30,0)*{};(40,0)*{}**\dir2{-};
\endxy}
\vspace{1mm}

Adding another special bisection corresponding to a $6$-torsion section via $jac_2$, we obtain another special fibration with a double fiber of type $\I_2$ and a special bisection $N$ as follows:

\vspace{-2mm}
\centerline{
\xy
@={(0,0),(0,10),(0,20),(10,0),(10,10),(10,20),(20,0),(30,0),(40,0),(20,10),(20,20),(30,20),(15,20)}@@{*{\bullet}};
(7,10)*{N};
(0,0)*{};(0,20)*{}**\dir{-};
(10,0)*{};(10,20)*{}**\dir{-};
(0,0)*{};(10,0)*{}**\dir{-};
(0,20)*{};(10,20)*{}**\dir{-};
(0,0)*{};(10,0)*{}**\dir{-};
(20,0)*{};(10,0)*{}**\dir2{-};
(30,0)*{};(20,0)*{}**\dir{-};
(20,10)*{};(20,20)*{}**\dir{-};
(20,0)*{};(20,10)*{}**\dir{-};
(20,0)*{};(30,20)*{}**\dir{-};
(20,20)*{};(30,20)*{}**\dir{-};
(30,20)*{};(20,10)*{}**\dir{-};
(30,0)*{};(40,0)*{}**\dir2{-};
(10,10)*{};(15,20)*{}**\dir2{-};
(20,10)*{};(15,20)*{}**\dir{-};
(20,20)*{};(15,20)*{}**\dir{-};
(30,0)*{};(15,20)*{}**\crv{(45,35)};
(7,-3)*{};(23,-3)*{}**\dir{--};
(7,3)*{};(23,3)*{}**\dir{--};
(7,-3)*{};(7,3)*{}**\dir{--};
(23,-3)*{};(23,3)*{}**\dir{--};
\endxy}
\vspace{2mm}

There are diagrams of type $A_3$, $A_2$, and $A_1$ which are disjoint from the double $\I_2$ fiber. Therefore, the fibration cannot have an $\I_8$ fiber. Since we have treated all the other cases with a double $\I_2$ fiber, we can assume that the fibration has singular fibers of type $\I_6, \I_3$ (or $\IV$) and $\I_2$ such that the $\I_6$ fiber is simple. But then, the fibers together with the bisection $N$ form the admissible fiber-bisection configuration of case a) or b), since $N$ meets distinct components of the $\I_6$ fiber. Therefore, this case is settled.

\item Fiber-bisection configuration:

\vspace{1mm}
\centerline{
\xy
@={(0,0),(0,10),(0,20),(10,0),(10,10),(10,20),(20,0),(30,0),(40,0),(20,10),(20,20),(30,20)}@@{*{\bullet}};
(0,0)*{};(0,20)*{}**\dir{-};
(10,0)*{};(10,20)*{}**\dir{-};
(0,0)*{};(10,0)*{}**\dir{-};
(0,20)*{};(10,20)*{}**\dir{-};
(0,0)*{};(10,0)*{}**\dir{-};
(20,0)*{};(10,0)*{}**\dir2{-};
(30,0)*{};(20,0)*{}**\dir{-};
(20,10)*{};(20,20)*{}**\dir{-};
(20,0)*{};(20,10)*{}**\dir2{-};
(20,20)*{};(30,20)*{}**\dir{-};
(30,20)*{};(20,10)*{}**\dir{-};
(30,0)*{};(40,0)*{}**\dir2{-};
\endxy}
\vspace{1mm}

Here, we can use the same $(-2)$-curves as in the previous case and the same argument right away without adding additional bisections.
\end{enumerate}

\item $\Gamma_1 = \tilde{A}_7 \oplus 2\tilde{A}_1$

\begin{enumerate}[label = \alph*)]

\item
Fiber-bisection configuration:

\vspace{1mm}
\centerline{
\xy
@={(0,0),(0,10),(0,20),(0,30),(10,0),(10,10),(10,20),(10,30),(20,0),(30,0),(40,0)}@@{*{\bullet}};
(0,0)*{};(0,30)*{}**\dir{-};
(10,0)*{};(10,30)*{}**\dir{-};
(0,0)*{};(10,0)*{}**\dir{-};
(10,30)*{};(0,30)*{}**\dir{-};
(10,0)*{};(20,0)*{}**\dir{-};
(10,10)*{};(20,0)*{}**\dir{-};
(30,0)*{};(20,0)*{}**\dir{-};
(30,0)*{};(40,0)*{}**\dir2{-};
\endxy}
\vspace{1mm}

This is the critical subgraph for type $\VII$.

\item
Fiber-bisection configuration:

\centerline{
\xy
@={(0,0),(0,10),(0,20),(0,30),(10,0),(10,10),(10,20),(10,30),(20,0),(30,0),(40,0)}@@{*{\bullet}};
(0,0)*{};(0,30)*{}**\dir{-};
(10,0)*{};(10,30)*{}**\dir{-};
(0,0)*{};(10,0)*{}**\dir{-};
(10,30)*{};(0,30)*{}**\dir{-};
(10,0)*{};(20,0)*{}**\dir2{-};
(30,0)*{};(20,0)*{}**\dir{-};
(30,0)*{};(40,0)*{}**\dir2{-};
\endxy}
\vspace{0.5mm}

There is another special elliptic fibration with a double fiber of type $\I_2$ and special bisection $N$ as follows:

\vspace{0.5mm}
\centerline{
\xy
@={(0,0),(0,10),(0,20),(0,30),(10,0),(10,10),(10,20),(10,30),(20,0),(30,0),(40,0)}@@{*{\bullet}};
(7,10)*{N};
(0,0)*{};(0,30)*{}**\dir{-};
(10,0)*{};(10,30)*{}**\dir{-};
(0,0)*{};(10,0)*{}**\dir{-};
(10,30)*{};(0,30)*{}**\dir{-};
(10,0)*{};(20,0)*{}**\dir2{-};
(30,0)*{};(20,0)*{}**\dir{-};
(30,0)*{};(40,0)*{}**\dir2{-};
(7,-3)*{};(7,3)*{}**\dir{--};
(23,-3)*{};(23,3)*{}**\dir{--};
(23,-3)*{};(7,-3)*{}**\dir{--};
(23,-3)*{};(7,-3)*{}**\dir{--};
(23,3)*{};(7,3)*{}**\dir{--};
\endxy}
\vspace{1mm}

This fiber-bisection configuration is not the same as the one we started with and since this is the last case, we have already treated this.

\end{enumerate}
\vspace{-5mm}

\end{itemize}
\end{proof}

\section{arithmetic of Enriques surfaces with finite automorphism group}\label{arithmetic}
In this section, we explain how to derive the results on the arithmetic of Enriques surfaces with finite automorphism group, which we mentioned in the introduction, from the equations we gave in the earlier chapters (see $\S 3,\hdots,\S 9$). In particular, we establish explicit models of Enriques surfaces of every type over the prime fields $\bbF_p$ and $\bbQ$.

%
%

\begin{lemma}\label{equationsfork3}
The following integral Weierstrass models of elliptic K3 surfaces admit a resolution of singularities over the ring $\rm{R}$, where $\rm{R}$ is as follows:

\begin{table}[!htbp]
\centering
\begin{tabular}{|>{\centering\arraybackslash}m{11.5cm}|>{\centering\arraybackslash}m{1.5cm}|>{\centering\arraybackslash}m{1cm}|}
\hline 
\rm{Equation} & $\rm{R}$  & \rm{Type} \\
\hline \hline
$y^2 + (s^2+s)xy = x^3 + (s^2+s)^3x$ &  $\bbZ[\frac{1}{257}]$  & \vspace{0.5mm} $\rm{I}$ \\ [0.5mm] \hline
$y^2 - (s^2+s)xy = x^3 - (s^2+s)^3x$ &  $\bbZ[\frac{1}{255}]$ & \vspace{0.5mm}$\rm{I}$\\ [0.5mm] \hline
$y^2 + (s^2+s)xy + (s^2+s)^2y = x^3 + (s^2+s)x^2$ &  $\bbZ[\frac{1}{65}]$ & \vspace{0.5mm}$\rm{II}$ \\ [0.5mm] \hline
$y^2 - (s^2+s)xy + (s^2+s)^2y = x^3 - (s^2+s)x^2$ &  $\bbZ[\frac{1}{63}]$ & \vspace{0.5mm}$\rm{II}$\\ [0.5mm] \hline
$y^2 +xy = x^3 + 4s^4x^2 + s^4x$ & $\bbZ[\frac{1}{2}]$ & \vspace{0.5mm}$\rm{III}$\\ [0.5mm] \hline
$y^2 = x^3 + 2(s^4+1)x^2 + (s^4-1)^2x$ & $\bbZ[\frac{1}{2}]$ & \vspace{0.5mm}$\rm{IV}$ \\ [0.5mm] \hline
$y^2  + (s^2+1)xy+(s^2+1)y = x^3 + (s^2+2)x^2 + (s^2+1)x$ &  $\bbZ[\frac{1}{6}]$ & \vspace{0.5mm}$\rm{V}$\\ [0.5mm] \hline
$y^2- 3(3s^2 + 3s +1)xy + (3s^2 + 3s +1)^2y = x^3$ & $\bbZ[\frac{1}{15}]$ & \vspace{0.5mm}$\rm{VI}$ \\ [0.5mm] \hline
$y^2 = x^3 -(s^4+s^2)x^2+(2s^6-3s^4+4s^2-2)x + (-s^6+2s^4-2s^2+1)$ &  $\bbZ[\frac{1}{10}]$ & \vspace{0.5mm}$\rm{VII}$ \\ [0.5mm] \hline
\end{tabular}
\end{table}

\end{lemma}
\begin{proof}
Let $f: \mathcal{X} \to \rm{Spec}(R)$ be one of the families defined by the above equations.
Since the non-smooth locus of $f$ is closed and $f$ is proper, the non-smooth locus of $f$ is proper. Hence, every singular point of the generic fiber $X_{\eta}$ of $f$ is the generic point of a subscheme $Z$ of $\mathcal{X}$ which is completely contained in the singular locus of $f$ and flat over $\rm{Spec}(R)$.
Since $Z$ is flat over $\rm{Spec}(R)$, a local computation shows that blowing up along $Z$ commutes with taking fibers of $f$. Moreover, we know that every fiber of $f$ has the same types of rational double points, hence we can repeat the above argument and deduce that the minimal resolution of singularities of the generic fiber extends uniquely to a minimal resolution of the whole family.
\end{proof}

\begin{remark}
The reason why we have to exclude some seemingly arbitrary characteristics is that the surface defined by the Weierstrass equation acquires additional singularities in these characteristics, because the degree $2$ morphism to a rational elliptic surface we used to find the equations branches over a multiplicative fiber. This happens for the first four equations and for the last two, where the double cover branches over a nodal fiber, producing an additional $A_1$ singularity in some fibers. This singularity cannot be resolved in families without a base change to an algebraic space (see \cite{Artin}).
\end{remark}

\begin{theorem}\label{integralmodels}
Let $K \in \{\I,\hdots,\VII\}$. There is a morphism $\varphi_K: \mathcal{X} \to \rm{Spec}(\bbZ[\frac{1}{P_K}])$ whose fibers are Enriques surfaces of type $K$ with full Picard rank, i.e. $\Pic(\mathcal{X}_{\bbF_p}) = \Pic(\mathcal{X}_{\bar{\bbF}_p})$. The numbers $P_K$ are given in Table \ref{integralmodelstable}.
\end{theorem}

\begin{proof}
By Lemma \ref{equationsfork3}, we have a family of K3 surfaces over $\bbZ[\frac{1}{P_K}]$. Now, observe that the Enriques involution is also defined over this ring. Hence, the only remaining claim is the one that the fibers of the family have full Picard rank. 

Let $X_p$ be the fiber over $p$ of one of the families of Enriques surfaces over $R$ and let $\tilde{X}_p$ be its canonical cover. By Vinberg's criterion (Proposition \ref{vinberg}), the geometric Picard group of $X_p$ is generated by $(-2)$-curves, hence it suffices to check that all these curves are defined over $\bbF_p$ (resp. over $\bbQ$ if $p = 0$). Then, one uses our explicit equations to check that the Galois action preserves the preimages of these curves in $\tilde{X}_p$ and therefore all $(-2)$-curves on $X_p$ are defined over $\bbF_p$ (resp. over $\bbQ$ if $p = 0$). Note that it suffices to check that the fiber components and special bisections of the fibration we used to construct the surfaces are fixed, since this will imply that the Galois action is trivial on the whole graph.
\end{proof}

\begin{remark}
In particular, note that there are Enriques surfaces of type $\VI$ and $\VII$ with full Picard rank over $\bbQ$, while this is not possible for their canonical cover due to a result of N. D. Elkies (see \cite{Schütt}).
\end{remark}

Moreover, Theorem \ref{integralmodels} proves the existence of a model for every type of Enriques surfaces with finite automorphism group together with its dual graph of $(-2)$-curves over the prime fields.

\begin{corollary}\label{primefields}
Suppose that there exists an Enriques surface of type $K \in \{\I,\hdots,\VII\}$ in characteristic $p$. Then, there exists an Enriques surface of type $K$ with Picard rank $10$ over $\bbF_p$ (resp. over $\bbQ$ if $p = 0$).
\end{corollary}

\begin{theorem}
Let $X$ be an Enriques surface of type $K \in \{\I,\hdots,\VII\}$ over a field $k$ such that $\Pic(X) = \Pic(X_{\bar{k}})$. 
\begin{itemize}
\item If $K \neq \III,\IV$, then $\Aut(X)$ is defined over $k$.
\item If $K = \III$, then $\Aut(X)$ is defined over $L \supseteq k$ with $[L:k] \leq 2$.
\item If $K = \IV$, then $\Aut(X)$ is defined over $L \supseteq k$ with $[L:k] \leq 16$. 
\end{itemize}
\end{theorem}

\begin{proof}
Let $X$ be an Enriques surface over $k$ such that $|\Aut(X_{\bar{k}})| < \infty$ and $\rm{rk}(\Pic(X)) = 10$. 
Since $\rm{rk}(\Pic(X)) = 10$, every elliptic fibration of $X$ is defined over $k$. Therefore, all Jacobian fibrations of elliptic fibrations of $X$ are defined over $k$. Now, if $X$ is of type $\I,\II,\V,\VI$ or $\VII$, the generic fiber of an elliptic fibration of $X$ whose Jacobian has non-trivial sections has $j$-invariant $\neq 0,1728$. Therefore, the Jacobian is unique up to quadratic twisting with elements in $\bar{k}$. We have shown in Propositions \ref{Aut1}, \ref{Aut2}, \ref{Aut5}, \ref{Aut6}, and \ref{Aut7} that $\Aut(X_{\bar{k}})$ is generated by the actions of $2$-torsion sections of the Jacobian fibrations of elliptic fibrations of $X$. Since quadratic twisting preserves $2$-torsion sections and all extremal and rational elliptic fibrations have a model over $k$ such that their $2$-torsion is already defined over $k$, all such sections, and hence $\Aut(X)$, are defined over $k$.

If $X$ is of type $\III$, we need to realize the additional automorphism of Remark \ref{extraaut3}. For this, a quadratic extension is sufficient.

If $X$ is of type $\IV$, we need the automorphism of Remark \ref{extrasection} and one non-$2$-torsion section (see Proposition \ref{Aut4}). As before, we need a field extension of degree at most $2$ per non-$2$-torsion section. To define the automorphism of Remark \ref{extrasection}, we need a field extension of degree at most eight, since we found a model of the corresponding fibration which acquires the required section after a quadratic extension and we need a quadratic extension to define $\iota$ (see Remark \ref{extrasection}).
\end{proof}

\begin{remark}
Over finite fields (and for our model), the proof shows that an extension of degree $4$ suffices to realize all automorphisms for type $\IV$. 
\end{remark}

\section{Semi-symplectic automorphisms}\label{semisection}
As an application of our explicit classification of Enriques surfaces with finite automorphism group, we determine the semi-symplectic automorphism groups of these surfaces.

\begin{definition}
Let $X$ be an Enriques surface. An automorphism of $X$ is called \emph{semi-symplectic} if it acts trivially on $\rm{H}^0(X,\omega_X^{\otimes 2})$. We denote the group of all semi-symplectic automorphisms of $X$ by $\Aut_{ss}(X)$.
\end{definition}

These automorphisms are studied in \cite{mukaiohashi}. There, the semi-symplectic automorphism groups of Enriques surfaces of type $\VI$ and $\VII$ have already been computed. See \cite{Ohashi} for a study of finite and non-semi-symplectic automorphisms.

\begin{theorem}\label{semisymplecticthm}
Let $X$ be an Enriques surface of type $K \in \{\I,\hdots,\VII\}$. Then, $\Aut_{ss}(X)$ is as given in the following table:
\begin{table}[!htb]
\centering
\begin{tabular}{|>{\centering\arraybackslash}m{2.5cm}|>{\centering\arraybackslash}m{6cm}|}
\hline
\rm{Type} & $\Aut_{ss}(X)$ \\
\hline \hline
\rm{I} &  \vspace{0.5mm} $D_4$ \\ [0.5mm] \hline
\rm{II} & \vspace{0.5mm} $\mathfrak{S}_4$ \\ [0.5mm] \hline
\rm{III} & \vspace{0.5mm} $(\bbZ/2\bbZ)^3 \rtimes D_4$ \\ [0.5mm] \hline
\rm{IV} &  \vspace{0.5mm} $(\bbZ/2\bbZ)^4 \rtimes (\bbZ/5\bbZ \rtimes \bbZ/2\bbZ)$ \\ [0.5mm] \hline
\rm{V} & \vspace{0.5mm} $\mathfrak{S}_4 \times \bbZ/2\bbZ$ \\ [0.5mm] \hline
\rm{VI} & \vspace{0.5mm} $\mathfrak{S}_5$ \\ [0.5mm] \hline
\rm{VII} & \vspace{0.5mm} $\mathfrak{S}_5$ \\ [0.5mm] \hline
\end{tabular}
\caption{Semi-symplectic automorphism groups}
\label{semisymplectic}
\end{table}
\end{theorem}

\vspace{-10mm}
\begin{proof}
Note that an automorphism induced by a section of the Jacobian of an elliptic fibration of $X$ is semi-symplectic, since it fixes the base of the fibration and acts as translation on the fibers. For all $K$, the group generated by such automorphisms is equal to the group given in Table \ref{semisymplectic}. If $K \neq \III,\IV$, these are all automorphisms, and if $K \in \{\III,\IV\}$, we have exhibited non-semi-symplectic automorphisms in Remarks \ref{extraaut3} and \ref{extrasection}. Since the groups in Table \ref{semisymplectic} have index $2$ in $\Aut(X)$ for $K \in \{\III,\IV\}$, this finishes the proof.
\end{proof}

\begin{remark}
The fact that surfaces of type $\III$ and $\IV$ admit non-semi-symplectic automorphisms is the reason why, in general, we need a field extension to realize all automorphisms of these surfaces. These non-semi-symplectic automorphisms act as $\sqrt{-1}$ on a non-zero global $2$-form of the K3 cover, hence it is necessary to adjoin at least $\sqrt{-1}$ to $k$ to realize all automorphisms of these surfaces. Since the K3 cover of Enriques surfaces of type $\III$ and $\IV$ is the Kummer surface associated to the self-product of an elliptic curve with $j$-invariant $1728$ \cite[p.193]{Kondo}, it is likely that this field extension always suffices.
\end{remark}
\label{Bibliography}

\bibliographystyle{alpha}  
\bibliography{Bibliography}

\end{document}